%% file: iterativecb.tex
\documentclass{article}
\usepackage{a4wide}
\input{inputs.tex}

\newcommand{\upd}[1]{{#1}}
\newcommand{\pcH}[1]{\ensuremath{\upd{\hat H({#1})}}}
\newcommand{\pcG}[1]{\ensuremath{\upd{\hat G({#1})}}}

\begin{document}
\title{A Preconditioned Iterative Interior Point Approach\\ to the
  Conic Bundle Subproblem}
\author{Christoph Helmberg\thanks{Department of Mathematics, 
Chemnitz University of Technology, D-09107 Chemnitz, Germany,
helmberg@math.tu-chemnitz.de,
www.mathematik.tu-chemnitz.de/$\sim$helmberg, ORCID 0000-0002-5288-8000}}
\date{\today}
\maketitle

\begin{abstract}
  The conic bundle implementation of the spectral bundle method for
  large scale semidefinite programming solves in each iteration a
  semidefinite quadratic subproblem by an interior point approach.
  For larger cutting model sizes the limiting operation is collecting
  and factorizing a Schur complement of the primal-dual KKT system. We
  explore possibilities to improve on this by an iterative approach
  that exploits structural low rank properties. Two preconditioning
  approaches are proposed and analyzed. Both might be of interest for
  rank structured positive definite systems in general. The first
  employs projections onto random subspaces, the second projects onto
  a subspace that is chosen deterministically based on structural
  interior point properties. For both approaches theoretic bounds are
  derived for the associated condition number. In the instances tested
  the deterministic preconditioner provides surprisingly efficient
  control on the actual condition number. The results suggest that for
  large scale instances the iterative solver is usually the better
  choice if precision requirements are moderate or if the size of
  the Schur complemented system clearly exceeds the active dimension
  within the subspace giving rise to the cutting model of the bundle
  method.

  \medskip
  
\end{abstract}

\section{Introduction}

In semidefinite programming the ever increasing number of applications
\cite{AnjosLasserre2012,WolkowiczSaigalVandenberghe2000,BoydVandenberghe2007,HenrionKordaLasserre2020}
leads to a corresponding increase in demand for reliable and efficient
solvers \upd{for linear programs over symmetric cones}. In general,
interior point methods are the method of choice.  Yet, if the order of
\upd{some semidefinite matrix variables} gets large and the affine matrix functions involved do
not allow to use decomposition or factorization approaches such as
proposed in
\cite{KimKojimaMevissenYamashita2011,BensonYeZhang2000,BurerMonteiro2003},
general interior point methods are no longer applicable. The limiting
factors are typically memory requirements and computation times
connected with forming and factorizing a Schur complemented system
matrix of the interior point KKT system. \upd{Large scale second order cone
variables do not cause such problems, this is indeed specific to 
semidefinite settings. In such cases,} the spectral
bundle method of \cite{HelmbergRendl2000} offers a viable alternative.

The spectral bundle method reformulates the semidefiniteness condition
via a penalty term on the extremal eigenvalues of a corresponding
affine matrix function and assumes these eigenvalues to be efficiently
computable by iterative methods. In each step it selects a subspace
close to the current active eigenspace. Then the next candidate point
is determined as the proximal point with respect to the extremal
eigenvalues of the affine matrix function projected onto this
subspace. The proximal point is the optimal solution to a quadratic
semidefinite subproblem whose matrix variable is of the order of the
dimension of the approximating subspace. If the subspace is kept
small, this allows to find approximately optimal solutions in
reasonable time. In order to reach solutions of higher precision it
seems unavoidable to go beyond the full active eigenspace
\cite{HelmbergOvertonRendl2014,DingGrimmer2020}. In the current
implementation within the callable library ConicBundle
\cite{ConicBundle2021}\upd{, which also supports second order cone and
  nonnegative variables, }the quadratic subproblem is solved by an
interior point approach. Again for each of its KKT systems the
limiting work consists in collecting and factorizing a Schur
complement matrix whose order is typically the square of the dimension
of the active eigenspace. The main question addressed here is whether
it is possible to stretch these limits by developing a suitably
preconditioned iterative solver that allows to circumvent the
collection and factorization of this Schur complement.  The focus is
thus not on the spectral bundle method itself but on solving KKT
systems of related quadratic semidefinite \upd{and more generally
  quadratic conic} programs by iterative methods. \upd{While the
  motivating and most general semidefinite case dominates in this
  work, natural extensions to second order and nonnegative cones will
  also be mentioned, because future applications may well expect and
  require support for arbitrary combinations of conic variables.}
Even though the methodology will be developed and discussed for low
rank properties that arise in the ConicBundle setting, some of the
considerations and ideas should be transferable to general conic
quadratic optimization problems whose quadratic term consists of a
positive diagonal plus a low rank Gram matrix or maybe even to general
positive definite systems of this form.

Here is an outline of the paper and its main
contributions. Section \ref{S:ConicBundleKKT} provides
the necessary background on the bundle philosophy underlying
ConicBundle and derives the KKT system of the bundle subproblem. The
core of the work is presented in Section \ref{S:lowrankprecond} on low
rank preconditioning for a Gram-matrix plus positive diagonal. For
slightly greater generality, denote the cone of positive
(semi)definite matrices of order $m$ by $\Sym^m_{++}\ (\Sym^m_+)$ and
let the system matrix be given in the form
\begin{displaymath}
  H=D+VV^\top\qquad\text{ with }D\in\Sym^m_{++}, V\in\R^{m\times n}, 
\end{displaymath}
where it is tacitly assumed that $D^{-1}$ times vector and $V$ times
vector are efficiently computable. Typically $n\le m$ but whenever $n$
is sizable one would like to approximate $V$ by a matrix
$\hat V\in\R^{m\times k}$ with significantly smaller $k<n$ to obtain a
preconditioner $\hat H=D+ \hat V\hat V^\top$ \upd{whose inverse, by a low
rank update, reads} $\hat H^{-1}=D^{-1}-D^{-1}\hat V(I_k+\hat V^\top D^{-1}\hat
V)^{-1}\hat V^\top D^{-1}$. Comparing this to the inverse of $H$, the
goal is to capture the large eigenvalues of $V^\top D^{-1}V$, more
precisely the directions belonging to large singular values of
$D^{-\frac12}V$. By the singular value decomposition (SVD) this can be
achieved by the projection onto a subspace, say $D^{-\frac12}VP$ for a
suitably chosen orthogonal $P\in\R^{n\times k}$. Because the full SVD
is computationally too expensive, two other approaches will be
developed and analyzed here. In the first, Section
\ref{S:randprecond}, the orthogonal $P$ is generated by a Gaussian
matrix $\Omega\in\R^{n\times k}$. In the second, Section
\ref{S:detprecond}, some knowledge about the interior point method
leading to $V$ will be exploited in order to form $P$
deterministically.

The projection onto a random subspace may be motivated geometrically
by interpreting the Gram matrix $VV^\top$ as the inner products of the
row vectors of $V$. The result of Johnson-Lindenstrauss, cf.~\cite{Achlioptas2001,DasguptaGupta2002}, allows to approximate this
with low distortion by a projection onto a low dimensional
subspace. In matrix approximations this idea seems to have first
appeared in \cite{PapadimitriouRaghavanTamakiVempala2000}.  In
connection with preconditioning a recent probabilistic approach is
described in \cite{HighamMary2019} in the context of controlling the
error of a LU preconditioner. \cite{HalkoMartinssonTropp2011} gives an
excellent introduction to probabilistic algorithms for constructing
approximate matrix decompositions and provides useful
bounds. Building directly on their techniques we provide
deterministic and probabilistic bounds on the condition number of the
random subspace preconditioned system in theorems
\ref{th:deterministicbound} and \ref{th:probabilisticbound}. In
comparison to the moment analysis of the Ritz values of the
preconditioned matrix presented in Theorem~\ref{T:moments}, the bounds seem
to fall below expectation and are maybe still improvable.  Random
projections do not require any problem specific structural insights,
but it remains open how to choose the subspace dimension in order to
obtain an efficient preconditioner.

In contrast, identifying the correct subspace seems to work well for the
deterministic preconditioning routine. It exploits structural
properties of the KKT system's origin in interior point methods.
Within interior point methods iterative approaches have been
investigated in quite a number of works, in conjunction with
semidefinite optimization see \eg \cite{Toh2008,KocvaraStingl2007}.
These methods were
mostly designed for exploiting sparsity rather than low rank
structure.  During the last months of this work an approach closely
related to ours appeared in \cite{HabibiKavandKocvaraStingl2021}. It significantly
extends ideas of \cite{ZhangLavaei2017} for a deterministic
preconditioning variant. It assumes the rank of the optimal solution to
be known in advance and provides a detailed analysis for this case. Their
ideas and arguments heavily influenced the condition number analysis
of our approach presented in theorems \ref{T:deterministiccond} and
\ref{T:speccondbound}. In contrast to \cite{HabibiKavandKocvaraStingl2021}, our
algorithmic approach does not require any a priori knowledge on the
rank of the optimal solution. Rather, Theorem~\ref{T:speccondbound}
and Lemma~\ref{L:TTTPSC} motivate an
estimate on the singular value induced by certain directions
associated with active interior point variables, that seems to offer a good indicator for the relevance of the corresponding subspace.

In Section \ref{S:experiments} the performance of the preconditioning
approaches is illustrated relative to the direct solver on sequences
of KKT systems that arise in solving three large
scale instances within ConicBundle.  The deterministic approach turns out to
be surprisingly effective in identifying a suitable subspace.  It
provides good control on the condition number and reduces the number
of matrix vector multiplications significantly.  The selected instances are
also intended to demonstrate the differences in the potential of the
methods depending on the problem characteristics.  Roughly, the direct
solver is mainly attractive if the model is tiny, if significant parts of the Schur
complement can be precomputed for all KKT systems of the same
subproblem or if precision requirements get exceedingly high with the entire
bundle model being strongly active. In general, however, the iterative
approach with deteriministic preconditioner can be expected to lead to
significant savings in computation time in large scale applications.
\upd{In order to demonstrate that this KKT systems based analysis suitably reflects
the performance of the solvers within the bundle method, the section
closes with reporting preliminary experiments on randomly generated Max-Cut
instances where ConicBundle is run for each solver separately with exactly the same
parameter settings that were developed and tuned for the direct solver.}
In Section \ref{S:concl} the paper ends with some concluding remarks.

\textbf{Notation.} For matrices or vectors $A,B\in\R^{m\times n}$ the
(trace) inner product is denoted by $\ip{A}{B}=\tr B^\top
A=\sum_{ij}A_{ij}B_{ij}$. $A\circ B=(A_{ij}B_{ij})$ denotes the
elementwise or Hadamard product. \upd{$A_{i,\bullet}$ refers to the
row-vector of the $i$-th row of $A$ and $A_{\bullet,j}$ to the
column-vector of the $j$-th column of $A$. For some ordered
index set $J\subseteq\{1,\dots,n\}$ the submatrix
$A_{\bullet,J}$  consists of the respective columns.} 
Consider symmetric matrices $A,B\in\Sym^n$ of order $n$. For representing
these as vectors, 
the operator $\svec
A=(A_{11},\sqrt{2}A_{21},\dots,\sqrt{2}A_{n1},A_{22
},\sqrt{2}A_{32},\dots,A_{nn})^\top$ stacks the columns of the lower
triangle with offdiagonal elements multiplied by $\sqrt{2}$ so that $\ip{A}{B}=\svec(A)^\top\svec(B)$. For matrices $F,G\in \R^{k\times n}$ the symmetric Kronecker product $\skron$ is defined by $(F\skron G)\svec(A)=\frac12\svec(FAG^\top+GAF^\top)$.  The Loewner partial
order $A\succeq B$ ($A\succ B$) refers to $A-B\in\Sym^n_{+}$
($A-B\in\Sym^n_{++})$ being positive semidefinite (positive
definite). The eigenvalues of $A$ are denoted by $\lmax(A)=\lambda_1(A)\ge \dots\ge \lambda_n(A)=\lmin(A)$. The norm $\|\cdot\|$ refers to the
Euclidean norm for vectors and to the spectral norm for matrices. 
$I_n$ ($I$) denotes the identity matrix of order $n$ (or of appropriate size), the canonical unit vectors $e_i$ refer to
the $i$-th column of $I$. Unless stated explicitly otherwise, $\one$
denotes the vector of all ones of appropriate size. 
$\E$ refers to the expected value of a random variable, $\Var$ to its variance and
 $\calN(\mu,\sigma^2)$ to the normal or Gaussian distribution with mean $\mu$ and standard deviation $\sigma$.

\section{The KKT system of the ConicBundle Subproblem}\label{S:ConicBundleKKT}

The general setting of bundle methods deals with minimizing a (typically closed) convex 
function $f\colon\R^m\to\clR:=\R\cup\{\infty\}$ over a closed convex
ground set $\gset\subseteq\dom f$ of simple structure like $\R^m$, a
box or a polyhedron,
\begin{displaymath}
\text{minimize}~f(y)~~\text{ subject to }~y\in\gset.
\end{displaymath}
Typically, $f$ is given by a first order oracle, \ie a routine that returns
for a given $\bar y\in\gset$ the function value $f(\bar y)$ and an
arbitrary subgradient $g\in\partial f(\bar y)$ from the
subdifferential of $f$ in $\bar y$. Value $f(\bar y)$ and subgradient
$g$ give rise to a
supporting hyperplane to the epigraph of $f$ in $(\bar y,f(\bar y))$. The
algorithm collects these affine minorants in the \emph{bundle} to form a
cutting model of $f$. It will be convenient to arrange the value at zero
and the gradient in a pair $\omega=(\gamma=f(\bar y)-\ip{g}{\bar
  y},g)$ and to denote, for $y\in\R^m$, the minorant's value in $y$ by
$\omega(y):=\gamma+\ip{g}{y}$.

Let $\mathcal{W}_f=\{\omega=(\gamma,g)\in\R^{1+m}\colon \gamma+\ip{g}{y}\le f(y),
y\in \R^m\}$ denote the set of all affine minorants of $f$. For closed
$f$ we have $f(y)=\sup_{\omega\in\mathcal{W}_f}\omega(y)$. Any
compact subset $W\subseteq\mathcal{W}_f$ gives rise to a minorizing
cutting model
of $f$,
\begin{displaymath}
  W(y):=\max_{\omega\in W} \omega(y)\le f(y),\quad y\in\R^m. 
\end{displaymath}
At the beginning of iteration $k=0,1,\dots$ the bundle method's state is described by a
current \emph{stability center} $\hat y_k\in \R^m$, a compact cutting
model $W_k\subseteq \mathcal{W}_f$, and a proximity term, here the square 
of a norm
$\|\cdot\|_{\upd{\mathfrak{H}}_k}^2:=\ip{\cdot}{\upd{\mathfrak{H}}_k\cdot}$
with positive definite $\upd{\mathfrak{H}}_k$ \upd{(this Fraktur 
  $H$ will form the core of the final system matrix $H$)}. The method
determines the next candidate $y_{k+1}\in\R^m$ as minimizer of the
\emph{augmented model} or \emph{bundle subproblem}
\begin{equation}\label{eq:candidate}
  y_{k+1}=\argmin_{y\in\gset} W_k(y)+\tfrac12\|y-\hat y_k\|_{\upd{\mathfrak{H}}_{k}}^2.
\end{equation}
Solving this bundle subproblem may be viewed as determining a
saddle point $(y_{k+1},\bar\omega_{k+1}=(\bar\gamma_{k+1},\bar g_{k+1}))\in \gset\times\conv W_k$,
which exists for any closed convex $C$ by \cite{Rockafellar70},
Theorems 37.3 and 37.6, due to the strong convexity in $y$ and the compactness of
$W_k$, 
\begin{eqnarray*}
\bar\omega_{k+1}(y_{k+1})+\tfrac12\|y_{k+1}-\hat
  y_k\|_{\upd{\mathfrak{H}}_{k}}^2&=&
  \inf_{y\in\gset}\sup_{\omega=(\gamma,g)\in
    W_k}\gamma+\ip{g}{y}+\tfrac12\|y-\hat
  y_k\|_{\upd{\mathfrak{H}}_{k}}^2\\
&=&\sup_{\omega\in\conv W_k}\inf_{y\in\gset}\omega(y)+\tfrac12\|y-\hat
  y_k\|_{\upd{\mathfrak{H}}_{k}}^2.
\end{eqnarray*}
Strong convexity in $y$ ensures uniqueness of $y_{k+1}$. First
order optimality with respect to $y$ implies
\begin{displaymath}
 \bar g_{k+1}+\upd{\mathfrak{H}}_k(y_{k+1}-\hat y_k)\in -N_C(y_{k+1}),
\end{displaymath}
where $N_C(y)$ denotes the normal cone to $C$ at $y\in C$.
In the unconstrained case of $C=\R^m$ the \emph{aggregate}
$\bar\omega_{k+1}$ is also unique. Whether unique or not, the \emph{aggregate}
will refer to the solution $\bar\omega_{k+1}\in\conv W_k$  produced by the algorithmic approach for
solving \eqref{eq:candidate}. The progress predicted by the model is 
$f(\hat y_k)-\bar\omega_{k+1}(y_{k+1})=f(\hat
y_k)-W_{k}(y_{k+1})$. Actual
progress will be compared to a threshold value which arises from
damping the progress predicted by the model by some  $\kappa
\in(0,1)$ 
\begin{displaymath}
  \vartheta_{k+1}=\kappa[f(\hat y_k)-\bar\omega_{k+1}(y_{k+1})].
\end{displaymath}
Next, $f$ is evaluated at $y_{k+1}$ by calling the oracle which returns $f(y_{k+1})$ and a
new minorant $\omega_{k+1}$ with
$\omega_{k+1}(y_{k+1})=f(y_{k+1})$. If progress in objective value
is sufficiently large in comparison to the
progress predicted by the model, \ie 
\begin{displaymath}
f(\hat
y_k)-f(y_{k+1})\ge \vartheta_{k+1},
\end{displaymath}
the method executes a \emph{descent step} which moves the center of
stability to the new point, $\hat y_{k+1}=y_{k+1}$. Otherwise, in a
\emph{null step}, the center remains unchanged,
$\hat y_{k+1}=\hat y_{k}$, but the new minorant $\omega_{k+1}$ is used
to improve the model. In fact, the requirement
$\{\bar\omega_{k+1},\omega_{k+1}\}\subseteq W_{k+1}$ ensures
convergence of the function values $f(\hat y_k)$ to $f^*=\inf_{y\in C} f(y)$ under mild technical conditions on $\upd{\mathfrak{H}}_{k+1}$. For these it suffices, \eg to
fix $0<\underline{\lambda}\le\bar\lambda$ and to choose
$\underline{\lambda}I\preceq \upd{\mathfrak{H}}_{k+1} \preceq \bar\lambda I$ following
a descent step and $\upd{\mathfrak{H}}_k\preceq \upd{\mathfrak{H}}_{k+1} \preceq \bar\lambda I$
following a null step, see \cite{BonnansGilbertLemarechalSagastizabal2006}. 

The decisive elements for an efficient implementation are the following:
\begin{itemize}
\item the choice of the cutting model $W_k$,
\item the choice of the proximal term, in our case of $\upd{\mathfrak{H}}_k$,
\item the solution method for the bundle subproblem
  \eqref{eq:candidate} with corresponding structural requirements on supported ground
  sets $C$. 
\end{itemize}
and their interplay. While most bundle implementations employ
polyhedral cutting models combined with a suitable active set QP
approach, the ConicBundle callable library \cite{ConicBundle2021} is
primarily designed for (nonnegative combinations of) conic cutting
models built from symmetric cones. In particular the cone of positive
semidefinite matrices and the second order cone lead to
nonpolyhedral models that change significantly in each step. In
solving \eqref{eq:candidate} for these models, interior point methods
are currently the best option available, but with these methods the
cost of assembling the coefficients and solving the subproblem
dominates the work per iteration for most applications. This paper
explores possibilities to replace the classical Schur complement
approach for computing the Newton step by an iterative approach
in order to improve applicability to large scale problems.

\upd{As the main focus of this work is on solving
  problem \eqref{eq:candidate} for a single iteration, we will refrain from
  giving the iteration index $k$ in the following.}
In order to describe the main structure of the primal dual KKT system
for \eqref{eq:candidate}, let us briefly sketch the
conic cutting models employed. These build on combinations of the
cone of nonnegative vectors, the second order cone and the cone
of positive semidefinite matrices, each with a specific trace
vector for measuring the ``size'' of its elements, 
\begin{align*}
  x\in\R^n_+&:=\{x\in\R^n\colon x_i\ge 0, i=1,\dots,n\} &&\Leftrightarrow\colon x\ge 0,&&\text{trace vector }\one_n=\begin{sbmatrix}1\\[-1ex]\vdots\\1\end{sbmatrix},\\  
x\in \socn&:=\left\{\vectwo{x_1}{\bar x}\in\R^n\colon x_1\ge\|\bar x\|\right\}&&
\Leftrightarrow\colon x\unrhd 0,&&\text{trace vector }e_1=\begin{sbmatrix}1\\0\\[-1ex]\vdots\\0\end{sbmatrix},\\
X\in\Sym^n_+&:=\left\{X\in\Sym^n\colon \lambda_{\min}(X)\ge 0\right\}&&\Leftrightarrow\colon X\succeq 0, &&\text{trace matrix }I_n.
\end{align*}
Cartesian products of these are described
by a triple $t=(n,q,s)$ with $n\in \N_0$, $q\in\N^{n_q}$, $s\in\N^{n_s}$ for some  $n_q\in\N_0$, $n_s\in\N_0$ specifying the cone
\begin{equation}\label{eq:St}
  \mathcal{S}^t_+:=\R_+^n\times \bigtimes_{i=1}^{n_q} \soc{q_i}\times \bigtimes_{i=1}^{n_s} \Sym_+^{s_{i}}.
\end{equation}
The cone $\mathcal{S}^t_+$ will be regarded as a full dimensional cone in $\R^t:=\R^{n(t)}$ with $n(t)=n+\one^\top q+\sum_{i=1}^{n_s}{s_i+1\choose 2}$. Whenever convenient,
an $x\in \mathcal{S}^t_+$ or $\R^t$ will be split into
\begin{equation}\label{eq:xt}
x=(\xi^\top,(x^{(1)})^\top,\dots,(x^{(n_q)})^\top,\svec(X^{(1)})^\top,\dots,\svec(X^{(n_s)})^\top)^\top
\end{equation}
in the natural way. Indices or elements may be omitted if the corresponding counters $n$, $n_q$, $n_s$ happen to be 0 or 1. 
The \emph{trace} $\tr(\cdot)$ of an element $x\in S_t$ is defined to be 
\begin{displaymath}
  \tr x := \ip{\one_t}{x}:= \ip{\one}{\xi}+\sum_{i=1}^{n_q}\ip{e_1}{x^{(i)}}+\sum_{i=1}^{n_s} \ip{I}{X^{(i)}}.
\end{displaymath}
In ConicBundle each cutting model may be considered to be specified by a tuple  $M=(t,\upd{\tau},K,\calB,\underline\omega)$ via
\begin{displaymath}
  W_M(y)=\maxset{x\in \mathcal{S}^t_+\\ \tau-\ip{\one_t}{x}\in K}[\underline\omega(y) +\ip{\calB(y)}{x}]
\end{displaymath}
where
\begin{align*}
  &t=(n,q,s)&&\text{ specifies the cone as above,}\\
  & \tau>0 &&\text{ gives the trace value or trace upper bound,}\\
  & K\in\{\{0\},\R_+\}&&\text{ specifies constant or bounded trace,}\\
  &\begin{array}{@{}r@{\;}r@{\,}c@{\,}l}\mathcal{B}\colon&\R^m&\to&\R^{n(t)}\\
     &y&\mapsto&\mathcal{B}(y)=B_0+By\end{array} &&\text{ represents the bundle as affine function,}\\
  & \underline{\omega}=(\underline{\gamma},\underline{g}) &&\text{ provides a constant offset subgradient}.
\end{align*}
For example, the standard polyhedral model for $h$ subgradients $\omega_i=(\gamma_i,g_i)$, $i=1,\dots,h$ is obtained for $M=(t=(h,0,0),\tau=1,K=\{0\},
(B_0=\begin{sbmatrix}
  \gamma_1\\[-1ex]\vdots\\\gamma_h
\end{sbmatrix},B=
\begin{sbmatrix}
  g_1^\top\\[-1ex]\vdots\\g_h^\top
\end{sbmatrix}),\underline{\omega}=0)$. Indeed, maximizing
$\ip{\calB(y)}{x}$ over $\xi\in\R^h_+$ with $\one^\top\xi=1$ finds the
best convex combination of the subgradients at $y$. In polyhedral
models $\calB$ may well be sparse \upd{(consider \eg Lagrangean
  relaxations of multi-commodity flow problems like in the train
time-tabling application described in \cite{FischerHelmberg2014a,FischerHelmberg2014b})}, but in combination with positive
semidefinite models large parts of $\calB$ will be dense \upd{(see
\cite{HelmbergRendl2000}; for concreteness, consider
$f(y)=\lmax(\sum_{i=1}^mA_iy_i)$ with $A_i\in\Sym^n$ and let the spectral bundle be
described by $h'$ orthonormal columns $P\in\R^{n\times h'}$ that
approximately span the eigenspace to large eigenvalues,  then
$h={h'+1\choose 2}$ and $B_{\bullet,i}=\svec(P^\top
A_iP)$ for $i=1,\dots,m$)}. Therefore we will not assume any specific structure in the bundle $\calB$.

Depending on the variable metric heuristic in use, see \cite{HelmbergOvertonRendl2014,HelmbergPichler2017} for typical choices in ConicBundle, the proximal term may either be a positive multiple of $I_m$ or it may be of the form $\upd{\mathfrak{H}}=D_\upd{\mathfrak{H}}+V_\upd{\mathfrak{H}}V_\upd{\mathfrak{H}}^\top\in\Sym^n_+$, where $D_\upd{\mathfrak{H}}$ is a diagonal matrix with strictly positive diagonal entries and $V_\upd{\mathfrak{H}}\in\R^{m\times h_\upd{\mathfrak{H}}}$ specifies a rank $h_\upd{\mathfrak{H}}$ contribution. If $V_\upd{\mathfrak{H}}$ is present, it typically consists of dense orthogonal columns.

The final ingredient is the basic optimization set $C$, which may have a polyhedral description of the form
\begin{displaymath}
  C=\{y\colon \underline{a} \le A y \le \overline{a}, \underline{y}\le y\le\overline{y}\},
\end{displaymath}
where $A\in\R^{h_A\times m}$, $\underline{a}\in
(\R\cup\{-\infty\})^{h_A},\overline{a}\in(\R\cup\{\infty\})^{h_A}$,
$\underline{y}\in(\R\cup\{-\infty\})^m,\overline{y}\in(\R\cup\{\infty\})^m$
are given data.
If $A$ is employed in applications, we expect the number of rows $h_A$ of $A$ to be small in comparison to $m$. The set $C$ is tested for feasibility in advance, but no preprocessing is applied.

In order to reduce the indexing load in this presentation, we consider
problem \eqref{eq:candidate} \upd{for a single} model 
$(t,\upd{\tau},K,\mathcal{B},\underline\omega=(\underline\gamma,\underline
g))$. Putting $b=-\upd{\mathfrak{H}}\hat y+\underline{g}$ and $\delta=\frac12\ip{\upd{\mathfrak{H}}\hat
  y}{\hat y}+\underline\gamma$ the bundle problem may be written in
the form 
\begin{equation}\label{eq:saddle}
  \minset{Ay-w=0\\\underline{a}\le w\le \overline{a}\\\underline{y}\le y \le \overline{y}}
  \maxset{x\in \mathcal{S}^t_+\\ \upd{\tau}-\ip{\one_t}{x}-\sigma=0\\\sigma \in K}
  \left[\frac12\ip{\upd{\mathfrak{H}}y}{y}+\ip{b}{y}+\ip{B_0}{x}+\ip{By}{x}+\delta\right].
\end{equation}
The existence of saddle points is guaranteed by compactness of the $x$-set and
by strong convexity for $y$ due to $\upd{\mathfrak{H}}\succ 0$, see \eg \cite{HiriartUrrutyLemarechal93a}. 
For the purpose of presentation, the primal dual KKT system
for solving the saddle point problem \eqref{eq:saddle} is built for $\underline{a}<\overline{a}$, $\underline{y}<\overline{y}$ and
$K=\R_+$. The extension to the equality cases follows quite naturally
and will be commented on at appropriate places.  Throughout we will
assume that $A,A^\top$ and $B,B^\top$ are given by matrix-vector multiplication
oracles. $A$ is assumed to have few rows, $B$ may actually have a
large number of columns, $\upd{\mathfrak{H}}$ is a positive diagonal plus low rank, but
no further structural information is assumed to be available.

The original spectral bundle approach of \cite{HelmbergRendl2000} was designed for the
unconstrained case $C=\R^m$ which allows direct elimination of $y$ by
convex optimality. Setting up the maximization problem for $x$ then
requires forming the typically dense Schur complement
$B\upd{\mathfrak{H}}^{-1}B^\top$. For increasing bundle sizes this is in fact the
limiting operation within each bundle iteration. The aim of
developing an iterative approach for \eqref{eq:saddle} is therefore
not only to allow for general $C$ but also to circumvent the explicit
computation of this Schur complement.

For setting up a primal-dual interior point approach for solving \eqref{eq:saddle}, the dual variables to constraints on the minimizing side
will be denoted by $s\in\R^{h_A}$, $s_{\underline{a}},s_{\overline{a}}\in\R^{h_A}_+$, $s_{\underline{y}},s_{\overline{y}}\in\R^m_+$, the dual variables to the constraints on the maximizing side will be 
$z\in\mathcal{S}^t_+$ and $\zeta\in K^*=\R_+$.

\begin{equation}\label{eq:origKKT}
  \begin{array}{c}
  \begin{array}{ccccccccccc}
    \upd{\mathfrak{H}}y & + &A^\top s & + &B^\top x & &    &+&s_{\overline{y}} -s_{\underline{y}} &= & -b \\
    Ay &  &        &  &          & &                 &-&w &=&0 \\
    By &   &       &  &          & -&\one_t\zeta     & +&z & =&-B_0 \\
       & &     &-& \ip{\one_t}{x} & &                 &-&\sigma&=&-\upd{\tau}\\
       &- &s     & & & &                 &+& s_{\overline{a}}-s_{\underline{a}} &=& 0\\
  \end{array}\\\\
  \begin{array}{ccc@{\qquad\qquad}ccc}
    (w-\underline{a})\circ s_{\underline{a}} &=& \mu\one & (\overline{a}-w)\circ s_{\overline{a}} &=& \mu\one \\
    (y-\underline{y})\circ s_{\underline{y}} &=& \mu\one & (\overline{y}-y)\circ s_{\overline{y}} &=& \mu\one \\
    x\circ_tz &=& \mu \one_t & \sigma\zeta &=&\mu.
  \end{array}
    \end{array}
\end{equation}
In this, ``$\circ$'' denotes the componentwise Hadamard product and ``$\circ_t$'' a canonical
generalization to the cone $\mathcal{S}^t_+$, employing the arrow operator for second order
cone parts and (typically symmetrized) matrix products for semidefinite parts.

In solving this by Newton's method, the linearization of the first perturbed complementarity
line yields
\begin{align*}
    \Delta s_{\underline a}&=\mu(w-\underline{a})^{-1}-s_{\underline{a}}-\Delta w\circ s_{\underline{a}}\circ(w-\underline{a})^{-1}\\
    \Delta s_{\overline
    a}&= \mu(\overline{a}-w)^{-1}-s_{\overline{a}}+\Delta w\circ
          s_{\overline{a}}\circ(\overline{a}-w)^{-1}
\end{align*}
\upd{For writing the difference  $\Delta s_{\overline a}-\Delta s_{\underline a}$ compactly it is advantageous
to introduce
\begin{align*}
    d_w &:=s_{\overline{a}}\circ(\overline{a}-w)^{-1}+s_{\underline{a}}\circ(w-\underline{a})^{-1}\\
  c_w&:=(\overline{a}-w)^{-1}-(w-\underline{a})^{-1}
\end{align*}
so that 
\begin{displaymath}
   \Delta s_{\overline a}-\Delta s_{\underline a}=\Delta w\circ d_w+s_{\underline{a}}-s_{\overline{a}}-\mu c_w.
\end{displaymath}
Likewise, for the second perturbed complementarity line, introduce
\begin{align*}
  d_y&:=s_{\overline{y}}\circ(\overline{y}-y)^{-1}+s_{\underline{y}}\circ(y-\underline{y})^{-1}\\
  c_y&:=(\overline{y}-y)^{-1}-(y-\underline{y})^{-1}
\end{align*}
to obtain
\begin{displaymath}
                  \Delta s_{\upd{\overline y}}-\Delta s_{\upd{\underline y}}= \Delta y\circ d_y+s_{\underline{y}}-s_{\overline{y}}-\mu c_y.
\end{displaymath}}
For dealing with the conic complementarity ``$\circ_t$'' we employ the
symmetrization operators $\calE_t$ and $\calF_t$ of \cite{ToddTohTuetuencue98} in diagonal block form corresponding to $\calS^t_+$, which give rise to a symmetric positive definite 
$\frakX_t=\calE_t^{-1}\calF_t\succ 0$ with diagonal block structure according to $\calS^t_{++}$. With this the last perturbed complementarity line results in
\begin{displaymath}
  \begin{array}{ccl}
    \Delta z&=&-\frakX_t^{-1}\Delta x + \mu x^{-1}-z,\\
    \Delta \sigma & = & -\zeta^{-1}\sigma\Delta \zeta + \mu \zeta^{-1}-\sigma.
  \end{array}
\end{displaymath}
Employing the linearization of the defining equation for $s$, 
\begin{displaymath}
  \Delta s_{\overline a}-\Delta s_{\underline a}=\Delta s + s +s_{\underline{a}}- s_{\overline{a}},
\end{displaymath}
the variable $\Delta w$ may now be eliminated via
\begin{displaymath}
  \Delta w=d_w^{-1}\circ \Delta s+d_w^{-1}\circ s+\mu d_w^{-1}\circ c_w.
\end{displaymath}
Put $D_y:=\Diag(d_y)>0$ and $D_w:=\Diag(d_w)>0$, then the Newton step is obtained by solving the system
\begin{equation}\label{eq:fullKKT}
  \lmat
  \begin{array}{cccc}
    \upd{\mathfrak{H}}+D_y&A^\top&B^\top & 0 \\
    A & -D_w^{-1} &  0  & 0 \\
    B & 0 & -\frakX_t^{-1} & -\one_t \\
    0 & 0 & -\one_t^\top & \zeta^{-1}\sigma
  \end{array}\rmat
  \lvec
  \begin{array}{c}
    \Delta y\\ \Delta s \\ \Delta x \\ \Delta \zeta
  \end{array}\rvec
  =
  \lvec
  \begin{array}{c}
    r_y\\ r_s \\ r_x \\ r_\zeta
  \end{array}\rvec,
\end{equation}
with right hand side 
\begin{displaymath}
  \lvec
  \begin{array}{c}
    r_y\\ r_s \\ r_x \\ r_\zeta
  \end{array}\rvec
  =
  \lvec
  \begin{array}{ccl}
    -(\upd{\mathfrak{H}}y+B^\top x+A^\top s+b)&+&\mu c_y\\
    -(Ay-w)+d_w^{-1}\circ s&+&\mu d_w^{-1}\circ c_w \\
    -(By+B_0-\one_t\zeta)&-&\mu x^{-1}\\
    -(\upd{\tau}-\ip{\one_t}{x})&+&\mu \zeta^{-1}
  \end{array}\rvec.
\end{displaymath}
The right hand side may be modified as usual to obtain predictor and corrector right hand sides, but this will not be elaborated on here. 
Note,  for $K=\{0\}$ the same system works with $\sigma=0$ and without the centering
terms associated with $\zeta$ in $r_\zeta$. Likewise, whenever line $i$ of $A$ corresponds to an equation, \ie $\underline{a}_i=\overline{a}_i$, the respective entry of $d_w^{-1}$ has to be replaced by zero.

Eq.~\eqref{eq:fullKKT} is a symmetric indefinite system, that could be solved
by appropriate iterative methods directly. So far, however,
we were not able to conceive suitable general preconditioning approaches
for exploiting the given structural properties in the full system.
Surprisingly, a viable path seems to be offered by the traditional
Schur complement approach after all. The resulting system allows to perform matrix vector
multiplications at minimal additional cost and is frequently positive
 definite. 

To see this, first take the Schur complement with respect to the $\frakX_t^{-1}$ block,
\begin{displaymath}\label{eq:XiZSchur}
  \lmat
  \begin{array}{@{}c@{\quad}c@{\quad}c@{\ }}
    \upd{\mathfrak{H}}+D_y+B^\top\frakX_tB &A^\top& -B^\top\frakX_t\one_{t} \\
    A &  -D_w^{-1}  & 0 \\
    -(B^\top\frakX_t\one_{t})^\top  & 0 & \zeta^{-1}\sigma+\one_t^\top\frakX_t\one_t \\
  \end{array}\rmat.
\end{displaymath}
Assuming $d_w>0$ (no equality constraint rows in $A$), eliminate the second and third block with further Schur complements and
split $B^\top \frakX_t B=B^\top\frakX_t^{\frac12}\frakX_t^{\frac12}B$,
\begin{equation}\label{eq:SchurH}
  H= \lmat
    \upd{\mathfrak{H}}+D_y+B^\top\frakX_t^{\frac12}(I-\tfrac{\frakX_t^{\frac12}\one_{t}(\frakX_t^{\frac12}\one_{t})^\top}{\zeta^{-1}\sigma+\one_t^\top\frakX_t\one_t})\frakX_t^{\frac12}B +A^\top D_wA  
  \rmat.
\end{equation}
Also in the equality case of $\sigma=0$ the matrix $I-\tfrac{\frakX^{\frac12}\one_{t}(\frakX^{\frac12}\one_{t})^\top}{\zeta^{-1}\sigma+\one_t^\top\frakX_t\one_t}\succeq 0$ is positive semidefinite, so the resulting system is positive definite. Equality constraints in $A$ induce zero diagonal elements in $D_w$ (or in $-D_w^{-1}$). In this case the corresponding rows will not be eliminated and give rise to an indefinite system
of the form $\begin{sbmatrix}
  H & \tilde A^\top \\ \tilde A & 0
\end{sbmatrix}$ with a large positive definite block $H$ and hopefully few
further rows in $\tilde A$. For such systems it is well studied how to
employ a preconditioner for $H$ to solve the full indefinite
system with \eg\ MINRES, see \cite{ElmanSilvesterWathen2006}.

The cost of multiplying the full KKT matrix of \eqref{eq:fullKKT} by a
vector is roughly the same as that of multiplying $H$ by a vector. Indeed, the same
multiplications arise for $\upd{\mathfrak{H}}+D_y,A,A^\top,B,B^\top$. So it remains to compare
the cost of a multiplication by $\begin{sbmatrix}
 -\frakX^{-1} &-\one_t\\-\one_t^\top &\zeta^{-1}\sigma 
\end{sbmatrix}$
to a multiplication by $\frakX_t^{\frac12}(I-\tfrac{\frakX_t^{\frac12}\one_{t}(\frakX_t^{\frac12}\one_{t})^\top}{\zeta^{-1}\sigma+\one_t^\top\frakX_t\one_t})\frakX_t^{\frac12}=\frakX_t-\tfrac{\frakX_t\one_{t}(\frakX_t\one_{t})^\top}{\zeta^{-1}\sigma+\one_t^\top\frakX_t\one_t}$. Recall that $\frakX_t$ is a block diagonal matrix with a separate block for each cone $\R_+$, $\soc{q_i}$, $\Sym^{s_i}_+$ specified by $t$ and the cost of
multiplying by $\frakX_t$ or $\frakX_t^{-1}$ is identical. Thus the only difference
are the multiplications by $\one_t$ in the first case and by the precomputed  vector $\frakX_t\one_{t}$ in the second. The vector $\frakX_t\one_{t}$ may be formed at almost
negligible cost along with setting up $\frakX_t$. So there is no noteworthy difference in the cost of matrix vector multiplications between the two systems and no
structural advantages are lost when working with $H$ instead of the full system. We will therefore concentrate on developing a preconditioner for $H$.

For this note that $H$ of \eqref{eq:SchurH} arises from adding a
Gram matrix to a positive diagonal,
\begin{displaymath}
  H=D+VV^\top,
\end{displaymath}
where (recall $\upd{\mathfrak{H}}=D_\upd{\mathfrak{H}}+V_\upd{\mathfrak{H}}V_\upd{\mathfrak{H}}^\top$)
\begin{equation}\label{eq:tildeV}
  D=D_\upd{\mathfrak{H}}+D_y \quad\text{and}\quad V=\left[V_\upd{\mathfrak{H}}, A^\top D_w^{\frac12}, B^\top\frakX_t^{\frac12}(I-\tfrac{\frakX_t^{\frac12}\one_{t}(\frakX_t^{\frac12}\one_{t})^\top}{\zeta^{-1}\sigma+\one_t^\top\frakX_t\one_t})^{\frac12}\right].
\end{equation}
Note that the multiplication of $V$ with a vector requires only a little bit more than half the number of operations of multiplying $H$ (or the full KKT matrix) with a vector. This suggests to explore possibilities of finding low rank approximations of $V$ for preconditioning. 

\section{Low rank preconditioning a Gram-matrix plus positive diagonal}\label{S:lowrankprecond}

Consider a matrix
\begin{displaymath}
  H=D+VV^\top
\end{displaymath}
with a positive definite matrix $D\in\Sym^m_{++}$ and
$V\in\R^{m\times n}$. In our application $D$ is diagonal, but the
results apply for general $D\succ 0$. This is applicable in practice
as long as $D^{-1}$ can be applied efficiently to vectors.
Matrix $V$ is assumed to be given by a matrix-vector multiplication
oracle, \ie $V$ and $V^\top$ may be multiplied by vectors but the
matrix does not have to be available explicitly.  

For motivating the following preconditioning approaches, first
consider \upd{(without actually computing it)} the singular value decomposition of
$$D^{-\frac12}V=Q_H\left[\begin{array}{c}\Sigma\\
    0\end{array}\right]P_H^\top$$ with orthogonal
$Q_H\in\R^{m\times m}$, diagonal
$\Sigma=\Diag(\sigma_1,\dots,\sigma_n)$ ordered nonincreasingly by
$\sigma_1\ge\cdots\ge\sigma_n\ge 0$ and orthogonal
$P_H\in\R^{n\times n}$ (for convenience, it is assumed that $n\le
m$). Then
\begin{displaymath}
  H=D^{\frac12}
  Q_H\left[\begin{array}{cc}I_{ n}+\Sigma^2& 0 \\ 0 & I_{m- n}\end{array}\right]
   Q_H^\top D^{\frac12},
\end{displaymath}
\begin{displaymath}
  H^{-1}=D^{-\frac12}
Q_H\left[\begin{array}{cc}(I_{ n}+\Sigma^2)^{-1}& 0 \\ 0 &
                                                                      I_{m- n}\end{array}\right]
Q_H^\top D^{-\frac12}.
\end{displaymath}
When $\Sigma$ is replaced by the $k$ largest singular values, this gives rise to a
good ``low rank'' preconditioner, see Theorem~\ref{T:kappa} below.  Computing the full matrix
$D^{-\frac12}V$ and its singular value decomposition will in general
be too costly or even impossible.  Instead the general idea is to work
with $D^{-\frac12}V\Omega$ for some random or deterministic choice of
$\Omega\in\R^{n\times k}$. 

Multiplying by a random $\Omega$ may be thought of as giving rise to a
subspace approximation in the style of Johnson-Lindenstrauss,
cf.~\cite{Achlioptas2001,DasguptaGupta2002}, and this formed the
starting point of this investigation. The actual randomized approach
and analysis, however, mainly builds on
\cite{HalkoMartinssonTropp2011} and the bounding techniques presented
there. For the deterministic preconditioning variant the
recent work \cite{HabibiKavandKocvaraStingl2021} provided strong
guidance for analyzing the condition number.

Here, $\Omega$ will mostly consist of orthonormal columns. Yet it is
instructive to consider more general cases, as well.
An arbitrary $\Omega\in\R^{n\times k}$ gives rise to the preconditioner
\begin{equation}\label{eq:HOmega}
  {\pcH{\Omega}}:=D+V\Omega\Omega^\top V=D^{\frac12}
  Q_H\left[\begin{array}{cc}I_{ n}+\Sigma P_H^\top\Omega\Omega^\top P_H\Sigma& 0 \\ 0 & I_{m- n}\end{array}\right]
   Q_H^\top D^{\frac12}.
\end{equation}
Putting $\pcG{\Omega}:=D^{\frac12}Q_H
\begin{sbmatrix}
  I_{ n}+\Sigma P_H^\top\Omega\Omega^\top P_H\Sigma & 0  \\ 0 & I_{m-n}
\end{sbmatrix}^{\frac12}$ we have ${\pcH{\Omega}=\pcG{\Omega}\pcG{\Omega}^\top}$. The preconditioner is better the
closer $\upd{\pcG{\Omega}}^{-1}H\upd{\pcG{\Omega}}^{-T}$ is to the identity. In the analysis of convergence rates, see \eg \cite{ElmanSilvesterWathen2006}, this enters via the condition number 
\begin{displaymath}
  \kappa_{\Omega}:=\frac{\lmax(\upd{\pcG{\Omega}}^{-1}H\upd{\pcG{\Omega}}^{-T})}{\lmin(\upd{\pcG{\Omega}}^{-1}H\upd{\pcG{\Omega}}^{-T})}=\frac{\lmax(H^{\frac12}{\pcH{\Omega}}^{-1}H^{\frac12})}{\lmin(H^{\frac12}{\pcH{\Omega}}^{-1}H^{\frac12})}=\frac{\lmax({\pcH{\Omega}}^{-\frac12}H{\pcH{\Omega}}^{-\frac12})}{\lmin({\pcH{\Omega}}^{-\frac12}H{\pcH{\Omega}}^{-\frac12})}.
\end{displaymath}
In this, the equations follow from $BB^\top$ and $B^\top B$ having the same eigenvalues for $B\in\R^{n\times n}$. 
\begin{theorem}\label{T:kappa} Let $H=D+VV^\top\in \Sym^m_{++}$ with positive
  definite $D\in\Sym^m_{++}$ and $V\in \R^{m\times n}$ with
  $n<m$ and singular value decomposition
  $D^{-\frac12}V=Q_H\Sigma P_H^\top$, $Q_H^\top Q_H=I_m$, $P_H^\top
  P_H=I_n$, $\Sigma=\Diag(\sigma_1\ge\dots\ge\sigma_n)\in\Sym^n_+$. For  $\Omega\in\R^{n\times k}$
  the preconditioner ${\pcH{\Omega}}$ of \eqref{eq:HOmega} results in condition number  
  \begin{eqnarray*}
    \kappa_\Omega&=&\frac{\max\{1,\lmax((I_{ n}+\Sigma^2)^{\frac12}(I_{ n}+\Sigma P_H^\top \Omega\Omega^\top P_H\Sigma)^{-1}(I_{ n}+\Sigma^2)^{\frac12})\}}{\min\{1,\lmin((I_{ n}+\Sigma^2)^{\frac12}(I_{ n}+\Sigma P_H^\top\Omega\Omega^\top P_H\Sigma)^{-1}(I_{ n}+\Sigma^2)^{\frac12})\}}\\
    &=&\frac{\max\{1,\lmax((I_{ n}+\Sigma^2)^{-\frac12}(I_{ n}+\Sigma P_H^\top \Omega\Omega^\top P_H\Sigma)(I_{ n}+\Sigma^2)^{-\frac12})\}}{\min\{1,\lmin((I_{ n}+\Sigma^2)^{-\frac12}(I_{ n}+\Sigma P_H^\top\Omega\Omega^\top P_H\Sigma)(I_{ n}+\Sigma^2)^{-\frac12})\}}
  \end{eqnarray*}
  In particular, for $0\le k<n$ and $\Omega=(P_H)_{\bullet,[1,\dots,k]}$ the condition number's value is $1+\sigma_{k+1}^2$.
\end{theorem}
\begin{proof} For $\pcG{\Omega}$ as above direct computation yields
\begin{eqnarray*}
  \upd{\pcG{\Omega}}^{-1}H\upd{\pcG{\Omega}}^{-T}&=&\begin{sbmatrix}
  I_{ n}+\Sigma\Omega\Omega^\top\Sigma & 0  \\ 0 & I_{m-n}
\end{sbmatrix}^{-\frac12}\begin{sbmatrix}
  I_{ n}+\Sigma^2 & 0  \\ 0 & I_{m-n}
\end{sbmatrix}\begin{sbmatrix}
  I_{ n}+\Sigma\Omega\Omega^\top\Sigma & 0  \\ 0 & I_{m-n}
\end{sbmatrix}^{-\frac12}\\
 &=&\begin{sbmatrix}
  (I_{ n}+\Sigma\Omega\Omega^\top\Sigma)^{-\frac12}(I_{ n}+\Sigma^2)(I_{ n}+\Sigma\Omega\Omega^\top\Sigma)^{-\frac12} & 0  \\ 0 & I_{m-n}
\end{sbmatrix}.                      
\end{eqnarray*}
The eigenvalues of $(I_{ n}+\Sigma\Omega\Omega^\top\Sigma)^{-\frac12}(I_{ n}+\Sigma^2)(I_{ n}+\Sigma\Omega\Omega^\top\Sigma)^{-\frac12}$ coincide with
those of $(I_{ n}+\Sigma^2)^{\frac12}(I_{ n}+\Sigma\Omega\Omega^\top\Sigma)^{-1}(I_{ n}+\Sigma^2)^{\frac12}$, because for $B=(I_{ n}+\Sigma\Omega\Omega^\top\Sigma)^{-\frac12}(I_{ n}+\Sigma^2)^{\frac12}$ the two matrices are $BB^\top$ and $B^\top B$. This gives rise to the first line. The second follows because for positive definite $A$ there holds $\lmax(A)=1/\lmin(A^{-1})$ and $\lmin(A)=1/\lmax(A^{-1})$. 
\end{proof}
Consider now a subspace spanned by $k$ orthonormal columns collected in
some matrix $P_ \ell\in \R^{n\times k}$ which hopefully generates most of
the large directions of  $D^{-\frac12}V$. In this orthonormal case a simpler bound on the condition number may be obtained by following the argument of Th.~5.1 in \cite{HabibiKavandKocvaraStingl2021}.
\begin{theorem}\label{T:deterministiccond}
  Let $H=D+VV^\top$ with positive definite $D\in\Sym^m_{++}$ and
  general $V\in\R^{m\times n}$, let $P=[\upd{\bar P},\upd{\underline{P}}]\in\R^{n\times n}$,
  $PP^\top=I_n$. Preconditioner ${\pcH{\bar P}}=D+V\upd{\bar P} \upd{\bar P}^\top
  V^\top$ has condition number $\kappa_{\upd{\bar P}}\le
  1+\lmax({\pcH{\bar P}}^{-\frac12}V\upd{\underline{P}}\,\upd{\underline{P}}^\top V^\top
  {\pcH{\bar P}}^{-\frac12})$. Equality holds if and only if $\rank(V\upd{\underline{P}})<m$. 
\end{theorem}
\begin{proof}
  Because $H=D+[V\upd{\bar P},V\upd{\underline{P}}][V\upd{\bar P},V\upd{\underline{P}}]^\top={\pcH{\bar P}}+V\upd{\underline{P}}\,\upd{\underline{P}}^\top V^\top$ we have
  \begin{displaymath}
    {\pcH{\bar P}}^{-\frac12}H {\pcH{\bar P}}^{-\frac12}= I_m+{\pcH{\bar P}}^{-\frac12}V\upd{\underline{P}}\,\upd{\underline{P}}^\top V^\top {\pcH{\bar P}}^{-\frac12}.
  \end{displaymath}
  The second summand is positive semidefinite with minimum eigenvalue
  0 if and only if $\rank(V\upd{\underline{P}})<m$. Thus,
  \begin{eqnarray*}
    \lmin({\pcH{\bar P}}^{-\frac12}H{\pcH{\bar P}}^{-\frac12})&\ge &1,\\
    \lmax({\pcH{\bar P}}^{-\frac12}H{\pcH{\bar P}}^{-\frac12})&= & 1+\lmax({\pcH{\bar P}}^{-\frac12}V\upd{\underline{P}}\,\upd{\underline{P}}^\top V^\top {\pcH{\bar P}}^{-\frac12}). 
  \end{eqnarray*}
  By $\kappa({\pcH{\bar P}}^{-\frac12}H{\pcH{\bar P}}^{-\frac12})=\frac{\lmax({\pcH{\bar P}}^{-\frac12}H{\pcH{\bar P}}^{-\frac12})}{\lmin({\pcH{\bar P}}^{-\frac12}H{\pcH{\bar P}})}$ the result is proved.
\end{proof}
Building on these two theorems we first analyze randomized approaches
that do not make any assumptions on structural properties of
$D^{-\frac12}V$ but only require a multiplication oracle. Afterwards
we present a deterministic approach that exploits some knowledge of
the bundle subproblem and the interior point algorithm. The
corresponding routines supply a $\hat V=V\Omega$. The actual
preconditioning routine, Alg.~\ref{alg:precond} below, does not use
${\pcH{\Omega}}$ directly, but a truncated preconditioner
${\pcH{\Omega\hat P}}$
that drops all singular values of $D^{-\frac12}V\Omega$ that are
less than one. The inverse is then formed via a Woodbury-formula, see
\cite[\S0.7.4]{HornJohnson85}.  Note, depending on the expected number
of calls to the routine and the structure preserved in $\hat V$, it may
or may not pay off to also precompute $\hat V\hat P$. For diagonal $D$
and dense $\hat V$ the cost of applying this preconditioner is
$O(m+mk+k\hat k)$.

\begin{algorithm}[Preconditioning by truncated
  ${\pcH{\Omega}}=D+V\Omega(V\Omega)^\top$]\label{alg:precond}\quad\\
\textbf{Input:} $v\in\R^m$, $D\in\Sym^n_{++}$, precomputed $\hat V=V\Omega\in\R^{n\times k}$
and, for $\hat V^\top D^{-1}\hat V=P\Diag(\hat
\lambda_1\ge\dots\ge\hat \lambda_k)P^\top$, 
$\hat k=\max\{0,i\colon\hat\lambda_i\ge 1\}$,
$\hat \Lambda=\Diag(\hat\lambda_1,\dots,\hat\lambda_{\hat k})$, $\hat
P=P_{\bullet,[1,\dots,\hat k]}$.\\
  \textbf{Output:} ${\pcH{\Omega\hat P}}^{-1}v$.
  \begin{compactenum}
  \item $v\gets D^{-1}v$.
  \item If $\hat k>0$ set $v \gets v - D^{-1}\hat
    V\hat P(I+\hat\Lambda)^{-1}\hat P^\top\hat V^\top v$.
  \item return $v$.
  \end{compactenum}
\end{algorithm}

\subsection{Preconditioning by Random Subspaces}\label{S:randprecond}

For the random subspace approach fix some
$k\in\N$ with $2\le k<n$\upd{. At first consider} $\Omega$ \upd{to} be a\upd{n} $n\times k$ random matrix whose elements are independently
identically distributed by the normal distribution $\mathcal{N}(0,\frac1{k})$. For this
$\Omega$ consider the low rank approximation $D^{-\frac12} V\Omega= 
Q_H\mat{c}{ \Sigma\\ 0}
P_H^\top\Omega$. Because the normal distribution is invariant under orthogonal
transformations, we may assume $ P_H=I$ and analyze the setting  $
Q_H\mat{c}{ \Sigma\\ 0}\Omega$ giving rise to the low rank approximation by the random
matrix
\begin{displaymath}
  {\pcH{\Omega}}=D^{\frac12}
  Q_H\left[\begin{array}{cc}I_{ n}+\Sigma\Omega\Omega^\top\Sigma& 0 \\ 0 & I_{m- n}\end{array}\right]
   Q_H^\top D^{\frac12}.
\end{displaymath}
In view of Theorem~\ref{T:kappa} such a preconditioner is good if $(I+\Sigma^2)^{-\frac12}(I+\Sigma
\Omega\Omega^\top\Sigma)(I+\Sigma^2)^{-\frac12}$ is close to the
identity. Based on the Johnson-Lindenstrauss interpretation,
it seems likely that large portions of the spectrum will be close to
one. This can be justified to some extent by studying the moments of
the Ritz values.
\begin{theorem}\label{T:moments}
  Let $\Omega\in\R^{ n\times k}$ have its elements i.i.d.\ according to
  the normal distribution $\mathcal{N}(0,\frac1{k})$, then for any $x\in\R^{ n}$ the
quadratic form
\begin{displaymath}
  q(x)=x^\top(I+\Sigma^2)^{-\frac12}(I+\Sigma
  \Omega\Omega^\top\Sigma)(I+\Sigma^2)^{-\frac12}x
\end{displaymath}
has expected value $\E(q(x))=\|x\|^2$ and variance $\Var(q(x))=\frac{2}{k}\big(\sum_{i=1}^{n}\frac{\sigma_i^2}{1+\sigma_i^2}x_i^2\big)^2$.
\end{theorem}
\begin{proof} Let $\Omega=(\omega_{ij})$ with i.i.d.\ elements
  $\omega_{ij}$ from $\mathcal{N}(0,\frac{1}{k})$. Recall that
  $\E(\omega_{ij})=0$, $\E(\omega_{ij}^2)=\frac1k$,
  $\E(\omega_{ij}^3)=0$, $\E(\omega_{ij}^4)=3/k^2$ and that for
  independent random variables $X,Y$ there holds $\E(XY)=\E(X)\E(Y)$.

  The expected value of the quadratic form evaluates to
\begin{eqnarray*}
  \lefteqn{\E\bigg(x^\top(I+\Sigma^2)^{-\frac12}(I+\Sigma
  \Omega\Omega^\top\Sigma)(I+\Sigma^2)^{-\frac12}x\bigg)=}\\
  &=&\sum_{i=1}^{
      n}\frac{1}{1+\sigma_i^2}x_i^2+\E\bigg(\sum_{h=1}^k\big(\sum_{i=1}^{
      n}\frac{\sigma_i}{\sqrt{1+\sigma_i^2}} \omega_{hi} x_{i}\big)^2\bigg)\\
  &=&\sum_{i=1}^{
      n}\frac{1}{1+\sigma_i^2} x_i^2+\sum_{h=1}^k\sum_{i=1}^{
      n}\frac{\sigma_i^2}{1+\sigma_i^2}x_{i}^2\E \omega_{hi}^2 \\
  &=& \sum_{i=1}^{ n} x_i^2\bigg(\frac{1}{1+\sigma_i^2} +
      \frac{\sigma_i^2}{1+\sigma_i^2}\sum_{h=1}^k\E \omega_{hi}^2\bigg)\\
  &=& \sum_{i=1}^{ n} x_i^2\bigg(\frac{1}{1+\sigma_i^2} +
      \frac{\sigma_i^2}{1+\sigma_i^2}\sum_{h=1}^k\frac1k\bigg) =
      \sum_{i=1}^{ n} x_i^2= \|x\|^2.
\end{eqnarray*}

For determining the variance, the second moment may be computed as follows.
\begin{eqnarray*}
\lefteqn{\E\bigg(\bigg[\sum_{h=1}^k\big(\sum_{i=1}^{
      n}\frac{\sigma_i}{\sqrt{1+\sigma_i^2}} \omega_{hi} x_{i}\big)^2\bigg]^2\bigg)=}\\
  &=&\E\bigg(\bigg[\sum_{h=1}^k\sum_{i=1}^{
      n}\sum_{i'=1}^{
      n}\frac{\sigma_i}{\sqrt{1+\sigma_i^2}}\frac{\sigma_{i'}}{\sqrt{1+\sigma_{i'}^2}}\omega_{hi} x_{i}\omega_{hi'} x_{i'}\bigg]^2\bigg)\\
  &=&\E\bigg(\sum_{h,h'=1}^k~
      \sum_{i,i',j,j'=1}^{ n} \\
  &&\phantom{\E\bigg(}\frac{\sigma_i}{\sqrt{1+\sigma_i^2}}\frac{\sigma_{i'}}{\sqrt{1+\sigma_{i'}^2}}\frac{\sigma_j}{\sqrt{1+\sigma_j^2}}\frac{\sigma_{j'}}{\sqrt{1+\sigma_{j'}^2}}x_{i} x_{i'}x_{j} x_{j'}\omega_{hi}\omega_{hi'}\omega_{h'j}\omega_{h'j'}\bigg).
\end{eqnarray*}
In the cases of $h\neq h'$ only terms with $i=i'$ and $j=j'$ are not zero. These
evaluate to $\E \omega_{hi}^2 \E \omega_{h'j}^2=\frac1{k^2}$ giving
\begin{eqnarray*}
  \frac{k(k-1)}{k^2}\sum_{i=1}^{
      n}\sum_{j=1}^{
      n}\frac{\sigma_i^2}{1+\sigma_i^2}x_{i}^2
     \frac{\sigma_j^2}{1+\sigma_j^2}x_{j}^2.
\end{eqnarray*}
For each $h=h'$ there remain $(i=i'=j=j')$ with value $\E \omega_{hi}^4=\frac3{k^2}$,
\begin{eqnarray*}
  \frac{3k}{k^2}\sum_{i=1}^{ n}\bigg(\frac{\sigma_i^2}{1+\sigma_i^2}\bigg)^2x_{i}^4,
\end{eqnarray*}
and the three pairings $(i=i',j=j')$,
$(i=j,i'=j')$ and $(i=j',i'=j)$ each with value $\frac1{k^2}$,
\begin{eqnarray*}
  \frac{3k}{k^2}\sum_{i\neq j}\frac{\sigma_i^2}{1+\sigma_i^2}x_{i}^2
     \frac{\sigma_j^2}{1+\sigma_j^2}x_{j}^2.
\end{eqnarray*}
Summing up these three expressions yields
\begin{displaymath}
  \E\bigg(\bigg[\sum_{h=1}^k\big(\sum_{i=1}^{
      n}\frac{\sigma_i}{\sqrt{1+\sigma_i^2}} \omega_{hi} x_{i}\big)^2\bigg]^2\bigg)=\frac{k^2+2k}{k^2}\bigg(\sum_{i=1}^{ n}\frac{\sigma_i^2}{1+\sigma_i^2}x_i^2\bigg)^2.
  \end{displaymath}
The result now follows from the usual $\Var X=\E(X^2)-(\E X)^2$ for any random variable $X$. 
\end{proof}
This suggests that even for relatively small $k$ the behavior of the
preconditioned system may be expected to be reasonably close to the
identity for a large portion of the directions.  The result, however,
does not seem to open a path towards good bounds on the condition
number.

A first possibility is offered by Theorem~\ref{T:deterministiccond}. 
Recall that for an arbitrary matrix $A\in\R^{m\times n}$ the projector
\begin{displaymath}
  \mathbf{P}_A=A(A^\top A)^\dagger A^\top
\end{displaymath}
projects any vector of $\R^m$ onto the range space of $A$ and
$\mathbf{P}$ depends only on this range space. Computationally
it may be determined by a $QR$-decomposition of $A=Q_AR_A$  with
orthogonal $Q_A\in\R^{m\times n'}$ for $n'=\rank(A)$ via
$\mathbf{P}_A=Q_AQ_A^\top$. The formula allows \upd{one} to verify $\mathbf{P}_A=\mathbf{P}_A^\top$, $\mathbf{P}_A\mathbf{P}_A=\mathbf{P}_A$ and $\mathbf{P}_AA=A$ by direct computation. Furthermore, for any $B\in\R^{n\times h}$ there holds $\mathbf{P}_{AB}\preceq \mathbf{P}_A\preceq I_m$ because of the containment relations between the ranges. 

In the following $\Omega\Omega^\top$ in ${\pcH{\Omega}}$ will be replaced by
the projector $\mathbf{P}_\Omega$. The random low rank approximation to be considered reads
\begin{displaymath}
  {\pcH{\mathbf{P}_\Omega}}=D^{\frac12}
  Q_H\left[\begin{array}{cc}I_{ n}+\Sigma \mathbf{P}_\Omega\Sigma& 0 \\ 0 & I_{m- n}\end{array}\right]
   Q_H^\top D^{\frac12}.
 \end{displaymath}
 \upd{The following deterministic result holds for any matrix $\Omega\in\R^{n\times k}$.}
\begin{corollary}
   Let $H=D+VV^\top\in \R^m$ with positive
   definite $D\in\Sym^m_{++}$ and $V\in \R^{m\times n}$.
  Given $\Omega\in\R^{n\times k}$, let
  $\mathbf{P}_\Omega=\Omega(\Omega^\top \Omega)^\dagger\Omega^\top$. For the
  preconditioner ${\pcH{\mathbf{P}_\Omega}}$ the condition number satisfies
  $\kappa_{\mathbf{P}_\Omega}\le 1+\|D^{-\frac12}V(I-\mathbf{P}_{\Omega})\|^2$,
  where $\|\cdot\|$ denotes the spectral norm. 
\end{corollary}
\begin{proof}
  Let $\mathbf{P}_\Omega=\upd{\bar P} \upd{\bar P}^\top$ with $\upd{\bar P}\in\R^{n\times
    k'}$ for some $k'\le k$ and $\upd{\bar P}^\top \upd{\bar P}=I_{k'}$. Add
  orthonormal columns $\upd{\underline{P}}$ so that $P=[\upd{\bar P},\upd{\underline{P}}]$ satisfies
  $PP^\top=I_n$. Note, $I-\mathbf{P}_\Omega=\upd{\underline{P}}\,\upd{\underline{P}}^\top$ is the projector onto
  the orthogonal complement. Use this choice in Theorem~\ref{T:deterministiccond},
  then ${\pcH{\mathbf{P}_\Omega}}={\pcH{\bar P}}\succeq D$. Observe that by $D\in \Sym^m_{++}$
  for any $\lambda\ge 0$, $A\in \Sym^m$  the relation $D^{-\frac12}AD^{-\frac12}\preceq
  \lambda I_m$ is equivalent to $A\preceq \lambda D$ and implies $A\preceq
  \lambda (D+V\upd{\bar P} \upd{\bar P}^\top V^\top)=\lambda {\pcH{\bar P}}$ which is
  equivalent to
  ${\pcH{\bar P}}^{-\frac12}A{\pcH{\bar P}}^{-\frac12}\preceq
  \lambda I_m$. Therefore  
  \begin{displaymath}
    \lmax(\pcH{\bar P}^{-\frac12}V\upd{\underline{P}}\,\upd{\underline{P}}^\top V^\top
    {\pcH{\bar P}}^{-\frac12})\le\lmax(D^{-\frac12}V\upd{\underline{P}}\,\upd{\underline{P}}^\top V^\top
    D^{-\frac12})=\|D^{-\frac12}V(I-\mathbf{P}_\Omega)\|^2.
  \end{displaymath}
\end{proof}
While this bound is rather straight forward to derive, it does not
seem strong enough to observe a reduced influence of the largest
singular values of $D^{-\frac12}V$.  Indeed, in its derivation only the
diagonal of ${\pcH{\mathbf{P}_\Omega}}$ was considered and the
influence of $V\Omega$ is lost.

In order to obtain stronger bounds, the rather involved techniques laid out in
\cite{HalkoMartinssonTropp2011} seem to be required.
The next steps and results follow their arguments closely. This time
the ${\pcH{\mathbf{P}_\Omega}}$-part is kept inverted in the analysis of the
condition number.

Because $I+\Sigma \upd{\mathbf{P}_{\Omega}}\Sigma\preceq I+\Sigma^2$ there holds
$$ (I+\Sigma^2)^{\frac12}(I+\Sigma
\upd{\mathbf{P}_{\Omega}}\Sigma)^{-1}(I+\Sigma^2)^{\frac12}\succeq I.$$ By Theorem~\ref{T:kappa} the condition number is bounded by
\begin{displaymath}
  \kappa_{\upd{\mathbf{P}_{\Omega}}}\le \lambda_{\max}( (I+\Sigma^2)^{\frac12}(I+\Sigma \upd{\mathbf{P}_{\Omega}}\Sigma)^{-1}(I+\Sigma^2)^{\frac12})
\end{displaymath}
and will attain it, whenever $n<m$.  In terms of $\Omega$, the best
possible outcome is an event resulting in
$\upd{\mathbf{P}_{\Omega}}=\mat{cc}{I_{k}&\upd{0} \\0 &\upd{0}}$ (see, \eg
\cite[7.4.52]{HornJohnson85}). It corresponds to the truncated SVD and
gives $\kappa_{
  \begin{sbmatrix}
I_{k}& \upd{0} \\0 & \upd{0}
\end{sbmatrix}
} = 1+\sigma_{k+1}^2$. Aiming for something more
realistic, one hopes for a good coverage of the first $k$ singular
values when oversampling with $\upd{\frakp}$ additional columns.

The first step in the analysis is to obtain a deterministic bound for
a fixed $\Omega\in\R^{ n\times(k+\upd{\frakp})}$ as outlined in
\cite[\S9.2]{HalkoMartinssonTropp2011}.
\begin{theorem}\label{th:deterministicbound}
  Given $\sigma_1\ge\dots\ge\sigma_{ n} \ge 0$ and a matrix $\Omega\in\R^{ n\times (k+\upd{\frakp})}$ with $k\le  n$ so that the
  first $k$ rows of $\Omega$ are linearly independent, split $\Sigma=\mat{cc}{\Sigma_1& 0\\ 0&\Sigma_2}$ into blocks $\Sigma_1=\Diag(\sigma_1,\dots,\sigma_k)$ and $\Sigma_2=\Diag(\sigma_{k+1},\dots,\sigma_{ n})$ and
$\Omega=\mat{c}{\Omega_1\\\Omega_2}$ into the first $k$ rows
$\Omega_1\in\R^{k\times(k+\upd{\frakp})}$ and the last $ n-k$ rows
$\Omega_2\in\R^{( n-k)\times k+\upd{\frakp}}$. Then
\begin{displaymath}
  \lambda_{\max}( (I+\Sigma^2)^{\frac12}(I+\Sigma \upd{\mathbf{P}_{\Omega}}\Sigma)^{-1}(I+\Sigma^2)^{\frac12})\le 2+ \sigma_{k+1}^2+\|(I+\Sigma_2^2)^{\frac12}\Omega_2\Omega_1^\dagger\|^2.
\end{displaymath}
\end{theorem}
\begin{proof}
By assumption $\Omega_1$ has full row rank and the range space of the matrix
\begin{displaymath}
  Z=\Omega\cdot\Omega_1^\dagger= \mat{c}{I_k\\ F=\Omega_2\Omega_1^\dagger} \in\R^{ n\times k}
\end{displaymath}
is contained in the range space of $\Omega$. Hence $\upd{\mathbf{P}_Z}\preceq \upd{\mathbf{P}_{\Omega}}$ and
\begin{equation}\label{eq:lmaxle}
\lambda_{\max}( (I+\Sigma^2)^{\frac12}(I+\Sigma \upd{\mathbf{P}_{\Omega}}\Sigma)^{-1}(I+\Sigma^2)^{\frac12})\le 
\lambda_{\max}( (I+\Sigma^2)^{\frac12}(I+\Sigma \upd{\mathbf{P}_Z}\Sigma)^{-1}(I+\Sigma^2)^{\frac12}).
\end{equation}
The projector $\upd{\mathbf{P}_Z}$ computes to
\begin{displaymath}
  \upd{\mathbf{P}_Z}=\mat{c}{I\\ F}\mat{c}{I+F^\top F}^{-1}\mat{c}{I\\F}^\top.
\end{displaymath}
Use this in the Woodbury-formula for inverses of rank adjustments \cite[\S0.7.4]{HornJohnson85} for $(I+\Sigma \upd{\mathbf{P}_Z}\Sigma)^{-1}$ to obtain
\begin{eqnarray}
(I+\Sigma  \upd{\mathbf{P}_Z}\Sigma)^{-1}&=&I-
                                           \mat{c}{\Sigma_1\\\Sigma_2
  F}\mat{c}{I+F^\top
  F+\Sigma_1^2+F^\top\Sigma_2^2F}^{-1}\mat{c}{\Sigma_1\\ \Sigma_2F
  }^\top\nonumber\\
                                       &=& I-\mat{c}{\Sigma_1\\\Sigma_1F}\mat{c}{I+\Sigma_1^2+F^\top(I+\Sigma_2^2)F}^{-1}\mat{c}{\Sigma_1\\\Sigma_2F}^\top\nonumber\\
  &\preceq& I- \mat{c}{\Sigma_1\\\Sigma_2F}\mat{c}{I+\Sigma_1^2+\|(I+\Sigma_2^2)^{\frac12}F\|^2I}^{-1}\mat{c}{\Sigma_1\\\Sigma_2F}^\top\hspace*{-1.5ex}.\label{eq:invpreceq}
\end{eqnarray}
The last line follows, because
$\|(I+\Sigma_2^2)^{\frac12}F\|^2=\lambda_{\max}(F^\top(I+\Sigma_2^2)F)=:\bar\lambda$
and therefore $F^\top(I+\Sigma_2^2)F\preceq \bar\lambda I$ giving
\begin{displaymath}
  I+\Sigma_1^2+F^\top(I+\Sigma_2^2)F\preceq
  I+\Sigma_1^2+\bar\lambda I \quad \Leftrightarrow\quad
  (I+\Sigma_1^2+F^\top(I+\Sigma_2^2)F)^{-1}\succeq
  (I+\Sigma_1^2+\bar\lambda I)^{-1},
\end{displaymath}
so the relation is implied by semidefinite scaling. Put
$\Lambda=(I+\Sigma_1^2+\bar\lambda I)$ and note that the second
diagonal block of \eqref{eq:invpreceq} asserts
$\Sigma_2F\Lambda^{-1}F^\top\Sigma_2\preceq I$, then
\begin{eqnarray*}
  \lefteqn{S:=(I+\Sigma^2)^{\frac12}(I+\Sigma
  \upd{\mathbf{P}_{\Omega}}\Sigma)^{-1}(I+\Sigma^2)^{\frac12}\overset{\eqref{eq:lmaxle}\eqref{eq:invpreceq}}{\preceq}}\\
&\preceq& (I+\Sigma^2)^{\frac12}\mat{cc}{I-\Sigma_1\Lambda^{-1}\Sigma_1 & -\Sigma_1\Lambda^{-1}F^\top\Sigma_2\\-\Sigma_2F\Lambda^{-1}\Sigma_1 &I-\Sigma_2F\Lambda^{-1}F^\top\Sigma_2} (I+\Sigma^2)^{\frac12}\\
  &\preceq& (I+\Sigma^2)^{\frac12}\mat{cc}{(\Lambda-\Sigma_1^2)\Lambda^{-1} & -\Sigma_1\Lambda^{-1}F^\top\Sigma_2\\-\Sigma_2F\Lambda^{-1}\Sigma_1 &I} (I+\Sigma^2)^{\frac12}\quad :=\bar S.
\end{eqnarray*}
Employing \cite[Prop.\ 8.3]{HalkoMartinssonTropp2011} now results in
  \begin{displaymath}
  \lambda_{\max}(S)\le\lambda_{\max}(\bar S)\le\|(I+\Sigma_1^2)(\Lambda-\Sigma_1^2)\Lambda^{-1}\|+\|I+\Sigma_2^2\|.
\end{displaymath}
The last term evaluates to $\lambda_{\max}(I+\Sigma_2^2)=1+\sigma_{k+1}^2$. For the second last term substituting in the definitions of $\Lambda$ and $\bar\lambda$ yields 
\begin{eqnarray*}
  \lefteqn{\|(I+\Sigma_1^2)(\Lambda-\Sigma_1^2)\Lambda^{-1}\|=}\\
  &=&(1+\|(I+\Sigma_2^2)^{\frac12}\Omega_2 \Omega_1^\dagger\|^2)\cdot \max_{i\in\{1,\dots,k\}} \frac{1+\sigma_i^2}{1+\sigma_i^2+\|(I+\Sigma_2^2)^{\frac12}\Omega_2 \Omega_1^\dagger\|^2}\\
  &=& \max_{i\in\{1,\dots,k\}} \frac{1+\sigma_i^2+\|(I+\Sigma_2^2)^{\frac12}\Omega_2 \Omega_1^\dagger\|^2+\sigma_i^2\|(I+\Sigma_2^2)^{\frac12}\Omega_2 \Omega_1^\dagger\|^2}{1+\sigma_i^2+\|(I+\Sigma_2^2)^{\frac12}\Omega_2 \Omega_1^\dagger\|^2}\\
  &\le& 1+\min\{\sigma_1^2,\|(I+\Sigma_2^2)^{\frac12}\Omega_2 \Omega_1^\dagger\|^2\}.
\end{eqnarray*}
\end{proof}
The current bound falls somewhat short of expectation because of
the identity in $\|(I+\Sigma_2^2)^{\frac12}\Omega_2\Omega_1^\dagger\|^2$.
By Theorem~\ref{T:kappa} and $I_n\preceq I_n+\Sigma \upd{\mathbf{P}_{\Omega}} \Sigma\preceq I_n+\Sigma^2$,
the use of projectors will never result in condition numbers larger than $1+\sigma_1^2$, so the
influence of the dimension seems to be too dominant in this.
Maybe a better bound is achievable by a more
sophisticated argument. 

The deterministic bound allows to also make use of the
probabilistic bounds on $\|(I+\Sigma_2^2)^{\frac12}\Omega_2
\Omega_1^\dagger\|$ \upd{for standard Gaussian $n\times(k+\frakp)$ matrices
  $\Omega$ (\ie matrix elements are independently $\calN(0,1)$ distributed)} given in \cite{HalkoMartinssonTropp2011}. These shed some light on the advantage of employing oversampling by $\upd{\frakp}$ additional random vectors in $\Omega$. In our application, oversampling
corresponds to computing the singular values of
$D^{-\frac12}V\mathbf{P}_\Omega$ for $k+\upd{\frakp}$ columns in order to get better
control on the $k$ largest singular values of $D^{-\frac12}V$ by the
preconditioner ${\pcH{\mathbf{P}_\Omega}}$.
\begin{theorem}\label{th:probabilisticbound}
  In the setting of Theorem~\ref{th:deterministicbound} let
  $\Omega$ be drawn as a standard Gaussian $n\times(k+\frakp)$ matrix. Then
  \begin{eqnarray*}
    \lefteqn{\E\,
    \lambda_{\max}\big((I+\Sigma^2)^{\frac12}(I+\Sigma
    \upd{\mathbf{P}_{\Omega}}\Sigma)^{-1}(I+\Sigma^2)^{\frac12}\big)}\\
    &\le &2+
          \sigma_{k+1}^2+\bigg(\sqrt{\frac{k}{\upd{\frakp}-1}}(1+\sigma_{k+1}^2)
          +\frac{e\sqrt{k+\upd{\frakp}}}{\upd{\frakp}}\big(\sum_{i=k+1}^{
          n}(1+\sigma_i^2)\big)^{\frac12}\bigg)^2.
  \end{eqnarray*}
  Furthermore, if $p\ge 4$ then for all $u,t\ge 1$ the probability for
  \begin{eqnarray*}
    \lefteqn{\lambda_{\max}\big((I+\Sigma^2)^{\frac12}(I+\Sigma
    \upd{\mathbf{P}_{\Omega}}\Sigma)^{-1}(I+\Sigma^2)^{\frac12}\big)}\\
    &> &2+
          \sigma_{k+1}^2+\bigg(t\big[\sqrt{\tfrac{3k}{\upd{\frakp}+1}}+u\tfrac{e\sqrt{k+\upd{\frakp}}}{\upd{\frakp}+1}\big](1+\sigma_{k+1}^2)
          +t\tfrac{e\sqrt{k+\upd{\frakp}}}{\upd{\frakp}+1}\big(\sum_{i=k+1}^{
          n}(1+\sigma_i^2)\big)^{\frac12}\bigg)^2    
      \end{eqnarray*}
      is at most $2t^{-\upd{\frakp}}+e^{-u^2/2}$. \\
      The same bounds hold for the condition number
      $\kappa_{\upd{\mathbf{P}_{\Omega}}}$.
\end{theorem}
\begin{proof}
A central and complex step in \cite[proof of Th.\ 10.2]{HalkoMartinssonTropp2011} is to establish the relation
  \begin{displaymath}
 \E \|(I+\Sigma_2^2)^{\frac12}\Omega_2 \Omega_1^\dagger\|\le\sqrt{\frac{k}{\upd{\frakp}-1}}\|(1+\Sigma_2^2)^{\frac12}\|+\frac{e\sqrt{k+\upd{\frakp}}}{\upd{\frakp}}\|(I+\Sigma_2^2)^{\frac12}\|_F
\end{displaymath}
which directly yields the bound on the expected value via
Theorem~\ref{th:deterministicbound}. 

Likewise, in \cite[proof of Th.\ 10.8]{HalkoMartinssonTropp2011} the
authors derive for $\upd{\frakp}\ge 4$ and $u,t\ge 1$ 
\begin{displaymath}
\begin{array}{r@{\;}l}  
  \Prob\bigg\{\|(I+\Sigma_2^2)^{\frac12}\Omega_2
    \Omega_1^\dagger\|>&\|(I+\Sigma_2^2)^{\frac12}\|\sqrt{\frac{3k}{\upd{\frakp}+1}}\cdot
    t+\|(I+\Sigma_2^2)^{\frac12}\|_F\frac{e\sqrt{k+\upd{\frakp}}}{\upd{\frakp}+1}\cdot
                         t\\
  &+\|(I+\Sigma_2^2)^{\frac12}\|\frac{e\sqrt{k+\upd{\frakp}}}{\upd{\frakp}+1}\cdot
    ut\bigg\}\le 2t^{-\upd{\frakp}}+e^{-u^2/2}.
    \end{array}
\end{displaymath}
\end{proof}
Again, the presence of the identity in the deterministic bound of
Theorem~\ref{th:deterministicbound} has a major
impact also in these probabilistic bounds. Indeed, one would hope
that a better deterministic bound helps to prove stronger decay.

Without some a priori knowledge of the singular values of
$D^{-\frac12}V\in\R^{m\times n}$ it is difficult to determine a
suitable number of columns for $\Omega$, \ie a suitable dimension of
the random subspace. For huge $m$ the Johnson-Lindenstrauss result as
presented in \cite{DasguptaGupta2002} suggests that for $k$ at most
$4(\varepsilon^2/2-\varepsilon^3/3)^{-1}\ln m$ a suitably chosen random $\Omega\in\R^{n\times k}$ results
in a distortion of $1\pm\varepsilon$ with sufficiently high
probability. This roughly translates to that
each matrix element of $D^{-\frac12}VV^\top D^{-\frac12}$ and
$D^{-\frac12}V\Omega\Omega^\top V^\top D^{-\frac12}$ differs by at
most this factor. When considering the sizes of $m$ aimed for here ---
the dimension of the design space will be a few thousands to a few
hundred thousands --- this number is still too large for efficient
computations even for a moderate $\varepsilon=0.1$. Indeed, the burden
of forming the preconditioner and of applying it would exceed the gain
by far. \cite{HalkoMartinssonTropp2011} propose an algorithmic variant
for identifying a significant drop in singular values, but this requires
successive matrix vector multiplications and these are quite costly in
practice. In preliminary experiments with a number of tentative
randomized variants, those relating the number of columns to the number
of matrix-vector multiplications of the previous solve seemed
reasonable. It will turn out, however, that even the cost of this is
too high and the gain too small in comparison to the deterministic
approach of the next section. The latter appears to capture the
important directions quite well and offers better possibilities to
exploit structural properties of the data. Due to the rather clear
superiority of the deterministic approach, the numerical experiments of
Section \ref{S:experiments} will only present results for one
particular randomized variant that performed best among the tentative
versions. It attempts to identify the
most relevant subspace by storing and extending the important
directions generated in the previous round. For completeness and
reproducability, its details are given in Alg.~\ref{alg:randprecond},
but in view of the rather disencouraging results we refrain from further
discussions.

\begin{algorithm}[Randomized subspace selection forming $\hat V=V \upd{\hat{P}}$]\label{alg:randprecond}\quad\\
  \textbf{Input:} $V\in\R^{m\times n}$, $D\in\Sym^n_{++}$,
  previous relevant subspace $P_{old}\in\R^{n\times
    \underline{k}}$ (initially $\underline{k}=0$),
  previous number of multiplications $n_{mult}$, previous $\hat k$ of Alg.~\ref{alg:precond}\\
  \textbf{Output:} $\hat V$ (and stores $P_{old}$)
  \begin{compactenum}
  \item If ($\underline{k}=0$) then
    \begin{compactenum}[(a)]
    \item set $k=\min\{n,3+2\hat k,\lceil\sqrt{n_{mult}\frac{n_{mult}+n}4}-\frac{n_{mult}}2\rceil\}$,
    \item generate a standard Gaussian $\Omega\in\R^{n\times k}$ and
      set $\upd{\hat P}\gets\Omega$,
    \end{compactenum}
    else
    \begin{compactenum}[(a)]
    \item set $k_+=\max\{3,\lfloor\frac{\sqrt{n_{mult}}}2\rfloor-\underline{k}\}$,
    \item generate a standard Gaussian $\Omega\in\R^{n\times k_+}$ and set $\upd{\hat{P}}\gets[P_{old},\Omega]$.
    \end{compactenum}
  \item Orthonormalize $\upd{\hat{P}}$, reset $k$ to its number of columns, set $\hat V=V \upd{\hat{P}}$.
  \item Compute eigenvalue decomposition $\hat V^\top D^{-1}\hat V=P\Diag(\hat
    \lambda_1\ge\dots\ge\hat \lambda_k)P^\top$.
  \item Compute threshold $\bar\lambda=\max\{10,e^{\frac1{10}\ln
      \hat\lambda_1-\frac9{10}\ln\hat\lambda_k}\}$ (enforce $\hat\lambda_k>0$).
  \item Set $\underline{k}\gets\min\big\{k,\max\{3,i>3\colon
    \hat\lambda_i>\bar\lambda\}\big\}$ and set $P_{old}\gets
    \upd{\hat{P}}P_{\bullet,[1,\dots,\underline{k}]}$.
    \item Return $\hat V$. 
  \end{compactenum}
\end{algorithm}

\subsection{A Deterministic Subspace Selection Approach}\label{S:detprecond}

In the conic bundle method, $H=D+VV^\top$ of
Theorem~\ref{T:deterministiccond} is of the form described in
\eqref{eq:tildeV}. An inspection of the column blocks of this $V$
suggests to concretize the bound of Theorem~\ref{T:deterministiccond} for
interior point related applications \upd{to
  Theorem~\ref{T:speccondbound} below. In this, $B^\top
  X^{\frac12}$ may be thought
of as an adapted factorization variant of block
$B^\top\frakX_t^{\frac12}(I-\tfrac{\frakX_t^{\frac12}\one_{t}(\frakX_t^{\frac12}\one_{t})^\top}{\zeta^{-1}\sigma+\one_t^\top\frakX_t\one_t})^{\frac12}$
in \eqref{eq:tildeV} with
$X={\frakX_t^{\frac12}(I-\tfrac{\frakX_t^{\frac12}\one_{t}(\frakX_t^{\frac12}\one_{t})^\top}{\zeta^{-1}\sigma+\one_t^\top\frakX_t\one_t})\frakX_t^{\frac12}}$. 
Alternatively, for the full $V$, consider $X$ as consisting of the
three diagonal blocks $I$,
$D_w$ and
${\frakX_t^{\frac12}(I-\tfrac{\frakX_t^{\frac12}\one_{t}(\frakX_t^{\frac12}\one_{t})^\top}{\zeta^{-1}\sigma+\one_t^\top\frakX_t\one_t})\frakX_t^{\frac12}}$
with suitably adapted $B$ and note 
that in the resulting bound each diagonal block of $X$ is added
separately.}
\begin{theorem}\label{T:speccondbound}
  Given $D\in\Sym^{m}_{++}$ and $B\in\R^{n\times m}$, let $X\in\Sym^n_+$ have eigenvalue decomposition $X=[\upd{\bar P},\upd{\underline{P}}]
  \begin{sbmatrix}
    \upd{\bar\Lambda} & 0\\
    0 & \upd{\underline{\Lambda}}
  \end{sbmatrix}[\upd{\bar P},\upd{\underline{P}}]^\top$ with $[\upd{\bar P},\upd{\underline{P}}]^\top[\upd{\bar P},\upd{\underline{P}}]=I_n$ and diagonal $\upd{\bar\Lambda}\in\Sym^k_+$, $\upd{\underline{\Lambda}}\in\Sym^{n-k}_+$. Put $V=B^\top X^{\frac12}$.
  For $H=D+VV^\top$ and preconditioner ${\pcH{\bar P}}=D+V\upd{\bar P} \upd{\bar P}^\top
  V^\top$ the condition number is bounded by 
  \begin{displaymath}
  \kappa_{\upd{\bar P}}\le 1+\sum_{i=1}^{n-k}(\upd{\underline{\Lambda}})_{ii}\|B^\top (\upd{\underline{P}})_{\bullet,i}\|^2_{D^{-1}}\le 1+(n-k)\bar\rho\bar\beta^2.
\end{displaymath}
where $\bar\rho=\max_{i=1,\dots,n-k}(\upd{\underline{\Lambda}})_{ii}$ and $\bar\beta=\max_{i=1,\dots,n}\|B_{i,\bullet}\|_{D^{-1}}$.

\end{theorem}
\begin{proof}
  We show $\lmax({\pcH{\bar P}}^{-\frac12}V\upd{\underline{P}}\,\upd{\underline{P}}^\top V^\top {\pcH{\bar P}}^{-\frac12})\le \sum_{i=1}^{n-k}(\upd{\underline{\Lambda}})_{ii}\|B^\top (\upd{\underline{P}})_{\bullet,i}\|^2_{D^{-1}}$, then the statement follows by Theorem~\ref{T:deterministiccond}. Note that
$V\upd{\underline{P}}\,\upd{\underline{P}}^\top V^\top=B^\top \upd{\underline{P}}\,\upd{\underline{\Lambda}}\,\upd{\underline{P}}^\top B$.
  Furthermore, if $\lambda\ge 0$ satisfies $B^\top B\preceq \lambda D$, it also satisfies
  $B^\top B\preceq \lambda (D+V\upd{\bar P} \upd{\bar P}^\top V^\top)$, therefore
  \begin{eqnarray*}
    \lmax({\pcH{\bar P}}^{-\frac12}V\upd{\underline{P}}\,\upd{\underline{P}}^\top V^\top {\pcH{\bar P}}^{-\frac12})
    &\le& \lmax(D^{-\frac12}B^\top \upd{\underline{P}}\,\upd{\underline{\Lambda}}\,\upd{\underline{P}}^\top BD^{-\frac12})\\
    &\le& \tr(D^{-\frac12}B^\top \upd{\underline{P}}\,\upd{\underline{\Lambda}}\,\upd{\underline{P}}^\top  BD^{-\frac12})\\
    &=&\tr (\upd{\underline{\Lambda}}\,\upd{\underline{P}}^\top B D^{-1}B^\top \upd{\underline{P}})\\
    &=& \sum_{i=1}^{n-k}(\upd{\underline{\Lambda}})_{ii}\|B^\top (\upd{\underline{P}})_{\bullet,i}\|^2_{D^{-1}}\le 1+(n-k)\bar\rho\bar\beta^2.
  \end{eqnarray*}
\end{proof}
Note, the proof weakens ${\pcH{\bar P}}$ to $D$, so the bound cannot be expected to be strong. Yet, it provides a good rule of thumb on which columns of $D^{-\frac12}BX^{\frac12}$ should be included, namely those with large value
$\Lambda_{ii}\|B^\top (\upd{\underline{P}})_{\bullet,i}\|^2_{D^{-1}}$. 

In interior point methods the spectral decomposition and the size of
the eigenvalues of $X$ of Theorem~\ref{T:speccondbound} strongly
depend on the current iteration, in particular on the value of the
barrier parameter $\mu$. Therefore it is worth to set up a new
preconditioner for each new KKT system. In order to do so in a
computationally efficient way, the following dynamic selection
heuristic for $\upd{\bar P}$ with respect to $V$ of \eqref{eq:tildeV}
tries to either pick columns of $V$ directly by including unit vectors
in $\upd{\bar P}$ or to at least take linear combinations of few
columns of $V$ in order to reduce the cost of matrix-vector
multiplications and to preserve potential structural proprieties. So
instead of forming $\upd{\bar P}$, the heuristic builds
$\hat V=V\upd{\bar P}$ directly by appending (linear combinations of
selected) columns of $V$ to $\hat V$. Also, it will often only employ
approximat\upd{ions of the} eigenvalues $\lambda_i$ \upd{together with
  approximations $p_i$ of} the eigenvectors of the $X$ described in
Theorem~\ref{T:speccondbound}. Generally, it will include those in
$\hat V$ for which an estimate of $\lambda_i\|B^\top p_i\|_{D^{-1}}^2$
exceeds a given bound $\underline\rho$. In order to reduce the number
of matrix vector multiplications, $\|B^\top p_i\|^2_{D^{-1}}$ will
only be computed for those $p_i$ with
$\big(\sum_{j=1}^n(p_i)_j^2\|(B^\top)_{\bullet,j}\|_{D^{-1}}\big)^2\ge
\underline\rho$ where the column norms of $B^\top$ are precomputed for
each KKT systems.  The implementation uses $\underline{\rho}=10$. Next
the selections are explained by going through $V$ of \eqref{eq:tildeV}
step by step for each of its three column groups $V_{\upd{\frakH}}$,
$A^\top D_w^{\frac12}$ and
$B^\top\frakX_t^{\frac12}(I-\tfrac{\frakX_t^{\frac12}\one_{t}(\frakX_t^{\frac12}\one_{t})^\top}{\zeta^{-1}\sigma+\one_t^\top\frakX_t\one_t})^{\frac12}$. Concerning
the third group, it will become clear in the discussion of the
semidefinite part that in practice it is advantageous to replace the
square root $\frakX_t^{\frac12}$ in the factorization of $\frakX_t$ by
a more general, possibly nonsymmetric factorization
$\frakX_t=\frakF_t\frakF_t^\top$. The matrix $\frakF_t$ will have the
same block structure and leads to a similar rank one correction by the
transformed trace vector $\frakF_t^\top\one_t$,
\begin{displaymath}
  \frakX_t-\tfrac{\frakX_t\one_{t}(\frakX_t\one_{t})^\top}{\zeta^{-1}\sigma+\one_t^\top\frakX_t\one_t}= \frakF_t(I-\tfrac{\frakF_t^{\top}\one_{t}(\frakF_t^\top\one_{t})^\top}{\zeta^{-1}\sigma+\one_t^\top\frakF_t\frakF_t^\top\one_t})^{\frac12}(I-\tfrac{\frakF_t^{\top}\one_{t}(\frakF_t^\top\one_{t})^\top}{\zeta^{-1}\sigma+\one_t^\top\frakF_t\frakF_t^\top\one_t})^{\frac12}\frakF_t^\top.
\end{displaymath}

\begin{algorithm}[Deterministic column selection heuristic forming $\hat V$]\label{alg:deterministicprecond}\quad\\
  \textbf{Input:} $D_\upd{\mathfrak{H}}$, $V_\upd{\mathfrak{H}}$, $D_y$, $A$, $D_w$, $B$, $\frakX_t$, $\zeta$, $\sigma$ specifying $D$ and $V\in\R^{m\times n}$ of \eqref{eq:tildeV}\\
  \textbf{Output:} $\hat V\in\R^{m\times n'}$ for some $n'\le n$ with $\hat V= V\upd{\bar P}$, $\upd{\bar P}^\top \upd{\bar P}=I_{n'}$.
  \begin{compactenum}
  \item Initialize $\hat V\gets 0\in\R^{m \times 0}$, $\underline{\rho}:=\upd{10}$.
   \item Find $\calJ_{V_{\upd{\mathfrak{H}}}}=\{j\colon
       \|(V_{\upd{\mathfrak{H}}})_{\bullet,j}\|_{D^{-1}}^2\ge\underline{\rho}\}$
      and set $\hat V\gets [\hat V,(V_{\upd{\mathfrak{H}}})_{\bullet,\calJ_{V_{\upd{\mathfrak{H}}}}}]$.
    \item Find $\calJ_A=\{j\colon (D_w)_{jj}\|(A^\top)_{\bullet,j}\|_{D^{-1}}^2\ge\underline{\rho}\}$ and set $\hat V\gets [\hat V,(A^\top D_w)_{\bullet,\calJ_A}]$.
    \item Compute $\frakF_t^{\top}\one_t$,
      $\eta=\zeta^{-1}\sigma+\one_t^\top\frakF_t\frakF_t^\top\one_t$, $B^\top\frakF_t^{\top}\one_t$ and for each
      conic diagonal block of $\frakX_t$ call
      append\_``cone''\_columns$(\hat V)$ with corresponding parameters.
    \item Return $\hat V$. 
  \end{compactenum}
\end{algorithm}

The first group of columns $V_{\upd{\mathfrak{H}}}\in\R^{m\times h_{\upd{\mathfrak{H}}}}$ matches, in the notation of Theorem~\ref{T:speccondbound}, (a subblock of) $B^\top=V_{\upd{\mathfrak{H}}}$ and (a diagonal block) $X=I_{h_{\upd{\mathfrak{H}}}}$. The heuristic appends those columns $j$ to $\hat V$ that satisfy $\|(V_{\upd{\mathfrak{H}}})_{\bullet,j}\|^2_{D^{-1}}\ge \underline\rho$.

For the second group of columns $A^\top D_w^{\frac12}$, Theorem~\ref{T:speccondbound} applies
to $B^\top=A^\top$ and $X=D_w=\Diag(d_1,\dots,d_{h_A})$. Thus, column $j$ is appended to $\hat V$ if $d_j\|(A^\top)_{\bullet,j}\|_{D^{-1}}^2\ge \underline\rho$. 

With the comment above regarding $\frakF_t$, the third column group is
formed by a term $B^\top\frakF_t(I-\tfrac{\frakF_t^{\top}\one_{t}(\frakF_t^{\top}\one_{t})^\top}{\zeta^{-1}\sigma+\one_t^\top\frakX_t\one_t})^{\frac12}$ for each cutting model (we assume just one here). With respect to Theorem~\ref{T:speccondbound}, $B$ is just right
and $X$ is the  positive \mbox{(semi-)}definite matrix $\frakX_t-\tfrac{\frakX_t\one_{t}(\frakX_t\one_{t})^\top}{\zeta^{-1}\sigma+\one_t^\top\frakX_t\one_t}=\frakF_t(I-\tfrac{\frakF_t\one_{t}(\frakF_t\one_{t})^\top}{\zeta^{-1}\sigma+\one_t^\top\frakF_t\frakF_t^\top\one_t})\frakF^\top$. Recall  that $\frakX_t$ is a block diagonal matrix with the structure of the diagonal blocks
governed by the linearization of the perturbed complementarity conditions of the
various cones.
The overarching rank one modification by
$\frakF_t\one_{t}$ couples the blocks within the same cutting model
and reappears in some form in each block together with $\eta=\zeta^{-1}\sigma+\one_t^\top\frakX_t\one_t$. Observe that with $\|\frakF_t^{\top}\one_{t}\|^2=\one_t^\top\frakX_t\one_t$
\begin{displaymath}
  (I-\tfrac{\frakF_t^{\top}\one_{t}(\frakF_t^{\top}\one_{t})^\top}{\eta})^{\frac12}=I-\big(1-\tfrac{\sqrt{\zeta^{-1}\sigma}}{\sqrt{\eta}}\big)\tfrac{\frakF_t^{\top}\one_{t}}{\|\frakF_t^{\top}\one_{t}\|}\tfrac{(\frakF_t^{\top}\one_{t})^\top}{\|\frakF_t^{\top}\one_{t}\|}.
\end{displaymath}
For each column $p$ of $\upd{\bar P}$ computing $B^\top\frakF_t(I-\tfrac{\frakF_t^{\top}\one_{t}(\frakF_t^{\top}\one_{t})^\top}{\eta})^{\frac12}p$ splits into
\begin{displaymath}
 B^\top\frakF_tp- \ip{\frakF_t^{\top}\one_{t}}{p}\tfrac{1}{\one_t^\top\frakX_t\one_{t}} \big(1-\tfrac{\sqrt{\zeta^{-1}\sigma}}{\sqrt{\eta}}\big)B^\top\frakX_t\one_{t}.
\end{displaymath}
Thus, by keeping the support of $p$ restricted to single blocks, the proper
column computations can be kept restricted to the respective
block. This also holds for the coefficient $\ip{\frakF_t^{\top}\one_{t}}{p}$. The overarching vector $B^\top\frakX_t\one_{t}$
needs to be evaluated only once and can be added to the columns afterwards. The latter step  only requires the respective coefficients but not the vectors of $\upd{\bar P}$.
This allows to speed up the process of forming $\hat V$
considerably. Therefore, when forming the conceptional $\upd{\bar P}$ in the
heuristic, the influence of $\frakF_t^{\top}\one_{t}$ on eigenvalues
and eigenvectors of the blocks will mostly be considered as restricted
to each single block. Next the actual selection procedure is described
for $\frakX_t$ blocks corresponding to cones $\R^h_+$
(Alg.~\ref{alg:nncselect}) and $\Sym^h_+$ (Alg.~\ref{alg:pscselect})
with Nesterov-Todd-scaling \cite{NesterovTodd98,ToddTohTuetuencue98}.

\begin{algorithm}[append\_$\R^h_+$\_columns$(\hat V)$]\label{alg:nncselect}\quad\\
  \textbf{Input:} column indices $J\in\N^h$ and $x\circ z^ {-1}$
  of this block in $\frakX_t$, $B^\top_{\bullet,J}$,  $\one_t^\top\frakX_t\one_t$, $B^\top\frakX_t\one_t$,
  $\eta=\zeta^{-1}\sigma+\one_t^\top\frakX_t\one_t$, $D$, threshold $\underline{\rho}$. \\
  \textbf{Output:} updated $\hat V$.
  \begin{compactenum}
    \item For each $i=1,\dots,h$ with $(\frac{x_i}{z_i}-\frac1\eta
      \frac{x_i^2}{z_i^2})\|(B^\top)_{\bullet,J(i)}\|_{D^{-1}}^2\ge\underline{\rho}$ 
      set
      $$\alpha\gets\sqrt{\tfrac{x_i}{z_i}}\tfrac{1}{\one_t^\top\frakX_t\one_t} \big(1-\tfrac{\sqrt{\zeta^{-1}\sigma}}{\sqrt{\eta}}\big),$$
      $$\hat b_i=\sqrt{\tfrac{x_i}{z_i}}(B^\top)_{\bullet,J(i)}- \alpha B^\top\frakX_t\one_t,$$
      and if $\|\hat b_i\|_{D^{-1}}^2>\underline\rho$ set 
      $\hat V\gets [\hat
      V,\hat b_i].$
    \end{compactenum}
\end{algorithm}
For Alg.~\ref{alg:nncselect} consider, within the cone specified by
$t$, a block with indices $J\in\N^h$ representing a nonnegative cone
$\R^h_+$ with primal dual pair $(x,z)$. The corresponding ``diagonal
block'' in $\frakX_t$ is of the form $\Diag(x\circ z^{-1})$ and for $\frakF_t$ it is $\Diag(x\circ z^{-1})^{\frac12}$. The
relevant part of the trace vector $\frakX_t\one_t$ reads
$\Diag(x\circ z^{-1})\one=x\circ z^{-1}$. Considering the influence of
the trace vector as restricted to this block alone gives
$\Diag(x\circ z^{-1})-\frac{1}{\eta}(x\circ z^{-1})(x\circ
z^{-1})^\top$ with the correct overarching
$\eta=\zeta^{-1}\sigma+\one_t^\top\frakX_t\one_t$. The eigenvectors to
large eigenvalues of this matrix have their most important coordinates
associated with the largest diagonal entries. The heuristic appends
the columns
\mbox{$B^\top\frakX_t^{\frac12}(I-\tfrac{\frakX_t^{\frac12}\one_{t}(\frakX_t^{\frac12}\one_{t})^\top}{\eta})^{\frac12}e_{J(i)}$}
to $\hat V$ for those $e_{J(i)}$ with
$(\frac{x_i}{z_i}-\frac1\eta
\frac{x_i^2}{z_i^2})\|(B^\top)_{\bullet,J(i)}\|_{D^{-1}}^2\ge\underline{\rho}$.

Note that in interior point methods $x_iz_i\approx\mu$ for barrier
parameter \mbox{$\mu\searrow 0$} and $x_i\to x_i^{opt}$, $z_i\to z_i^{opt}$.
Due to $\eta\ge\frac{x_i}{z_i}$ with $\eta$ mostly much larger,
the estimated value roughly behaves like
$\frac{x_i^2}\mu\|(B^\top)_{\bullet,J(i)}\|_{D^{-1}}^2$ and, indeed, by
experience it seems that columns are almost exclusively included only
for active $x_i^{opt}>0$ and only as $\mu$ gets small enough. When
computing high precision solutions with small $\mu$, the rank of
the preconditioner can thus be expected to match the number of
active subgradients in the cutting model. Theorem~\ref{T:speccondbound}
suggests that in iterative methods these columns have to be
included in some form in order to obtain reliable convergence
behavior.

Alg.~\ref{alg:pscselect} below deals with a positive semidefinite cone
$\Sym^h_+$ with Nesterov-Todd-scaling. For the current purposes it
suffices to know that the diagonal block of $\frakX_t$ indexed by
appropriate $J\in\N^{h+1\choose 2}$ is of the form
$W\skron W$ for a positive definite $W\in\Sym^h_{++}$; see
\cite{ToddTohTuetuencue98} for its efficient computation and for an
appendix of convenient rules for computing with symmetric Kronecker
products. The next result derives the eigenvectors and eigenvalues
when considering the rank one correction restricted to this block.
\begin{lemma}\label{L:TTTPSC}
  Let $W=P_W\Lambda_W P_W^\top$ with $\Lambda_W=\Diag(\lambda^W_1\ge \dots\ge\lambda^W_h>0)$ and $P_W^\top P_W=I_h$, $P_W=[w_1,\dots,w_h]$. Furthermore let $U=\Lambda_W^2-\frac1\eta(\Lambda_W^2\one)(\Lambda_W^2\one)^\top$ have eigenvalue decomposition
  $U=P_U\Lambda_UP_U^\top$ with $P_U^\top P_U=I_h$.  
  The eigenvalues of $W\skron W-\frac1\eta\big((W\skron W)\svec(I_n)\big)\big((W\skron W)\svec(I_h)\big)^ \top$ are $\lambda^U_i=(\Lambda_U)_{ii}$ with eigenvectors $\sum_{j=1}^h (P_U)_{ji}\svec(w_jw_j^\top)$ for $i=1,\dots,h$ and $\lambda_i^W\lambda^W_j$ with eigenvectors $\frac1{\sqrt{2}}\svec(w_iw_j^\top+w_jw_j^\top)$ for $1\le i< j\le h$. 
\end{lemma}
\begin{proof}
  By \cite{AlizadehHaeberlyOverton98} the eigenvalues of $(W\skron W)$ are $\lambda^W_i\lambda^W_j$
  for $1\le i\le j\le h$ with orthonormal eigenvectors
  \begin{equation}\label{eq:wij}
    \begin{array}{r@{\;\,}l@{\quad\text{for }}l}
      w_{ii}:=&\svec(w_iw_i^\top)&1\le i=j\le h,\\
      w_{ij}:=&\frac1{\sqrt2}\svec(w_iw_j^\top+w_jw_i^\top)& 1\le i<j\le h.
    \end{array}
  \end{equation}
  To see this \eg for $i<j$ observe $w_{ij}^\top w_{ij}=\frac12[2\ip{w_iw_j^\top}{w_iw_j^\top}+2\ip{w_iw_j^\top}{w_jw_i^\top}]$ and $(W\skron W)w_{ij}=\frac1{\sqrt2}\svec(W w_iw_j^\top W+ W w_jw_i^\top W)= \lambda^W_i\lambda^W_j w_{ij}$.

  From $(W\skron W)\svec(I_h)=\svec(W^2)=\sum_{i=1}^h(\lambda^W_i)^2w_{ii}$ one obtains
  \begin{displaymath}
    \begin{array}{r@{\;}l@{\quad\text{for }}l}
      \svec(W^2)^\top w_{ii}=&(\lambda_i^W)^2&i=1,\dots,h,\\
      \svec(W^2)^\top w_{ij}=&0&1\le i<j\le h.
    \end{array}
\end{displaymath}
The eigenvector-sorting $P_{\skron}=[w_{11},w_{22},\dots,w_{hh},w_{12},w_{13},\dots,w_{h-1,h}]$ gives
\begin{displaymath}
  W\skron W-\frac1\eta\svec(W^2)\svec(W^2)^\top=P_{\skron}
  \begin{sbmatrix}
    U & 0 &\cdots & 0 \\
    0 & \lambda^W_1\lambda^W_2 & \ddots &\vdots\\
    \vdots &\ddots &\ddots& 0\\
    0 &     \cdots & 0 & \lambda^W_{h-1}\lambda^W_h
  \end{sbmatrix}P_{\skron}^\top
\end{displaymath}
The result now follows by direct computation. 
\end{proof}
For semidefinite blocks, numerical experience indicates that it is
indeed worth to determine the eigenvalue decomposition of $U$ as in
Lemma \ref{L:TTTPSC}.  Finding the eigenvalues and eigenvectors
roughly requires the same amount of work as forming $W$ and is of no
concern.  With $J\in \N^{h+1\choose 2}$ denoting the column indices of
this block within $B^\top$, columns to corresponding eigenvectors are
computed by
$(\Lambda_U^{\frac12})_{ii}\cdot (B^\top)_{\bullet,J}\sum_{j=1}^h
(P_U)_{ji}w_{jj}$ or
$\sqrt{\lambda^W_i\lambda^W_j}(B^\top)_{\bullet,J}w_{ij}$. This
involves linear combinations of ${h+1\choose 2}$ columns and is
computationally expensive if the order $h$ of $W$ gets large. Indeed, when testing
all columns by their correct norms $\|(B^\top)_{\bullet,J}w_{ij}\|_{D^{-1}}^2$, too much time is spent in forming the preconditioner. Therefore  the
heuristic Alg.~\ref{alg:pscselect} first selects candidate eigenvectors to use for
$\upd{\bar P}$ via the rough estimate $\sum_{\hati=1}^{h\choose 2}(w_{ij})_\hati^2\|(B^\top)_{\bullet,J(\hati)}\|_{D^{-1}}^2=\|w_{ij}\|_{\Diag(BD^{-1}B^\top)_J}^2$. 
For the selected eigenvectors it then
computes the precise values after the following transformation that is only seemingly involved.

In order to also account for the possibly overarching contribution of $\frakF_t\one_{t}$ it is advantageous to find a representation equivalent to $B^\top X^{\frac12} \upd{\bar P}$ with orthonormal
columns in $\upd{\bar P}$ as in Theorem~\ref{T:speccondbound} for a suitable
factorization of $X$ other than its square root. For this, let $V_W=P_W\Lambda_W^{\frac12}$, then $W\skron W=(V_W\skron V_W)(V_W^\top\skron V_W^\top)$. Because $(V_W^\top\skron V_W^\top)\svec I=\svec(V_W^\top V_W)=\svec \Lambda_W$ and $(V_W\skron V_W)=(P_W\skron P_W)(\Lambda_W^{\frac12}\skron\Lambda_W^{\frac12})$, the notation of Lemma~\ref{L:TTTPSC} and its proof allows to rephrase the semidefinite block of $\frakX_t-\tfrac{\frakX_t\one_{t}(\frakX_t\one_{t})^\top}{\eta}$ as
\begin{eqnarray*}
  \lefteqn{W\skron W-\frac{(W\skron W)\svec(I)\svec(I)^\top(W\skron W)}{\eta}=}\\
  &=&(V_W\skron V_W)(I-\frac{\svec(\Lambda_W)\svec(\Lambda_W)^\top}{\eta})(V_W\skron V_W)^\top\\
  &=&(P_W\skron P_W)(\Lambda_W^{\frac12}\skron\Lambda_W^{\frac12})(I-\frac{\svec(\Lambda_W)\svec(\Lambda_W)^\top}{\eta})(\Lambda_W^{\frac12}\skron\Lambda_W^{\frac12})(P_W\skron P_W)^\top\\
  &=&P_{\skron}\begin{sbmatrix}
    U & 0 &\cdots & 0 \\
    0 & \lambda^W_1\lambda^W_2 & \ddots &\vdots\\
    \vdots &\ddots &\ddots& 0\\
    0 &     \cdots & 0 & \lambda^W_{h-1}\lambda^W_h
  \end{sbmatrix}P_{\skron}^\top\\
  &=&(P_W\skron P_W)FF^\top(P_W\skron P_W)^\top,      
\end{eqnarray*}
where
\begin{displaymath}
F=(\Lambda_W^{\frac12}\skron\Lambda_W^{\frac12})(I-\frac{\svec(\Lambda_W)\svec(\Lambda_W)^\top}{\eta})^{\frac12}.
\end{displaymath}
This suggests to put $V=B^\top P_{\skron} F$ and to derive the columns
corresponding to $\upd{\bar P}$ via the singular value decomposition of $F=Q_F\Sigma_FP_F^\top$. Lemma~\ref{L:TTTPSC} provides the squared singular values $\Sigma_F^2$ and the eigenvectors give the left-singular vectors in $Q_F$. For $e_{ij}:=\frac1{\sqrt2}\svec(e_ie_j^\top+e_je_i^\top)$ there holds $\svec(\lambda_W)^\top e_{ij}=0$, so the right-singular vectors corresponding to $\sqrt{\lambda_i\lambda_j}$ read $(P_F)_{\bullet,ij}=e_{ij}$. The remaining right-singular vectors of $P_F$ may be computed via $P_F=F^\top Q_F\Sigma_F^{-1}$. In this it is sufficient and convenient to consider only the $U$ block, \ie the support restricted to the $ii$-coordinates. Denote the columns of $P_U=[u_1,\dots,u_h]$ in Lemma~\ref{L:TTTPSC}
by $u_j$ for $j=1,\dots,h$, then the corresponding right-singular vectors $u^F_j\in\R^h$  read for $\Lambda_W=\Diag(\lambda^W)$ and $\Lambda_U=\Diag(\lambda^U_1,\dots,\lambda^U_h)$
\begin{eqnarray}
  u^F_j&=&(I-\tfrac{\lambda^W(\lambda^W)^\top}{\eta})^{\frac12}\Lambda_W\cdot u_j\cdot \tfrac1{\sqrt{\lambda^U_j}}\label{eq:ujF}\\
       &=& \tfrac1{\sqrt{\lambda^U_j}}\bigg(\Lambda_W u_j-\frac{\sqrt{\eta}-\sqrt{\eta-\|\lambda^W\|^2}}{\sqrt{\eta}\|\lambda^W\|^2}\ip{\lambda^W}{\Lambda_Wu_j}\lambda^W\bigg).\nonumber
\end{eqnarray}
By expanding the $U$ block to the correct positions, the right-singular vector to singular value $\sqrt{\lambda^U_j}$ is $(P_F)_{\bullet,{jj}}=\svec(\Diag(u_j^F))$  for $j=1,\dots,h$.

With these preparations the selected semidefinite columns are appended to $\hat V$ as follows. First note that the semidefinite block with
coordinates $J$ of the factor $\frakF_t$ is $(V_W \skron V_W)$, which is nonsymmetric in general. The transformed trace vector $\frakF_t^\top\one_t$ reads $(\frakF_t^\top\one_t)_J=(V_W^\top\skron V_W^\top)\svec I=\svec(\Lambda_W)$.
If column $p_{ij}^F$ of $P^F$ with $1\le i\le j\le h$ is selected for $\upd{\bar P}$ by the heuristic, the column to be appended to $\hat V$ reads
\begin{displaymath}
 (B^\top)_{\bullet,J} (V_W \skron V_W)p_{ij}^F- \ip{\svec{\Lambda_W}}{p_{ij}^F}\tfrac{1}{\one_t^\top\frakX_t\one_{t}} \big(1-\tfrac{\sqrt{\zeta^{-1}\sigma}}{\sqrt{\eta}}\big)B^\top\frakX_t\one_{t}.
\end{displaymath}
If the selected indices satisfy $i<j$, the vector $p_{ij}^F$ is just
$e_{ij}=\frac1{\sqrt2}\svec(e_ie_j^\top+e_je_i^\top)$. By
$(V_W \skron V_W)e_{ij}=\frac{\sqrt{\lambda_i^W\lambda_j^W}}{\sqrt2}\svec(w_iw_j^\top+w_jw_i^\top)=\sqrt{\lambda_i^W\lambda_j^W}w_{ij}$ and $\ip{\svec{\Lambda_W}}{e_{ij}}=0$ the column computation simplifies to
\begin{displaymath}
  \sqrt{\lambda_i^W\lambda_j^W}(B^\top)_{\bullet,J}w_{ij}=\sqrt{2\lambda_i^W\lambda_j^W}\big[w_i^\top{\svec}^{-1}([B^\top]_{k,J})w_j\big]_{k=1,\dots,n}.
\end{displaymath}
Typically, several mixed eigenvectors $w_{ij}$ have the same index $i$
corresponding to a large value $\lambda_i^W$, so it
quickly pays off to precompute $w_i^\top\svec^{-1}([B^\top]_{k,J})$
and to use these $h$-vectors for each $w_j$. For ease of presentation
this implementational detail is not described in
Alg.~\ref{alg:pscselect}. Also, this is
not helpful for the non-mixed vectors
$p_{jj}^F=\svec(\Diag(u_j^F))$, because
\begin{displaymath}
  (V_W\skron V_W)p_{jj}^F=(V_W\skron V_W)\sum_{i=1}^h (u_j^F)_i\svec{e_ie_i^\top}=\sum_{i=1}^h (u_j^F)_i\lambda^W_iw_{ii} 
\end{displaymath}
consists of a linear combination over all $w_{ii}$.  Fortunately\upd{,
  throughout our experiments,} only few of the non-mixed vectors are
among those selected for preconditioning. \upd{A possible explanation
  for this might be that with respect to the selected bundle subspace
  the large non-mixed terms reflect the rank of the currently strongly
  active eigenspace while large mixed terms reflect its ongoing
  interaction with the eigenspace of moderately active or inactive
  eigenvalues.} The transformed trace vector coefficient for
$p_{jj}^F$ evaluates to
$\ip{\Lambda_W}{\Diag(u_j^F)}=\ip{\lambda^W}{u_j^F}$.  With this, the
algorithm for appending semidefinite columns reads as follows.
\begin{algorithm}[append\_$\Sym^h_+$\_columns$(\hat V)$]\label{alg:pscselect}\quad\\
  \textbf{Input:}  column indices $J\in\N^{h+1\choose 2}$
  and Nesterov-Todd scaling matrix $W\succ 0$ of this block in
  $\frakX_t$, $B^\top_{\bullet,J}$,  $\one_t\frakX_t\one_t$, $B^\top\frakX_t\one_t$,
  $\eta=\zeta^{-1}\sigma+\one_t^\top\frakX_t\one_t$, $D$, threshold $\underline{\rho}$ \\
  \textbf{Output:} updated $\hat V$.
  \begin{compactenum}
  \item Compute norms $\|(B^\top)_{\bullet,J(i)}\|_{D^{-1}}$, set $\hat\rho=\underline{\rho}/\max_{i=1,\dots,{h+1\choose 2}}\|(B^\top)_{\bullet,J(i)}\|_{D^{-1}}^2$,\\
    compute eigenvalue decomposition $W=P_W\Lambda_W
    P_W^\top$, $\Lambda_W=\Diag(\lambda^W)$ with $\lambda^W_{1}\ge \dots\ge\lambda^W_{h}$,
    let $w_{ij}$ be defined by \eqref{eq:wij}.
  \item If $(\lambda^W_{1})^2<\hat\rho$ do nothing and return $\hat V$.  
  \item Compute
    $U=\Lambda_W^2-\frac1\eta(\lambda^W)(\lambda^W)^\top$,
    eigenvalue decomposition $U=P_U\Lambda_U P_U^\top$,\\
    with $P=[u_1,\dots,u_h]$.
  \item For each $\hati=1,\dots,h$ with $(\Lambda_U)_{\hati\hati}\ge\hat\rho$ do:\\
    Compute $\hat w_{\hati\hati}=\sum_{i=1}^h (u_\hati)_iw_{ii}$ ($\in\R^{h+1\choose 2}$).\\
    If $(\Lambda_U)_{\hati\hati}\sum_{j=1}^{h+1\choose 2}(
    \hat w_{\hati\hati})_j^2\|(B^\top)_{\bullet,J(j)}\|_{D^{-1}}^2\ge\underline{\rho}$ then:
    \begin{compactenum}[(a)]
    \item Compute $u_{\hati}^F$ according to \eqref{eq:ujF} and set 
      \begin{displaymath}
        \alpha\gets\ip{\lambda^W}{u_\hati^F} \tfrac{1}{\one_t^\top\frakX_t\one_t} \big(1-\tfrac{\sqrt{\zeta^{-1}\sigma}}{\sqrt{\eta}}\big),
      \end{displaymath}
      \begin{displaymath}
        \hat b_{\hati\hati}=(B^\top)_{\bullet,J}\sum_{i=0}^h (u_\hati^F)_i\lambda^W_iw_{ii}-\alpha B^\top \frakX_t\one_t.
      \end{displaymath}
    \item If $\|\hat b_{\hati\hati}\|_{D^{-1}}^2\ge\underline\rho$ set $\hat V\gets [\hat
      V,\hat b_{\hati\hati}]$.
    \end{compactenum}
  \item For each $1\le\hati<\hatj\le h$ with $\lambda^W_{\hati}\lambda^W_{\hatj}>\hat\rho$ do:\\
    If $\sqrt{\lambda^W_{\hati}\lambda^W_{\hatj}}\sum_{j=1}^{h+1\choose 2}(
    \hat w_{\hati\hatj})_j^2\|(B^\top)_{\bullet,J(j)}\|_{D^{-1}}^2\ge\underline{\rho}$ set \\    
    $$\hat b_{\hati\hatj}=\sqrt{\lambda^W_{\hati}\lambda^W_{\hatj}}(B^\top)_{\bullet,J}w_{\hati\hatj}$$
    and if $\|\hat b_{\hati\hatj}\|_{D^{-1}}^2\ge\underline{\rho}$ set $\hat V\gets [\hat
      V,\hat b_{\hati\hatj}]$.
      \item Return $\hat V$.
    \end{compactenum}
\end{algorithm}
As for the
linear case it can be argued that for small barrier parameter $\mu$
the number of selected columns corresponds at least to the order of the
active submatrix in the cutting model. Thus if $\hat h\le h$ eigenvalues of $X\in\Sym^h_+$
converge to positive values in the optimum, the heuristic will end up
with selecting at least ${\hat h+1\choose 2}$ columns once $\mu$ gets small.

For second order cones $\soc{h}$ the structural properties of the arrow operator and
the Nesterov-Todd-direction allow to restrict considerations to just
two directions per cone for preconditioning, but as the computational
experiments do not involve second order cones this will not be discussed here.

\section{Numerical Experiments}\label{S:experiments}

The purpose of the numerical experiments is to explore and compare the
behavior and performance of the pure and preconditioned iterative variants to
the original direct solver on KKT instances that arise in the course
of solving large scale instances by the conic bundle method.

It has to be emphasized that the experiments are by no means designed
and intended to investigate the efficiency of the conic bundle method
with internal iterative solver. Indeed, many aspects of the
ConicBundle code \cite{ConicBundle2021} such as the cutting model
selection routines, the path following predictor-corrector approach
and the internal termination criteria have been tuned to work
reasonably well with the direct solver. As the theory suggests and the
results support, the performance of iterative methods depends more on
the size of the active set than on the size of the model. Thus
somewhat larger models might be better in connection with iterative
solvers. Also, the predictor-corrector approach is particularly
efficient if setting up the KKT system is expensive. For iterative
methods with deterministic preconditioning this hinges on the cost of
forming the preconditioner which gets expensive once the barrier
parameter gets small. Furthermore iterative methods might actually
profit from staying in a rather narrow neighborhood of the central
path. Therefore many implementational decisions need to be reevaluated
for iterative solvers. This is out of scope for this paper. Hence, the
experiments only aim to highlight the relative performance of the
solvers on sequences of KKT systems that currently arise in
ConicBundle. \upd{For the sole purpose of demonstrating the relevance
  of this KKT system based analysis, Section \ref{S:bundleperf} will
  present a comparison on the performance of ConicBundle when
  employing the KKT solver variants without any further adaptations of
  parameters.}

The \upd{KKT system oriented} experiments will report on the performance for three different
instances: the first, denoted by MC, is a classical semidefinite
relaxation of Max-Cut on a graph with 20000 nodes as described in
\cite{GoemansWilliamson95,HelmbergRendl2000}, the second, BIS, is a
semidefinite Minimum-Bisection relaxation improved by dynamic
separation of odd cycle cutting planes on the support of the Boeing
instance KKT\_traj33 giving a graph on 20006 nodes explained in
\cite{Helmberg2004}, and the third, MMBIS, refers to a
min-max-bisection problem problem shifting the edge weights so as to
minimize a restricted maximum cut on a graph of 12600 nodes. All three
have a single semidefinite cutting model which consists of a
semidefinite cone with up to one nonnegative variable, so the model
cone $\calS^t_+$ of \eqref{eq:St} typically has $t=(1,[],[h])$ for
some $h\in\N$.  In the Max-Cut instance the design variables are
unconstrained, in the Bisection instance the design variables
corresponding to the cutting planes are sign constrained ($D_y$ is
needed) and in the min-max-bisection problem some design variables
have bounds and there are linear equality and inequality constraints
($D_y$, $D_w$ and $A$ appear). Throughout, the proximal term is a
multiple of the identity for a dynamic weight, \ie $\upd{\mathfrak{H}}_k=u_kI$ with
$u_k>0$ controlled as in \cite{HelmbergKiwiel2002}.

In each case ConicBundle is run with default settings for the
internal constrained QP solver with direct
KKT solver for the bundle subproblems. Whenever a new KKT system arises, it is
solved consecutively but independently on the same machine by
\begin{compactitem}
\item (DS) the original direct solver,
\item (IT) MINRES without preconditioning (the implementation follows \cite{ElmanSilvesterWathen2006}),
\item (RP) MINRES with randomized preconditioning
  (Alg.~\ref{alg:precond} with  Alg.~\ref{alg:randprecond}),
\item (DP) MINRES with deterministic preconditioning (Alg.~\ref{alg:precond} with
Alg.~\ref{alg:deterministicprecond}).
\end{compactitem}
Only the results of the direct solver are then used to continue the
algorithm.  Note, for nonsmooth optimization problems tiny deviations
in the solution of the subproblem may lead to huge differences in the
subsequent path of the algorithm. Therefore running the bundle method
with different solvers would quickly lead to incomparable KKT
systems. \upd{That the chosen approach does not impair the validity of
  the conclusions regarding the performance of the solvers within the
  bundle method will be demonstrated in Section~\ref{S:bundleperf}.}

The details of the direct solver DS are of little relevance at this
point. Suffice it to say that its main work consists in Schur
complementing the $\upd{\mathfrak{H}}$ and $\zeta^{-1}\sigma$ blocks of the KKT system \eqref{eq:fullKKT} into the joined $\Diag(D_w^{-1},\frakX_t^{-1})$
block and factorizing this. In the Max-Cut setting (no $D_y$), the
$\upd{\mathfrak{H}}$ block is constant throughout each bundle subproblem. In this case the
Schur complement is precomputed once for each bundle subproblem ---
thus for several KKT systems --- and this makes this approach extremely
efficient as long as the order $h$ of the semidefinite model is
small. Precomputation is no longer possible if $D_y$ is needed which
is the case in the two other instances.  Finally, if $A$ is also
present, the system to be factorized in every iteration gets
significantly larger. These differences motivated the choice of the
instances and explain part of the strong differences in the performance
of the solvers.

For Max-Cut and Bisection the iterative solver could exploit the positive
definiteness of the system by employing conjugate gradients instead of
MINRES. The min-max-bisection problem comprises equality constraints
in $A$, so the system is no longer positive definite and conjugate
gradients are not applicable. Employing MINRES for all three
facilitates the comparison, in particular as MINRES seemed to perform
numerically better on the other instances as well. MINRES computes the residual norm with respect to the inverse of the preconditioner and
the implementation uses this norm for termination. To safeguard against
effects due to the changes in this norm, the relative precision
requirement \upd{$\min\{10^{-6},10^{-2}\mu\}$ of ConicBundle}
is multiplied, in the notation of Alg.~\ref{alg:precond}, by the
factor $\big(\sqrt[m]{\prod_{i=1}^{\hat k}(1+\hat\lambda_i)^{-1}}\cdot\min_i (D^{-1})_i\big)^{\frac12}$.

The results on the three instances will be presented in eight plots per
instance. The first four compare all four solvers, the last four plots
are devoted to information that is only relevant for iterative
solvers, so DS will not appear in these. 
\begin{compactenum}
\item Plot ``time per subproblem (seconds)'' gives for each of the
  four methods a box plot on the seconds (in logarithmic scale)
  required to solve the subproblems.  For each subproblem this is the
  sum of the time required for initializing/forming and solving all
  KKT systems of this subproblem.  This is needed, because in the case
  of Max-Cut instance MC, the direct solver DS forms the Schur complement of the
  $\upd{\mathfrak{H}}$-block only once per subproblem and this is also accounted for
  here.
\item Plot ``subproblem time (seconds) per iteration'' displays
  the same cumulative time per subproblem in seconds (in logarithmic
  scale) for each successive iteration so that the development
  in solution time is aligned to the progress of the bundle method. 
\item Plot ``time per subproblem vs.\ bundle size'' serves to
  highlight the dependence of the solution time on the size of the
  cutting model (number of rows of $B$). For this the subproblems are
  grouped in the bundle size ranges $(0,50]$, $(50,500]$,
  $(500,1500]$, $(1500,\infty]$. Instead of \upd{infinity the actual
    observed maximum is listed} in the bottom line of the plot.
\item Plot ``time per subproblem vs.\ last $\mu$'' illustrates the
  dependence of the solution time on the last barrier parameter $\mu$
  for which the subproblem has to be solved. Roughly this corresponds
  to the precision required for the subproblem. \upd{Because of the
  comparatively small number of subproblems and the strongly differing
  ranges of last $\mu$ values} results are presented
  for a subdivision of the subproblems into four groups of equal
  cardinality (up to integer division) sorted according to the $\mu$
  value of their respective last KKT system. The minimum $\mu$
  of each group is given in the bottom line of the plot.
\item Plot ``time per KKT system (seconds)'' compares exclusively the
  iterative methods on the KKT systems belonging to the four different
  ranges of the barrier parameter $\mu$ as collected over all
  subproblems. The first three box plots give
  the box plot statistics on the seconds (in logarithmic scale)
  spent in solving KKT systems for barrier parameter values
  $\mu\ge 100$, the next three for $100>\mu\ge 1$, etc. Note, DS
  would require the same time for all KKT systems of the same
  subproblem, because its solution time does not depend on $\mu$
  or the associated required relative precision \upd{described above}.
\item Plot ``matrix vector multiplications per KKT system'' shows box plots on
  the number of matrix-vector multiplications (in logarithmic scale)
  needed by MINRES, again subdivided into the same ranges of barrier
  parameter values.
\item Plot ``KKT system condition number estimate'' presents the box plot
  statistics of an estimate of the condition number (in logarithmic
  scale) for the same ranges of the barrier parameter. The estimate
  is obtained by a limited number of Lanczos iterations on the respective
  (non-)pre\-con\-di\-tioned system of the $\upd{\mathfrak{H}}$ block; a possibly remaining
  equality part of $A$ is ignored in this. Computation times for the
  condition number are not included in the time measurement listed above.  
\item Plot ``preconditioning columns per KKT system'' gives the box
  plot statistics of the number of columns $\hat k$ in
  Alg.~\ref{alg:precond} for RP and DP for the usual ranges of the
  barrier parameter.
\end{compactenum}
In all box plots, the width of the boxes indicates the relative size of
the number of instances in the group, the horizontal lines of the boxes give the
values of the upper quartile, the median and the lower quartile. The
upper whisker shows the largest value below upper
quartile$ + 1.5\cdot$IQR, where IQR$ = $(upper quartile $-$ lower
quartile) is the interquartile range. The lower whisker displays the
smallest value above lower quartile$ - 1.5\cdot$IQR. The stars show
maximum and minimum value.

Computation times refer to a virtualized compute server of 40 Intel Xeon Processor (Cascadelake) cores with 600 GB RAM under Ubuntu 18.04. This virtual machine
is hosted on hardware consisting of two processors Intel(R) Xeon(R)
Gold 6240R CPU with 2.40GHz with 24 cores and 768 GB RAM. The code,
however, is purely sequential and does not exploit any parallel
computation possibilities. 

\subsection{Max-Cut (Instance MC, Figure \ref{fig:MC})}

\begin{figure}[htbp]
\centerline{
\begin{tabular}{@{}cc@{}}
\includegraphics[width=.48\textwidth]{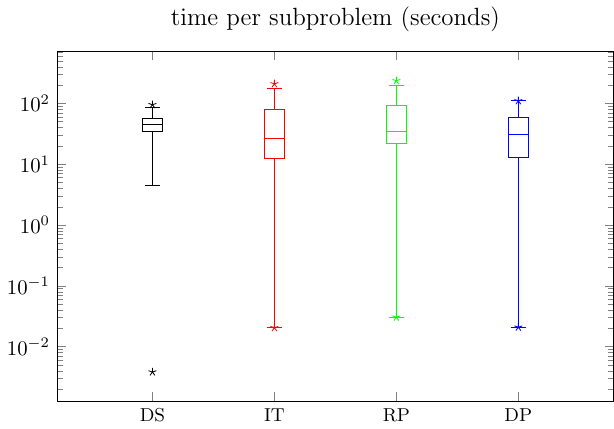}
&
\includegraphics[width=.48\textwidth]{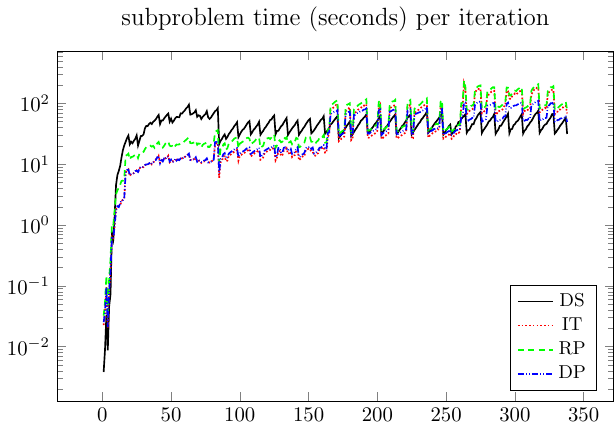}
\\[1ex]
\includegraphics[width=.48\textwidth]{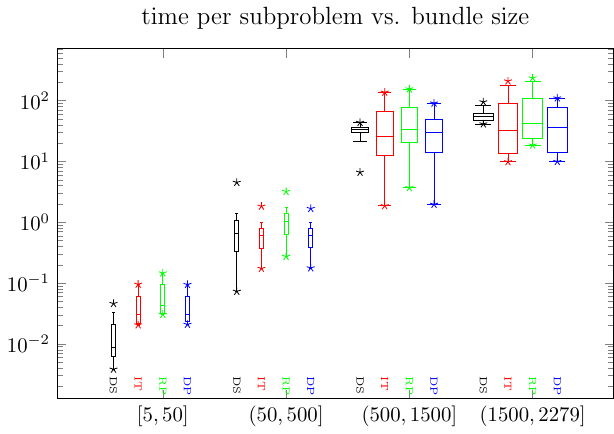}
&
\includegraphics[width=.48\textwidth]{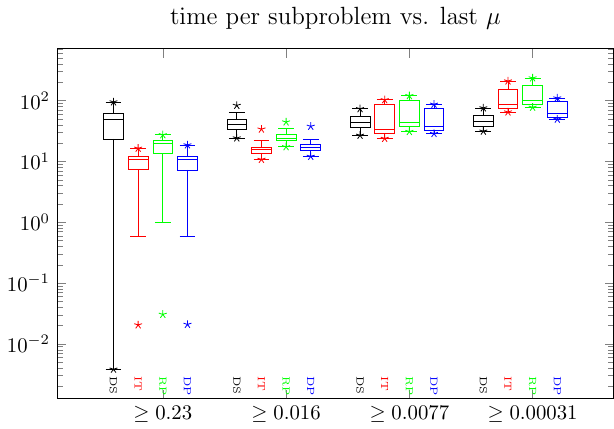}
\\[1ex]
\includegraphics[width=.48\textwidth]{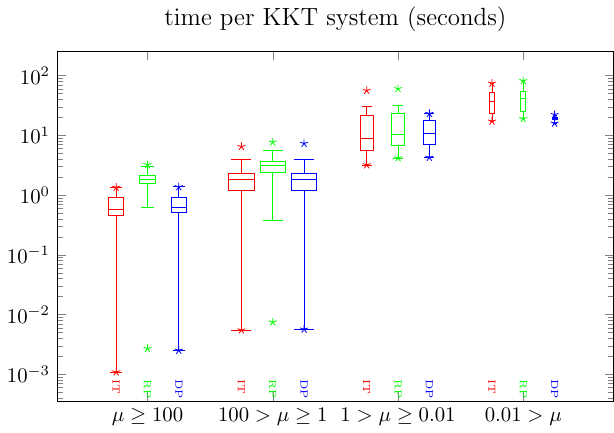}
&
\includegraphics[width=.48\textwidth]{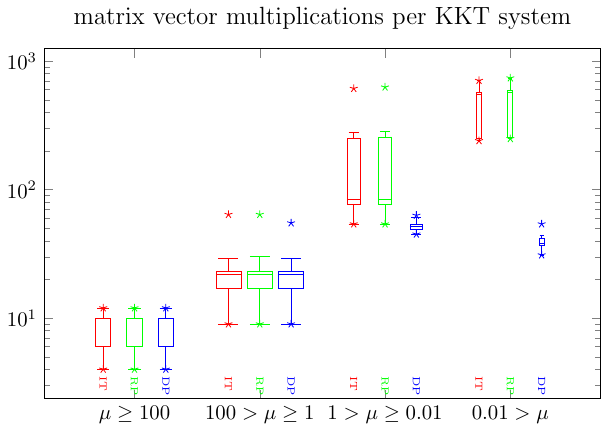}
\\[1ex]
\includegraphics[width=.48\textwidth]{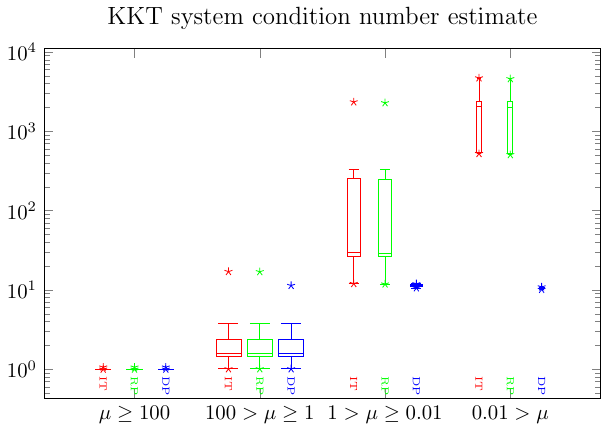}
&
\includegraphics[width=.48\textwidth]{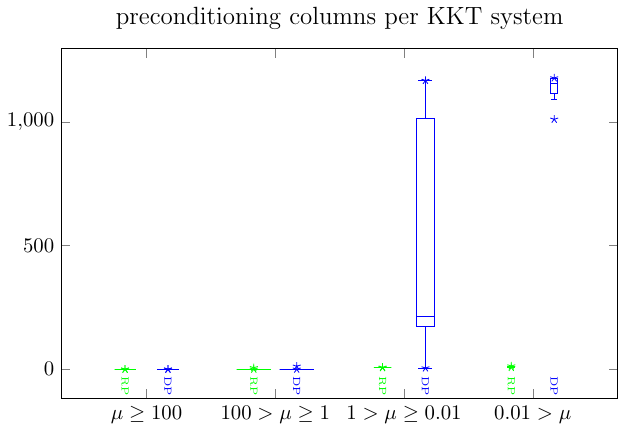}
\end{tabular}
}\vspace*{-.5ex}
\caption{\label{fig:MC} Instance MC, 338 subproblems, 2832 KKT systems. In order to highlight the dependence on the barrier parameter $\mu$ with corresponding precision requirements the results for KKT systems are grouped into $\mu$-value ranges $(\infty,100]$, $(100,1]$, $(1,0.01]$, $(0.01,0)$.}
\end{figure}

The graph was randomly generated (\cite{rudy}, call \mbox{\texttt{rudy
    -rnd\_graph 20000 1 1}} for 20000 nodes, edge density one percent,
seed value 1). The semidefinite relaxation gives rise to an
unconstrained problem with $20000$ variables. Each variable influences
one of the diagonal elements of the Laplace matrix of the graph with
cost one and the task is to minimize the maximum eigenvalue of the
Laplacian times the number of nodes, see \cite{HelmbergRendl2000} for the
general problem description.

For graphs of this type but smaller size like 5000 or 10000 nodes the
direct solver DS still seemed to perform better, so rather large sizes
are needed to see some advantage of iterative methods. Other than that
the relative behavior of the solvers was similar also for the smaller
sizes. The jaggies within subproblem time in the second plot are due
to the reduction of the model to its active part after each descent
step while the model typically increases in size during null steps.
During the very first iterations the bundle is tiny and DS is the best
choice.  Once the bundle size increases sufficiently, the iterative
methods dominate. Over time, as precision requirements get higher and
the choice of the bundle subspace converges, the advantage of
iterative methods decreases. In the final phase of high precision the
direct solver may well be more attractive again.

The plots also show that for this instance (and presumably for most
instances of this random type) the performance of IT (MINRES without
preconditioning) is almost as good as DP (deterministic
preconditioning) while RP (randomized preconditioning) is not
competitive. Note that the condition number does not grow excessively
for IT in this instance. Deterministic preconditioning succeeds
in keeping the condition number almost exactly at the intended
value 10. For smaller values of $\mu$, so for higher precision
requirements, DP requires distinctly fewer matrix-vector
multiplications, but it then also selects a large number of columns. In
comparison to no preconditioning DP helps to improve stability but
does not lead to significantly better computation times except maybe
for the very last phase of the algorithm with high precision
requirements.

\subsection{Minimum Bisection (Instance BIS, Figure \ref{fig:BIS})}

\begin{figure}[htbp]
\centerline{
\begin{tabular}{@{}cc@{}}
\includegraphics[width=.48\textwidth]{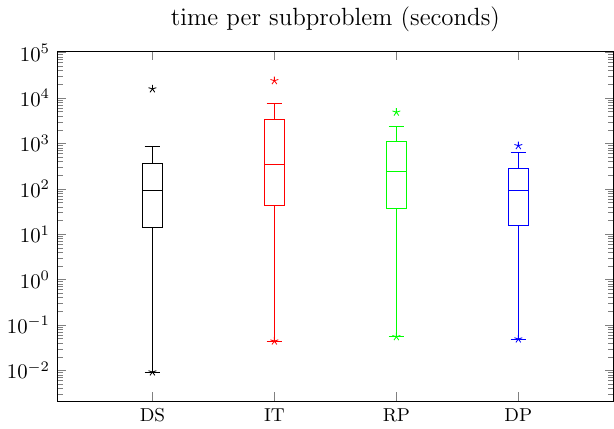}
&
\includegraphics[width=.48\textwidth]{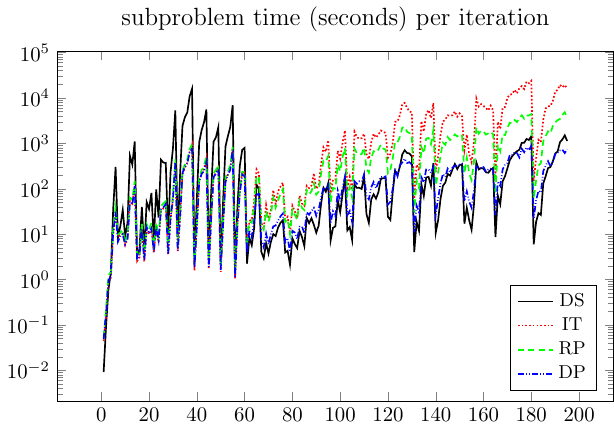}
\\[1ex]
\includegraphics[width=.48\textwidth]{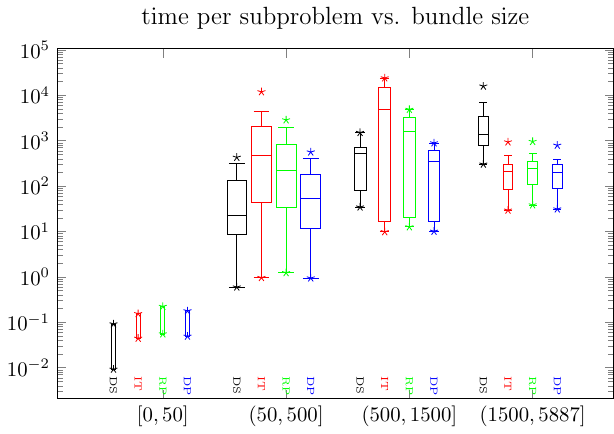}
&
\includegraphics[width=.48\textwidth]{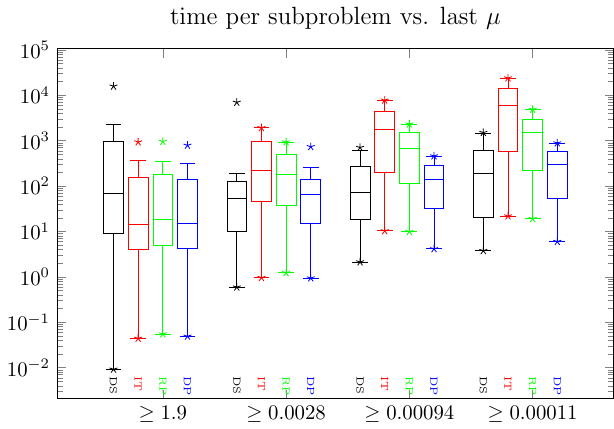}
\\[1ex]
\includegraphics[width=.48\textwidth]{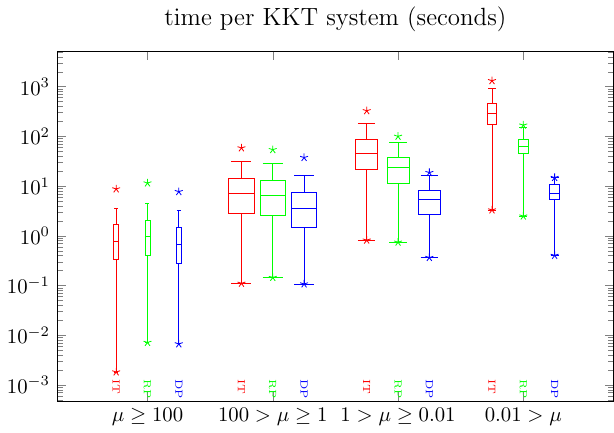}
&
\includegraphics[width=.48\textwidth]{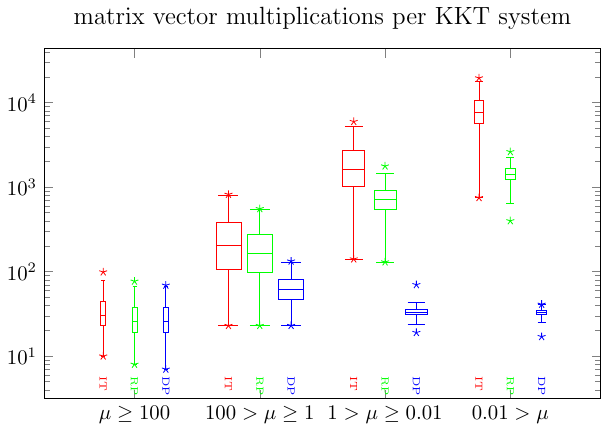}
\\[1ex]
\includegraphics[width=.48\textwidth]{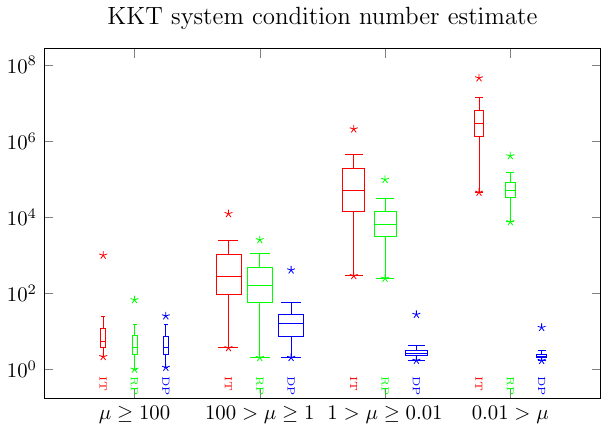}
&
\includegraphics[width=.48\textwidth]{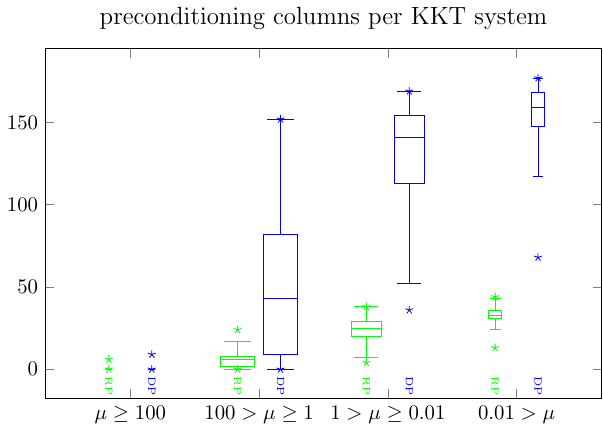}
\end{tabular}
}\vspace*{-.5ex}
\caption{\label{fig:BIS} Instance BIS, \upd{195} subproblems, \upd{6180} KKT systems. In order to highlight the dependence on the barrier parameter $\mu$ with corresponding precision requirements the results for KKT systems are grouped into $\mu$-value ranges $(\infty,100]$, $(100,1]$, $(1,0.01]$, $(0.01,0)$.}
\end{figure}

The semidefinite relaxation of minimum bisection is similar in nature
to max-cut, but in addition to the single diagonal elements there is a
variable with coefficient matrix of all ones. Furthermore, variables
with sparse coefficient matrices corresponding to odd cycles in the
underlying graph are added dynamically in rounds, see
\cite{Helmberg2004} for the general framework and also for the origin
of the instance KKT\_traj33 with 20006 nodes and roughly 260000 edges.

Again, after the very first iterations the iterative methods turn out
to perform distinctly better in the initial phase of the
algorithm. Iterative methods get less attractive as precision
requirements increase.  The model size is often rather small (a bit
larger than the active set of about 150 columns) which is favorable for
DS. Indeed, additional output information of the log file indicates
that the performance of DS drops off whenever the cutting model is
significantly larger than that.

While for this instance RP is better than IT, the advantage of DP over
the other iterative variants is quite apparent and its superiority
also increases with precision requirements and smaller $\mu$. In fact,
for DP the condition number and the number of matrix-vector multiplications
decrease again for smaller $\mu$. Possible causes might be that the
active set is easier to identify correctly.  Due to the reduction in
matrix-vector multiplications, computation time does not increase for
DP in spite of a growing number of columns in the preconditioner.

\subsection{A Min-Max-Bisection Problem (Instance MMBIS, Figure \ref{fig:MMBIS})}

This problem arose in the context of an unpublished
attempt\footnote{together with B. Filipecki (TU Chemnitz), S. Heyder
  (TU Ilmenau), Th.\ Hotz (TU Ilmenau) within BMBF-project grant
  05M18OCA.} to optimize vaccination rates for five population groups
$N_1\dot\cup\dots\dot\cup N_5=N$ in a virtual town of $n=|N|$
inhabitants. Briefly, within the town $k$ anonymous people are assumed to be
infectious. There is vaccine for at most $\underline n$ people. The
aim is to reduce the spreading rate of the disease by vaccinating each
person with the respective group's probability. The task of
determining these vaccination rates motivated the following ad hoc model
which would be hard to justify rigorously. In a graph
$G=(N,E)$ each edge $ij=\{i,j\}\in E$ with $i\in N_\hati$,
$j\in N_\hatj$ has a weight $\hat w_{\hati\hatj}$ representing the infectiousness
of the typical contact for these two persons of the respective groups.
It will be convenient
to define the weighted Laplacians $L_{\hati\hatj}=\hat w_{\hati\hatj}\sum_{ij\in E, i\in N_\hati,j\in N_\hatj}(e_i-e_j)(e_i-e_j)^\top$. In
this simplified approach, vaccination rates $v_\hati,v_\hatj$ of the node
groups reduce a nominal infectiousness $\hat w_{\hati\hatj}$ between
these groups by the factor
$y_{\hati\hatj}\ge\max\{0,1-v_\hati-v_\hatj\}$.
The spreading probability to be minimized is considered proportional to the restricted
max-cut value
\begin{eqnarray*}
  &\displaystyle\max_{\calI\subset N,|\calI|=k}\sum y_{\hati\hatj}\hat w_{\hati\hatj}\cdot|\{ij\in E\colon i\in\calI\cap N_\hati,j\in N_\hatj\setminus\calI\}|=&\\
  &=\displaystyle\max_{x\in\{-1,1\}^n\atop (\one^\top x)^2=(n-2k)^2}\frac14 x^\top\bigg(\sum_{\hati\le\hatj}y_{\hati\hatj}L_{\hati\hatj}\bigg)x,& 
\end{eqnarray*}
For determining the vaccination rates the combinatorial problem is replaced by the usual (dual) semidefinite relaxation
\begin{displaymath}
  \begin{array}{rl}
    \text{minimize} &\frac{n}4\lmax(\sum_{\hati\le\hatj}y_{\hati\hatj}L_{\hati\hatj}-\Diag(d)-u\one\one^\top)+\one^\top d+(n-2k)^2u\\
    \text{subject to} & y_{\hati\hatj}\ge 1-v_\hati-v_\hatj,\qquad\hati\le\hatj,\\
                    &\sum_{\hati} |N_\hati|v_\hati=\underline n,\\
    & d\in\R^n,u\in\R,y\ge 0, v\ge 0.
  \end{array}
\end{displaymath}

\begin{figure}[htbp]
\centerline{
\begin{tabular}{@{}cc@{}}
\includegraphics[width=.48\textwidth]{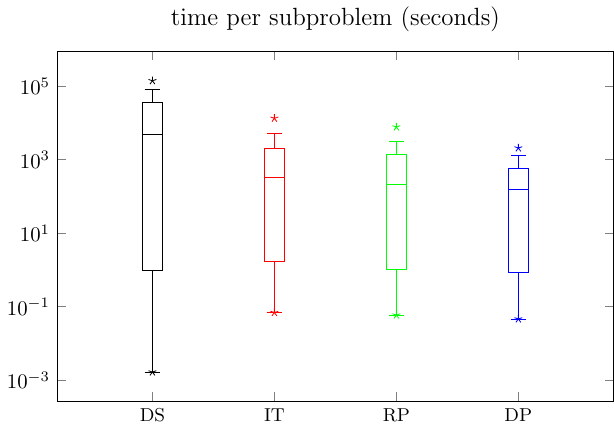}
&
\includegraphics[width=.48\textwidth]{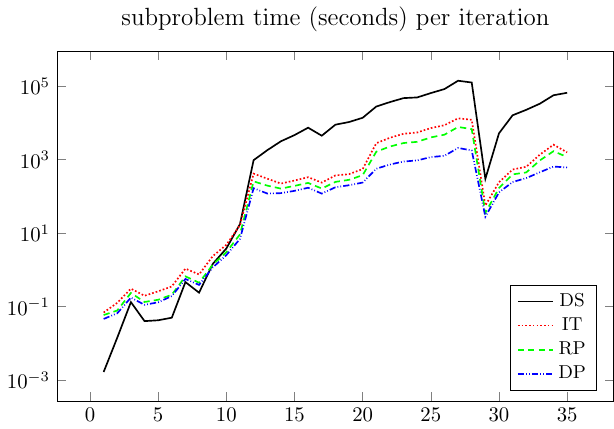}
\\[1ex]
\includegraphics[width=.48\textwidth]{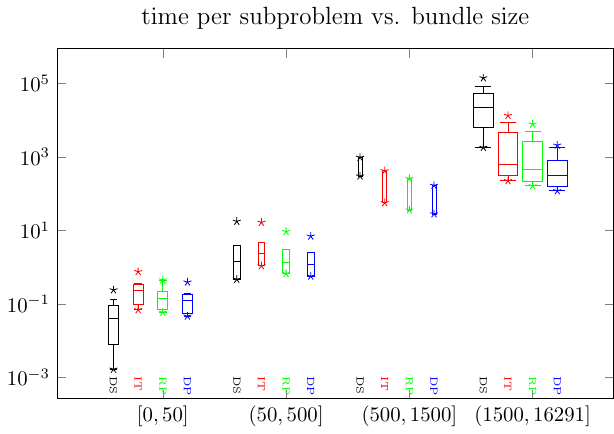}
&
\includegraphics[width=.48\textwidth]{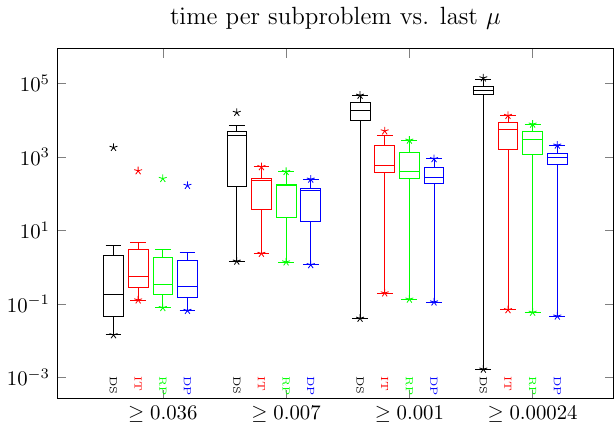}
\\[1ex]
\includegraphics[width=.48\textwidth]{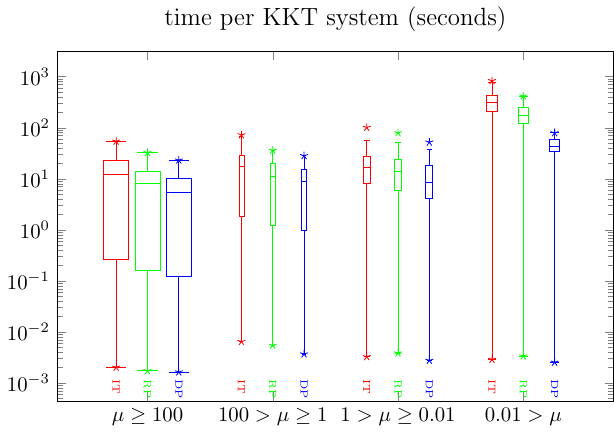}
&
\includegraphics[width=.48\textwidth]{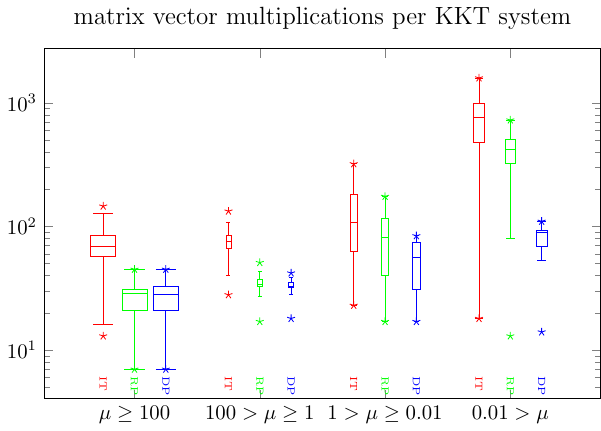}
\\[1ex]
\includegraphics[width=.48\textwidth]{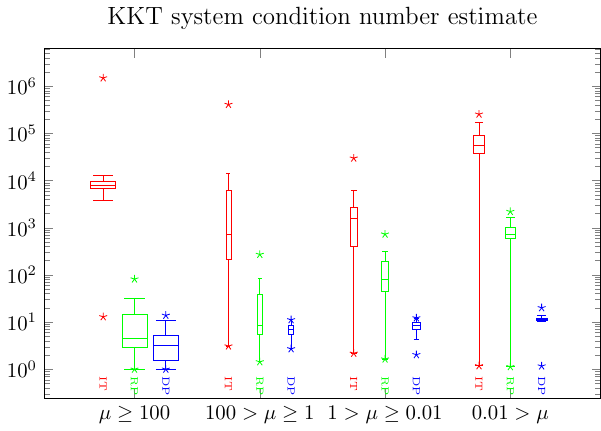}
&
\includegraphics[width=.48\textwidth]{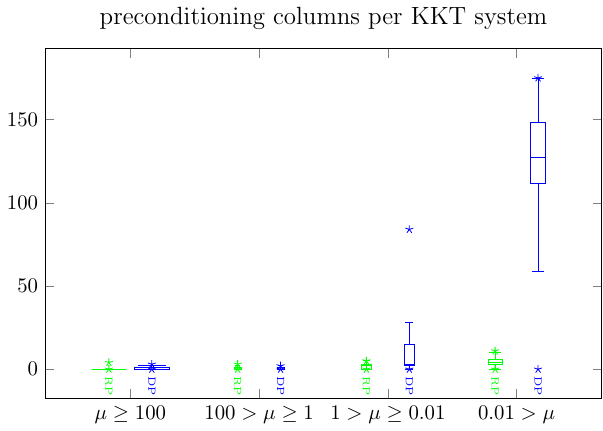}
\end{tabular}
}\vspace*{-.5ex}
\caption{\label{fig:MMBIS} Instance MMBIS, \upd{35} subproblems, \upd{782} KKT systems. In order to highlight the dependence on the barrier parameter $\mu$ with corresponding precision requirements the results for KKT systems are grouped into $\mu$-value ranges $(\infty,100]$, $(100,1]$, $(1,0.01]$, $(0.01,0)$.}
\end{figure}

In this case the resulting KKT system also has an equality and several inequality constraints in the block $A$. Preconditioning results are presented for the KKT systems of an instance with $n=12600$ inhabitants splitting into groups of sizes $5770, 6000, 600, 30, 200$, with $k=126$ infectious persons and $\underline n=1260$ available vaccinations.

In the actual computations the bundle size grows surprisingly
fast. This not only entails enormous memory requirements but also
excessive computation times for DS; indeed, computations of DS may
exceed those of DP by a factor of 70. In consequence comparative
results can only be reported for a very limited number of subproblem
evaluations. In particular, the precision requirements remain rather
moderate throughout these iterations. Still, the same initial behavior
can be observed as for the previous two instances. For very small
bundle sizes DS is best. Once the bundle size grows, the iterative
methods take over. Among the iterative solvers RP is better than IT,
but DP is the method of choice. It succeeds in tightly controlling the
condition number by selecting rather few columns. With this DP
requires the fewest matrix vector multiplications which seems to pay
of quickly on this instance.

\subsection{\upd{Performance within the Bundle Method for
  Max-Cut}}\label{S:bundleperf}

\upd{The purpose of this section is to provide evidence for the
  reliability of the KKT oriented evaluations when the iterative
  solvers are employed within the bundle method directly.  As
  explained in the introductory remarks to this Section
  \ref{S:experiments}, a full assessment of the use of iterative
  solvers within conic bundle methods is out of scope and beyond the
  possibilities of this work. Therefore results will only compare,
  without any further adaptations, the direct replacement of DS with
  the solvers IT, RP and DP, within the current ConicBundle
  implementation that was developed and tuned for DS. Note, however,
  that the evaluation of bundle methods requires a statistical approach.}

\upd{In oracle based nonsmooth optimization it is typical that even
  slight numerical changes in the computation of candidates bring
  along significant differences in the actual trajectories. Indeed,
  candidates are generically close to ridges. Which subgradient is
  returned depends on which side of the ridge the candidate ends
  up. In particular, the use of different KKT solvers quickly leads to
  considerable differences in the models and subproblems and therefore
  also in the sequence of KKT problems. This erratic behavior is intrinsic at
  any level of precision, therefore it may be expected that the
  average number of bundle iterations (descent and null steps) does
  not depend too much on the actual KKT solver in use. Yet, due
  to this incomparability of trajectories, any attempt to assess the
  scope of the iterative solvers in comparison to the direct solver
  needs to be based on a reasonable collection of comparable instances.
  Their choice should help to illustrate the effects of parameters, that can be
  expected to be influential in the current context,
  \begin{compactitem}
  \item the cost of matrix-vector multiplications,
  \item the size of the model,
  \item precision requirements,
  \item the use or non-use of a predictor corrector approach,
  \item the number of KKT instances and solves per subproblem.
  \end{compactitem}
  In order to cover these aspects with manageable effort, results will
  be presented for eight methods and four groups of 25 randomly
  generated Max-Cut instances. The methods without predictor corrector
  approach are denoted by DS, IT, RP, DP and those with predictor
  corrector approach by DSp, ITp, RPp, DPp. The names refer to using
  the respective direct or iterative solver for the KKT systems of the
  internal interior point method of ConicBundle for solving the
  subproblems.  The four instance classes arise by generating five
  instances per number of nodes $n\in\{10000,20000\}$ and per density
  out of two edge density groups, one with smaller densities
  $d\in\{0.1,0.2,0.3,0.4,0.5\}$ and one with higher density
  $d\in\{1,2,3,4,5\}$ (\cite{rudy}, call \texttt{rudy -rnd\_graph n d
    s} for seed $s\in\{1,2,3,4,5\})$. The instances were solved with
  ConicBundle \cite{ConicBundle2021} on computers having
  QUAD-Core-processors INTEL-Core-I7-4770 with 4$\times$ 3400MHz, 8 MB
  Cache, 32 GB RAM and operating system Ubuntu 18.04. The code was run
  in sequential mode with each instance solved en suite for all
  methods on the same machine and all time measurements refer to user
  time.  Mandated by limited resources, some volatility may have been
  caused by running two instances on each machine at the same time as
  well as by occasional further jobs. As instances and methods were
  randomly affected by this, influence on the conclusions should be
  marginal in view of the number of examples. }

\upd{The max-cut instances serve the purpose well for the
  following reasons. First, as explained before, the direct solver DS
  is particularly efficient for Max-Cut instances, because the Schur
  complement needs to be computed only once at the beginning of each
  bundle step for all interior point iterations / KKT systems
  associated with this subproblem. Thus, if iterative solvers are
  competitive for Max-Cut this should also hold for more general
  cases. Likewise, the iterative solver IT without preconditioning
  performed better on the KKT instances for Max-Cut than on those of
  the two other examples, therefore the limits of preconditioning are
  best discussed for Max-Cut. Second, for Max-Cut even large scale
  instances can be solved to reasonably high precision in manageable
  time which allows to compare the performance on several precision
  levels. To make this comparison reasonably efficient and fair in
  view of the weaknesses of the lack of progress stopping criterion of
  bundle methods, the comparisons use for each level of relative
  precision $10^{-3}$, $10^{-4}$, $10^{-5}$ and $10^{-6}$ the first
  descent step that produces a value below an instance dependent
  common relative reference value. For each instance this reference
  value is obtained by taking the minimum objective value obtained
  over all methods by running ConicBundle with termination precision
  $10^{-6}$.  Third, random max-cut instances having the same
  number of nodes and similar edge density can be expected to have
  similar parameters and properties in terms of model size, cost of
  matrix-vector multiplications, and precision requirements. Note that
  a higher edge density increases the cost of matrix-vector
  multiplications but also causes a larger offset in objective value
  which entails somewhat reduced precision requirements for the KKT
  systems to reach the same relative precision. As the experiments
  will show, the model size --- it is selected by ConicBundle on basis
  of the active rank, starts with roughly twice this size after
  descent steps for reasons investigated in \cite{DingGrimmer2020} and
  increases further over null steps --- seems to be less dependent on
  the edge densities but grows markedly with the number of nodes.}

\begin{figure}[htbp]
\centerline{
  \upd{\begin{tabular}{@{}cc@{}}
   \multicolumn{2}{c}{Development of subproblem / bundle step
    parameters for the 10000
    node instances}\\
    density $\in\{0.1,0.2,0.3,0.4,0.5\}$ &   density $\in\{1,2,3,4,5\}$\\
\includegraphics[width=.48\textwidth]{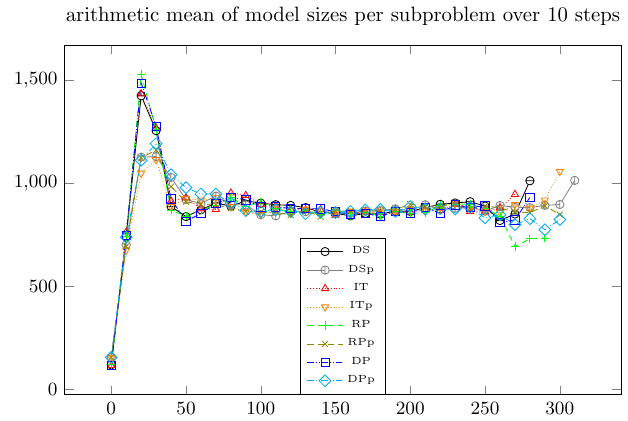}
&
\includegraphics[width=.48\textwidth]{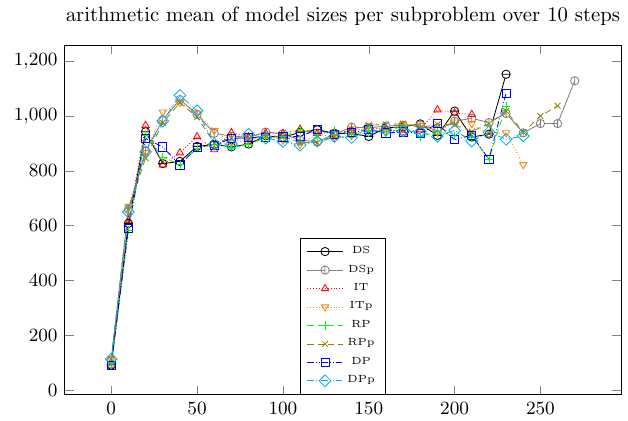}
\\[1ex]
\includegraphics[width=.48\textwidth]{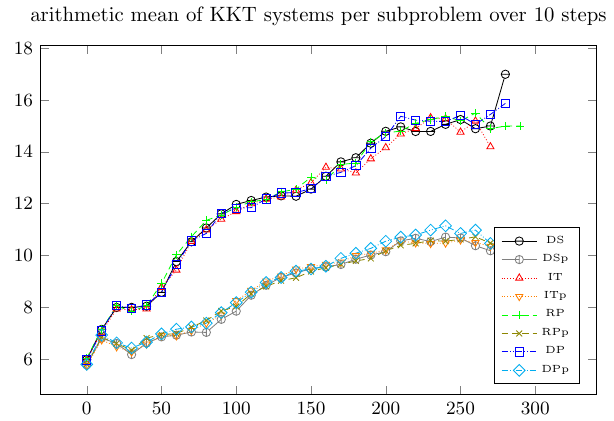}
&
\includegraphics[width=.48\textwidth]{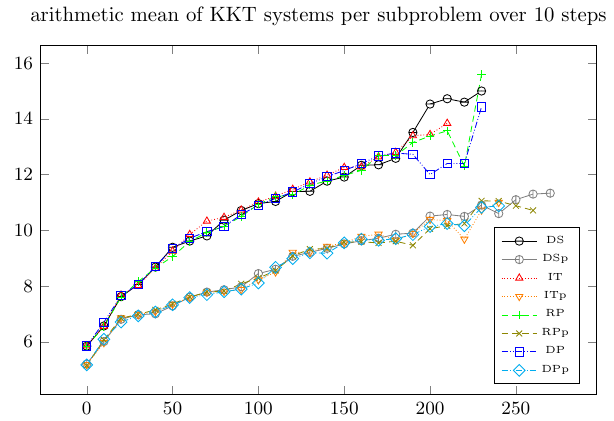}
\\[1ex]
\includegraphics[width=.48\textwidth]{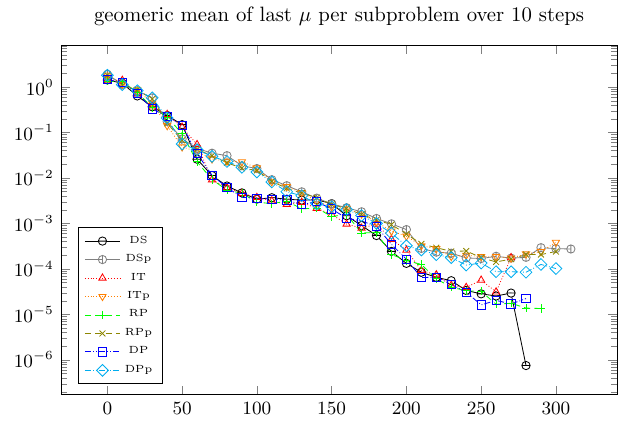}
&
\includegraphics[width=.48\textwidth]{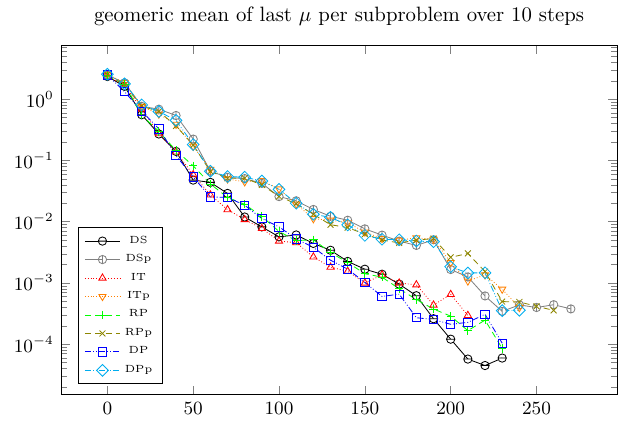}
\\[1ex]
\includegraphics[width=.48\textwidth]{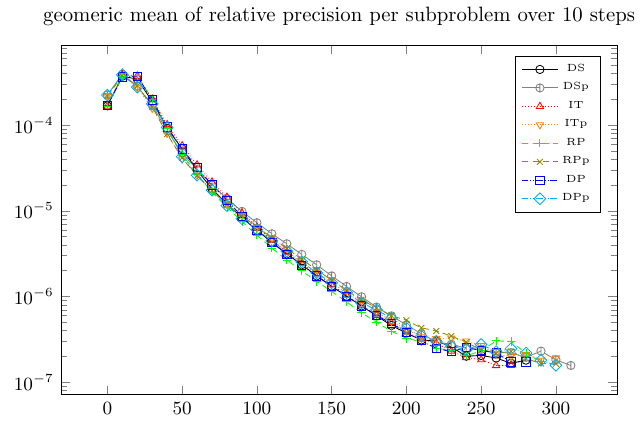}
&
\includegraphics[width=.48\textwidth]{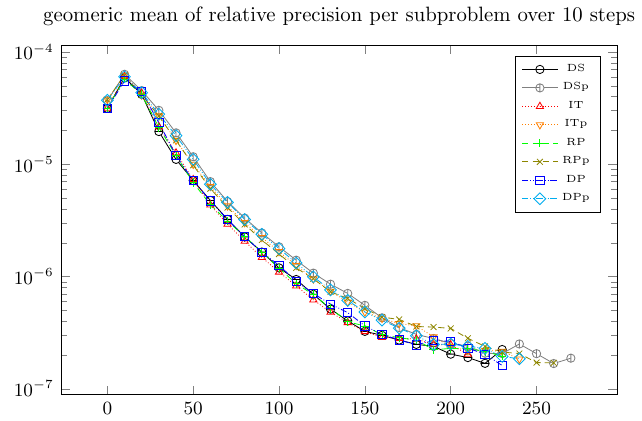}
\\[1ex]
\end{tabular}}
}\vspace*{-.5ex}
\caption{\label{fig:param1} \upd{The plots show for each method for
  successive groups of 10 bundle iterations (null or descent steps)
  over all instances the arithmetic mean of
  the model size and of the number of interior point iterations of the
  subproblems in this group and the geometric mean of the
  last $\mu$ value and of the relative precision the interior
  point method needed to reach. All iterations of the performance
  profile for relative precision level $10^{-6}$ are included.}}
\end{figure}

\upd{\begin{figure}[htbp]
\centerline{
  \upd{\begin{tabular}{@{}cc@{}}
   \multicolumn{2}{c}{Development of subproblem / bundle step
    parameters for the 20000
    node instances}\\
    density $\in\{0.1,0.2,0.3,0.4,0.5\}$ &   density $\in\{1,2,3,4,5\}$\\
\includegraphics[width=.48\textwidth]{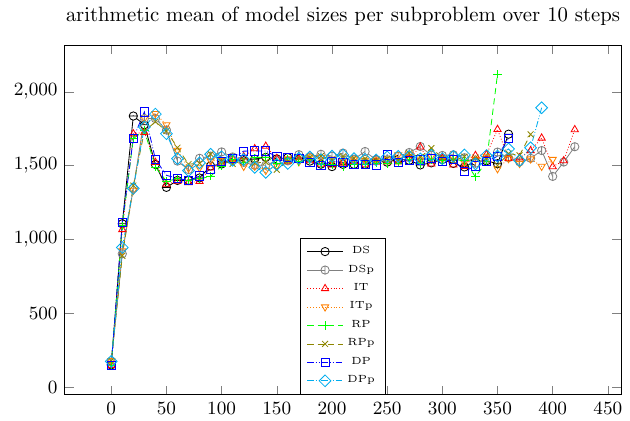}
&
\includegraphics[width=.48\textwidth]{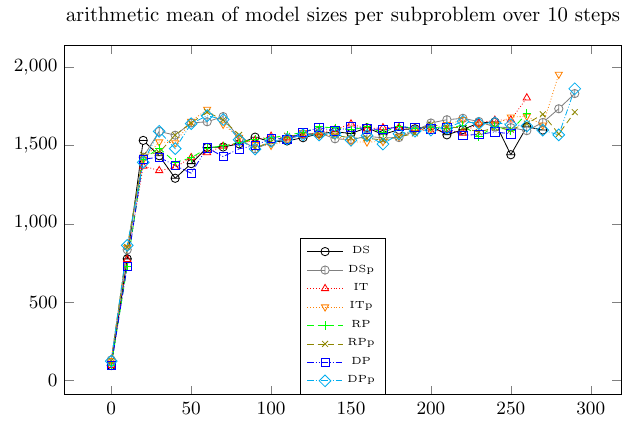}
\\[1ex]
\includegraphics[width=.48\textwidth]{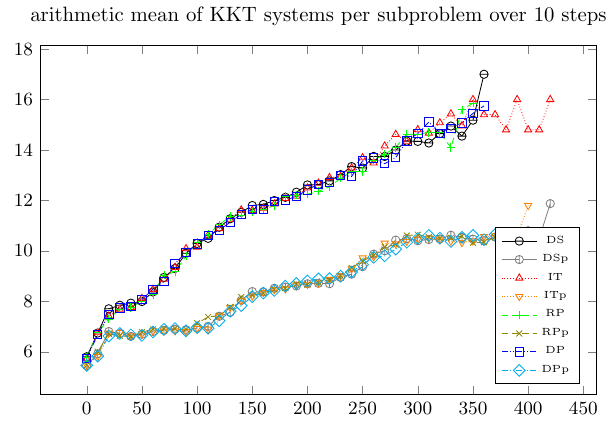}
&
\includegraphics[width=.48\textwidth]{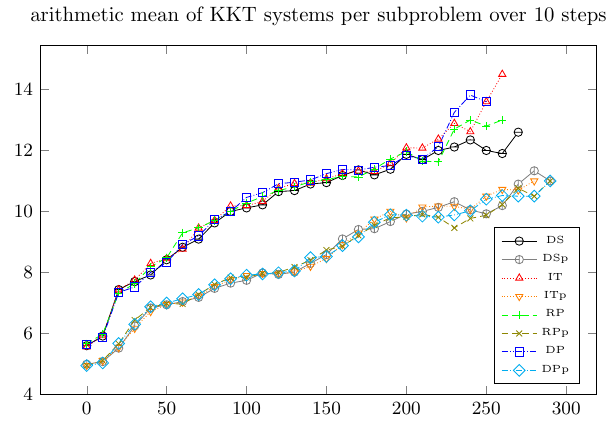}
\\[1ex]
\includegraphics[width=.48\textwidth]{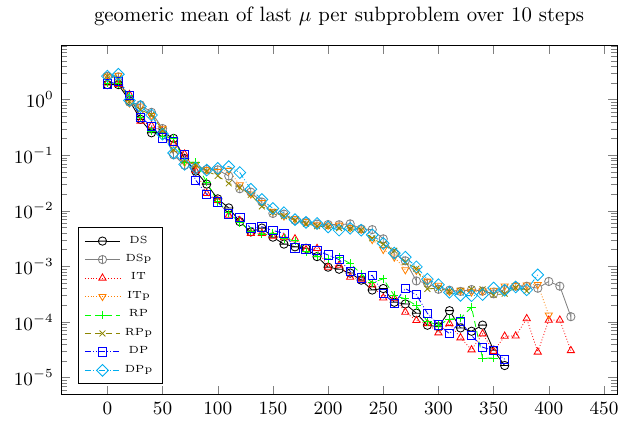}
&
\includegraphics[width=.48\textwidth]{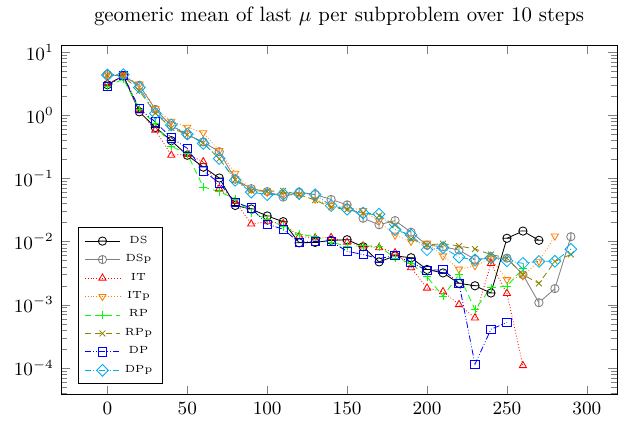}
\\[1ex]
\includegraphics[width=.48\textwidth]{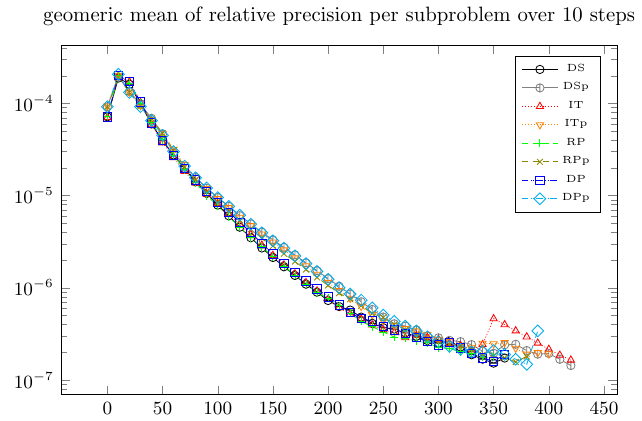}
&
\includegraphics[width=.48\textwidth]{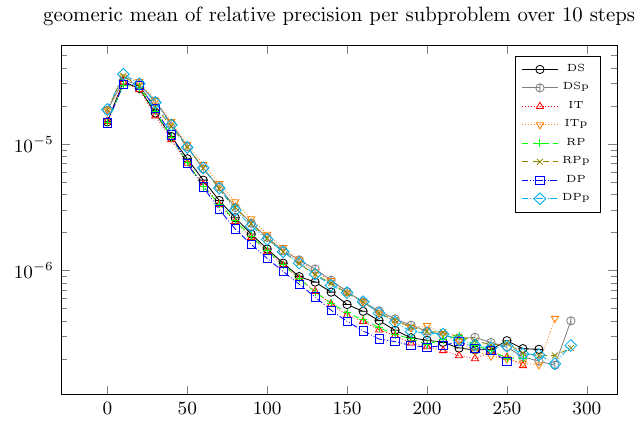}
\\[1ex]
\end{tabular}}
}\vspace*{-.5ex}
\caption{\label{fig:param2} \upd{The plots show for each method for
  successive groups of 10 bundle iterations (null or descent steps)
  over all instances the arithmetic mean of
  the model size and of the number of interior point iterations of the
  subproblems in this group and the geometric mean of the
  last $\mu$ value and of the relative precision the interior
  point method needed to reach. All iterations of the performance
  profile for relative precision level $10^{-6}$ are included.}}
\end{figure}}

\upd{The first aspect to address is the dependence of the bundle
  method on the solvers.  For this figures \ref{fig:param1} and
  \ref{fig:param2} display for each group and solver the average
  development of the model sizes, the number of KKT systems solved per
  subproblem together with the last $\mu$-value occuring there (it
  reflects the precision requirement of the final KKT system within
  the subproblem) and also the precision requirements on the
  subproblems themselves.  This development is recorded in averages
  over groups of 10 steps and all 25 instances for each of the four
  instance groups. This rather detailed view will also help to explain
  differences in the computation times of the solvers for various
  precision levels.}

\begin{figure}[htbp]
\centerline{
  \upd{\begin{tabular}{@{}cc@{}}
   \multicolumn{2}{c}{Performance profiles of bundle steps and KKT
         systems for sparser 20000 node instances}\\
\includegraphics[width=.48\textwidth]{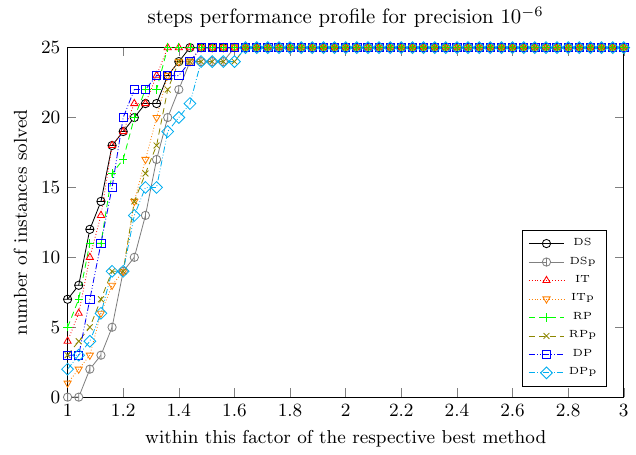}
&
\includegraphics[width=.48\textwidth]{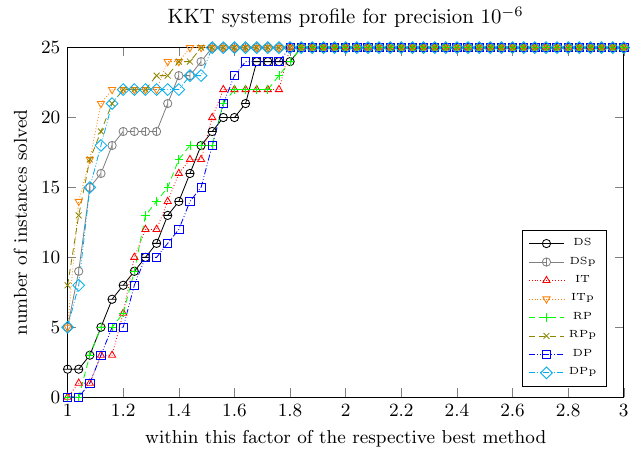}
\end{tabular}
}}\vspace*{-.5ex}
\caption{\label{fig:stepperf} \upd{The performance profiles are generated for the 20000 node
  instances having densities in $\{0.1,0.2,0.3,0.4,0.5\}$ with respect to the
  total number of bundle iterations (null and descent steps) [left] and the
  total sum of KKT systems over all subproblems [right] needed to
  reach a relative precision of $10^{-6}$ relative
  to the best value obtained over all methods.}}
\end{figure}

\upd{The plots of figures \ref{fig:param1} and \ref{fig:param2}
  exhibit a natural separation between the methods with and without
  predictor corrector approach. In comparison, the differences between
  the solvers are almost negligible except for a few final iterations,
  which also suffer from stronger volatility due to a reduced number
  of samples. It is worth noting that the same holds for the overall
  number of KKT systems and bundle steps throughout all relative
  precision levels. A helpful visualization to illustrate such
  comparisons are performance profiles \cite{DolanMore2002} for the
  total number of steps and KKT systems for each precision level.
  These turn out to have a shape similar to that displayed in
  Figure~\ref{fig:stepperf} which presents the two profiles for bundle
  steps and KKT systems for the sparser case on 20000 nodes and
  relative precision $10^{-6}$. While the smaller number of KKT
  systems (i.e. interior point iterations) of the predictor corrector
  variants is an expected outcome, it is rather surprising that the
  variants without predictor corrector seem to need a few bundle steps
  less on average to reach the required precision. A comparison with
  the model size plots of \ref{fig:param1} and \ref{fig:param2}
  suggests, that there is a distinct difference in the nature of these
  interior point solutions that also has its effect on the bundle
  selection mechanism, but so far this lacks a mathematically sound
  explanation.  Still, a first conclusion might read that the behavior
  of the bundle method itself is, on average, independent of the
  choice of the four solvers.}

\begin{figure}[htbp]
\centerline{
  \upd{\begin{tabular}{@{}cc@{}}
    \multicolumn{2}{c}{Time performance profiles for max-cut on 10000 nodes with 5
    instances per 5 densities}\\
    density $\in\{0.1,0.2,0.3,0.4,0.5\}$ &   density $\in\{1,2,3,4,5\}$\\
\includegraphics[width=.48\textwidth]{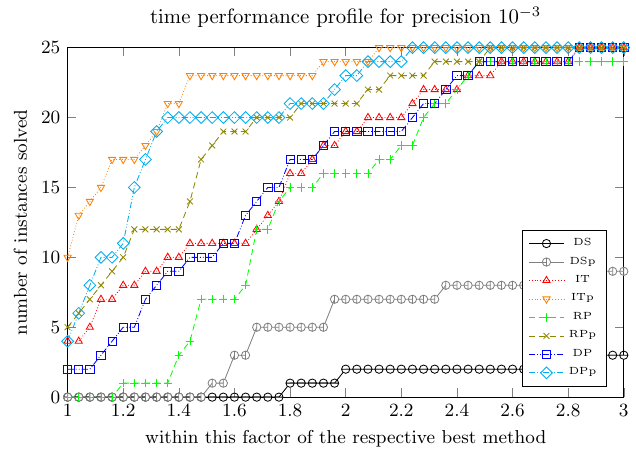}
&
\includegraphics[width=.48\textwidth]{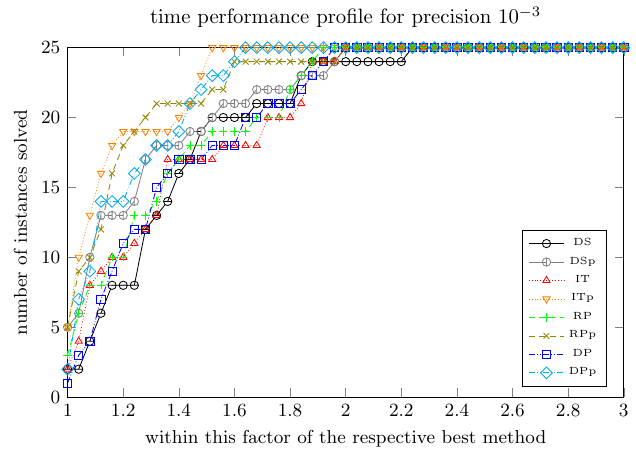}
\\[1ex]
\includegraphics[width=.48\textwidth]{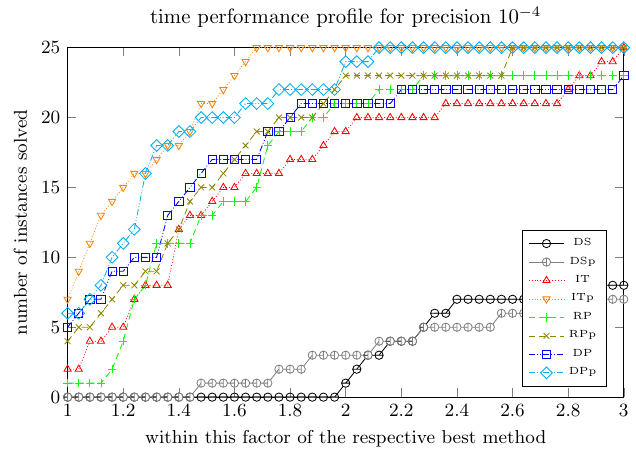}
&
\includegraphics[width=.48\textwidth]{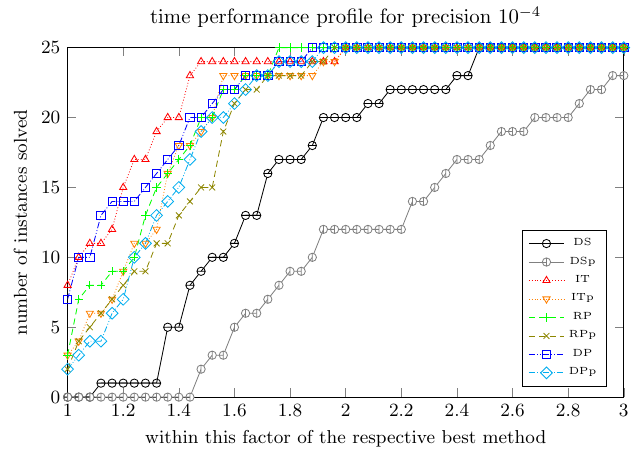}
\\[1ex]
\includegraphics[width=.48\textwidth]{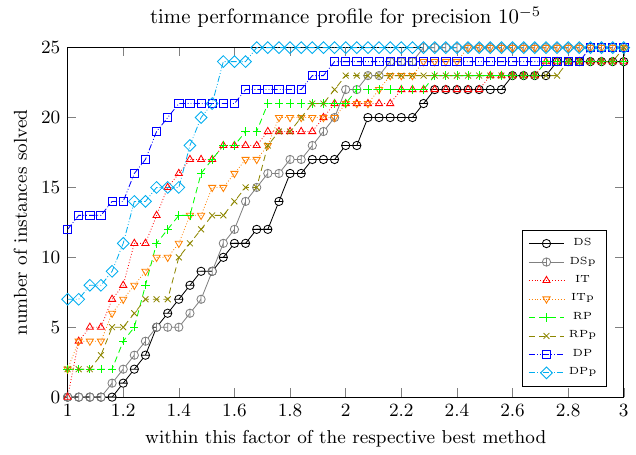}
&
\includegraphics[width=.48\textwidth]{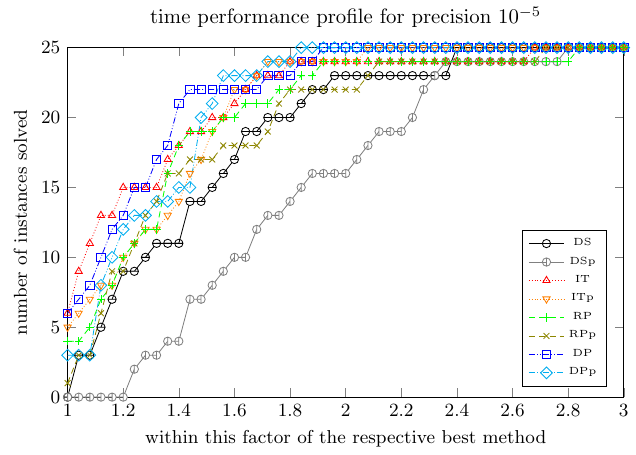}
\\[1ex]
\includegraphics[width=.48\textwidth]{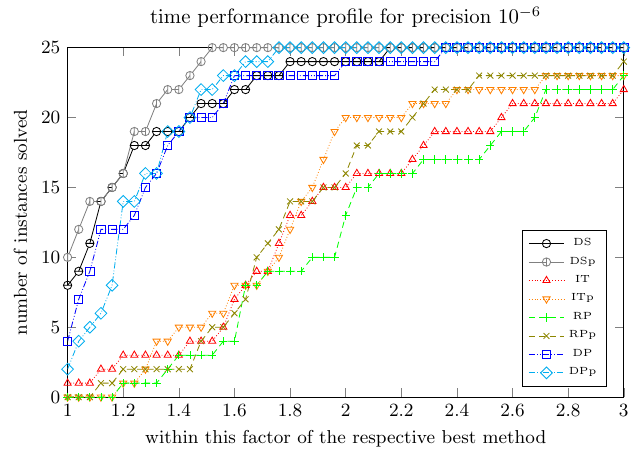}
&
\includegraphics[width=.48\textwidth]{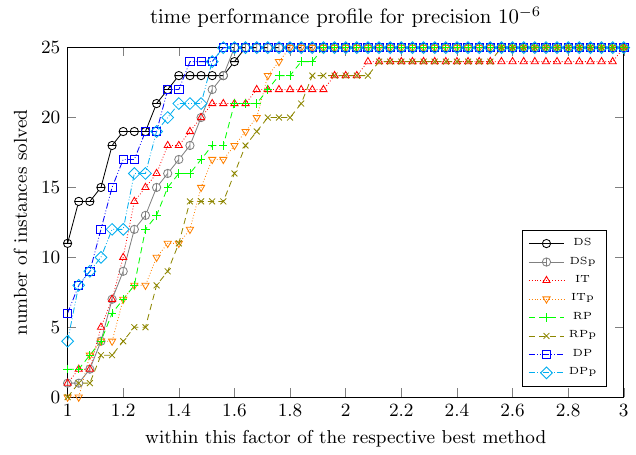}
\end{tabular}}
}\vspace*{-.5ex}
\caption{\label{fig:timeperf1} \upd{The profile is taken with respect to the time needed to
  reach a relative precision of $10^{-3}$, $10^{-4}$, $10^{-5}$, and $10^{-6}$ relative
  to the best value obtained over all methods.}}
\end{figure}

\begin{figure}[htbp]
\centerline{
  \upd{\begin{tabular}{@{}cc@{}}
   \multicolumn{2}{c}{Time performance profiles for max-cut on 20000 nodes with 5
   instances per 5 densities}\\
    density $\in\{0.1,0.2,0.3,0.4,0.5\}$ &   density $\in\{1,2,3,4,5\}$\\
\includegraphics[width=.48\textwidth]{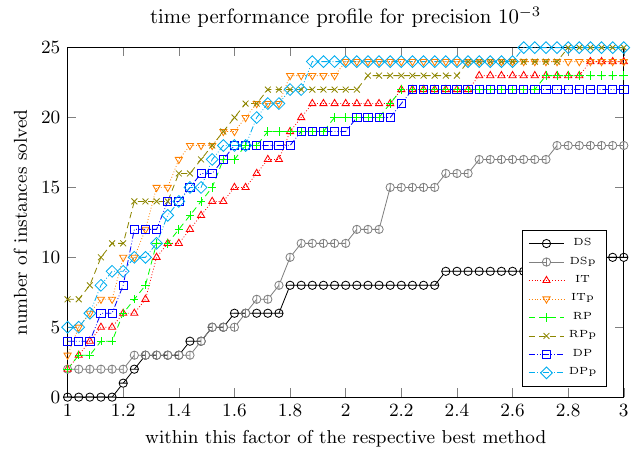}
&
\includegraphics[width=.48\textwidth]{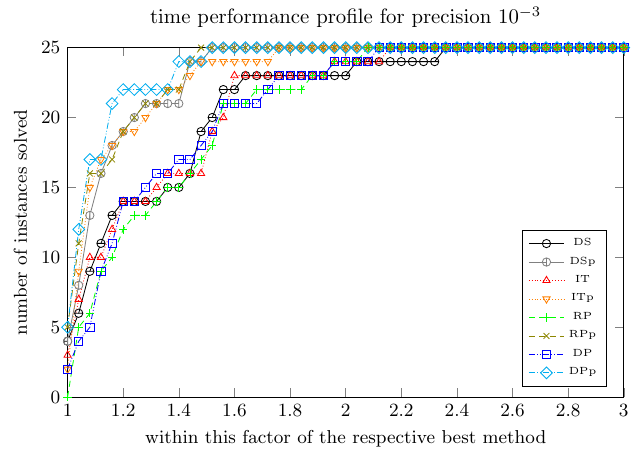}
\\[1ex]
\includegraphics[width=.48\textwidth]{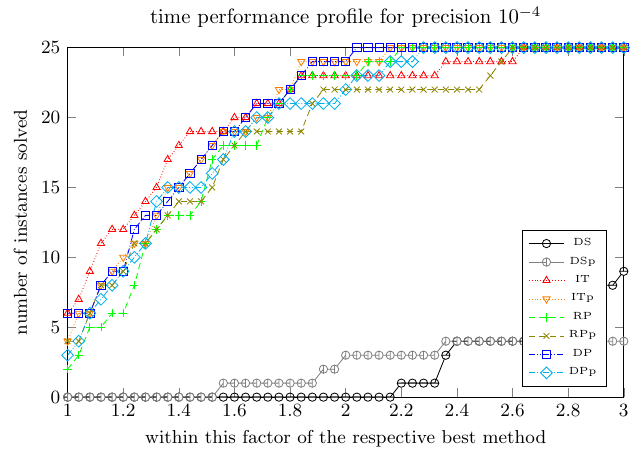}
&
\includegraphics[width=.48\textwidth]{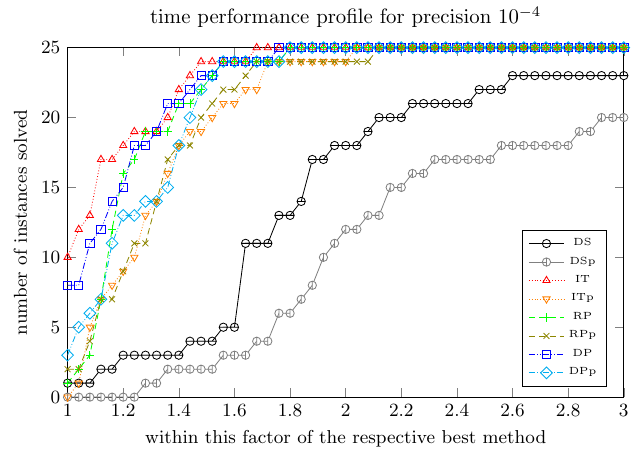}
\\[1ex]
\includegraphics[width=.48\textwidth]{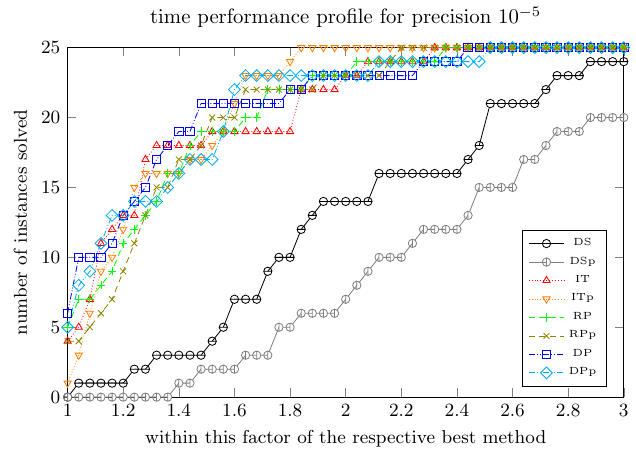}
&
\includegraphics[width=.48\textwidth]{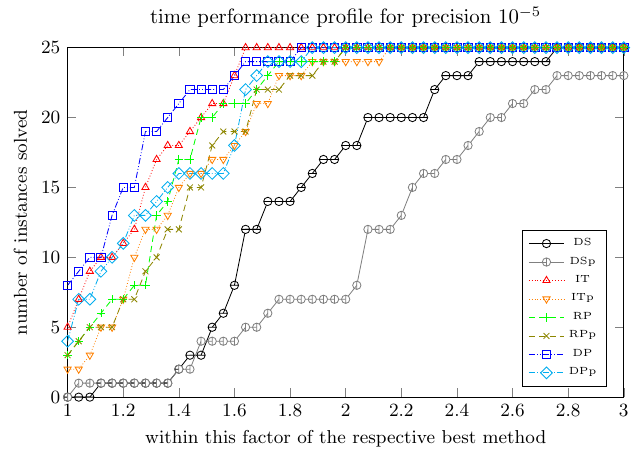}
\\[1ex]
\includegraphics[width=.48\textwidth]{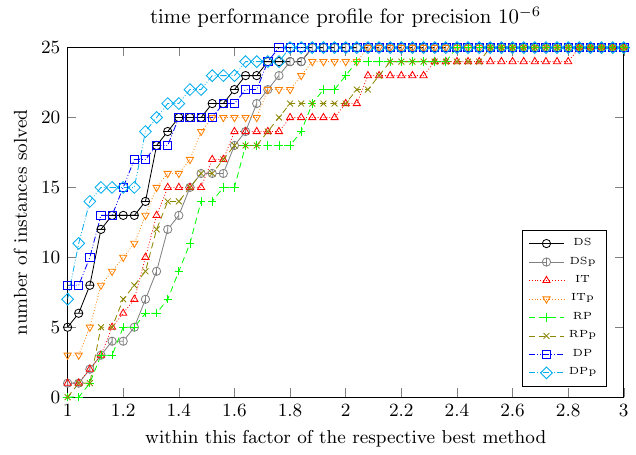}
&
\includegraphics[width=.48\textwidth]{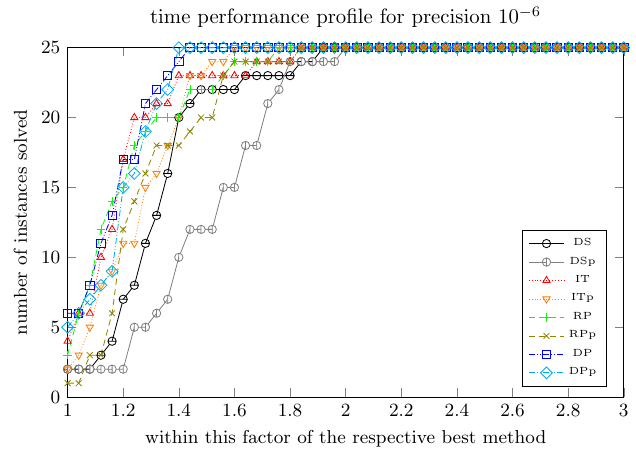}
\end{tabular}}
}\vspace*{-.5ex}
\caption{\label{fig:timeperf2} \upd{The profile is taken with respect to the time needed to
  reach a relative precision of $10^{-3}$, $10^{-4}$, $10^{-5}$, and $10^{-6}$ relative
  to the best value obtained over all methods.}}
\end{figure}

\upd{Figures \ref{fig:timeperf1} and \ref{fig:timeperf2} display
  computation time performance profiles of the eight methods on the
  four classes of 25 instances for relative precision levels
  $10^{-3}$, $10^{-4}$, $10^{-5}$ and $10^{-6}$.  These largely match
  the results on individual KKT systems. First consider the direct
  solvers DS and DSp. Again, a good explanation is lacking for the
  fact that DS dominates DSp in many cases with higher number of edges
  and higher precision. Whether DS and DSp are attractive compared to
  iterative methods depends on the ratio of the time invested into
  forming the Schur complement to the number of KKT steps required for
  solving the subproblem, \ie they are preferable if the model size is
  small or the number of interior point iterations becomes large
  enough due to increasing precision requirements.  In cases of strong
  initial growth of the model size (see the model size plots of
  figures \ref{fig:param1} and \ref{fig:param2}) iterative solvers are
  quickly better. The seemingly good performance of the direct solvers
  on the denser instances for precision level $10^{-3}$ is mostly due
  to the large constant offset that causes the methods to reach this
  precision often within ten steps (compare this to
  the asymptotic analysis in \cite{HelmbergMoharPoljakRendl95}); at
  this point model sizes are still small. Iterative solvers dominate
  precision levels $10^{-4}$ and $10^{-5}$ with reasonably low
  accuracy and few interior point iterations. The influence of the
  cost of a matrix-vector multiplication is visible in the difference
  of the initial head start to direct solvers between sparser and
  denser instances for precision $10^{-4}$. For instances on 20000
  nodes the average model size is one and a half times the average
  size of the 10000 node instances (see figures \ref{fig:param1} and
  \ref{fig:param2})) and this explains part of the stronger
  performance of the iterative solvers on larger instances. For
  increasing precision requirements and number of interior point
  iterations the profiles also suggest that DS and DSp catch up faster
  for instances with fewer edges. This effect might again be caused by the
  constant offset, that is larger for random Max-Cut instances with
  larger number of edges. Indeed, an inspection of the last $\mu$
  plots and KKT systems per subproblem in figures \ref{fig:param1} and
  \ref{fig:param2} suggests that for max-cut instances with a larger
  number of edges the subproblem solutions require less absolute
  accuracy which favors iterative solvers and compensates somewhat the
  higher cost of the matrix-vector multiplications. Note that the
  relative precision requirements for the solution of the subproblems
  (see figures \ref{fig:param1} and \ref{fig:param2}) are almost
  identical.}

\upd{For the iterative methods the predictor corrector variant seems
faster on lower precision levels but again the methods without
predictor corrector catch up or may even dominate higher precision
levels. For IT (no preconditioning) this should largely be due to the
fact that predictor corrector requires two solves per KKT system.
Thus, taking twice the number of KKT systems per problem for
predictor corrector variants in figures
\ref{fig:param1} and \ref{fig:param2} as the number of
required solves provides a satisfactory explanation for IT.
For RP and DP the situation is less clear cut, because
the preconditioner is formed only once per KKT system, but
the line of argument is similar. For the moderate accuracy
levels $10^{-3}$ and $10^{-4}$ Max-Cut instances could
do without preconditioning, but the preconditioned variants
do a good job. For $10^{-5}$ the advantage begins to show
and for $10^{-6}$ the DP variants are almost consistently
better than the other iterative methods.} 

\upd{Based on this analysis, a hybrid approach seems advisable that
  switches dynamically between the solvers depending on precision,
  model size and number of interior point iterations. In implementing
  these ideas a number of further design aspects would have to be
  reconsidered as outlined before.  The true advantage of iterative
  solvers, however, is that dynamic model adaptations become feasible
  during the solution of the subproblem, because there is no need to
  recompute the Schur complement each time. This allows for entirely
  new strategies such as combining the ideas of
  \cite{OskoorouchiGoffin2003,BabonneauBeltranHaurieTadonkiVial2007} and \cite{HelmbergRendl98} in order
  to cut down on the number null steps at an early stage. This remains
  to be addressed in future work.}

\section{Conclusions}\label{S:concl}

In search for efficient low rank preconditioning techniques for the
iterative solution of the internal KKT system of the quadratic bundle
subproblem two subspace selection heuristics --- a randomized and a
deterministic variant --- were proposed.  For the randomized approach
the results are ambivalent in theory and in practice; obtaining a good
subspace this way seems to be difficult and the cost of exploratory
matrix-vector multiplications quickly dominates. In contrast, the
deterministic subspace selection approach allows to control the
condition number (and with it the number of matrix vector
multiplications) at a desired level without the need to tune any
parameters in theory as well as on the test instances. On these
instances, for low precision requirements (large barrier parameter)
the selected subspace is negligible small. For high precision
requirements (small barrier parameter) the subspace grows to the
active model subspace. If the bundle size is close to this active
dimension, the work in forming the preconditioner may be comparable to
forming the Schur complement for the direct solver. Still, for large
scale instances the deterministically preconditioned iterative
approach seems to be preferable.

Conceivably it is possible to profit in ConicBundle from the
advantages of the deterministic iterative and the direct solver by
switching dynamically between both. The current experiments relied on
a predictor-corrector approach that was tuned for the direct
solver. In view of the properties of the iterative approach it may
well be worth to devise a different path following strategy for the
iterative approach, in particular for the initial phase of the
interior point method when the barrier parameter is still
comparatively large and the work invested in forming the
preconditioner is still negl\upd{igible}.  Similar ideas should be
applicable to interior point solvers for solving convex quadratic
problems with low rank structure.

\textbf{Acknowledgments} I have profited a lot from discussions with
many colleagues, in part years back. In particular I have to thank
K.-C.~Toh as well as my colleagues O.~Ernst, R.~Herzog, A.~Pichler and
M.~Stoll in Chemnitz. Much of the preparatory restructuring of
ConicBundle was done during my sabbatical at the University of
Klagenfurt, thank you to F.~Rendl and A.~Wiegele for making this
possible.  The support of research grants 05M18OCA of the German
Federal Ministry of Education and Research and the DFG CRC 1410 is
gratefully acknowledged.


\newpage 
\begin{appendix}
  \section{Tables}
  For each box plot of figures \ref{fig:MC}--\ref{fig:MMBIS} for
  the three instances MC, BIS and MMBIS the following tables list the
  number of instances and the values of the parameters minimum, lower
  quartile ($Q_1$), median, upper quartile ($Q_3$), maximum. For each of the three
  instances an additional table gives the statistics on the Euclidean
  norm of the resulting residual of \eqref{eq:fullKKT} achieved by the
  respective solver for the KKT systems grouped by the usual
  value ranges of the barrier parameter.

  \subsection{Max-Cut (Instance MC, Figure \ref{fig:MC})}
    
\noindent Time per subproblem in seconds (338 instances):\\
\centerline{\small
\begin{tabular}{c|rrrrr}
solver & min & $Q_1$ & median & $Q_3$ & max  \\\hline
DS &   0.003832 &    34.3212 &    44.6572 &    55.9219 &    95.1265  \\
IT &   0.020491 &    12.6342 &    26.1729 &    79.1747 &     209.58  \\
RP &   0.030548 &    22.0347 &    34.3771 &     92.207 &    235.653  \\
DP &   0.020898 &    13.0997 &    30.6566 &    58.1994 &     110.33  \\
\end{tabular}}

\bigskip

\noindent Time per subproblem in seconds vs.\ ranges of bundle sizes:\\
\centerline{\small\begin{tabular}{lr|c|rrrrr}
\upd{bundle size} & \# & solver & min & $Q_1$ & median & $Q_3$ & max  \\\hline
 $[6,46]$ & 5 & 
DS &   0.003832 &   0.006275 &  0.0088455 &   0.021121 &   0.045992  \\
 & & IT &   0.020491 &   0.021558 &   0.031059 &   0.059577 &   0.095407  \\
 & & RP &   0.030548 &  0.0314005 &   0.043048 &   0.093556 &   0.146144  \\
 & & DP &   0.020898 &  0.0230755 &  0.0309575 &  0.0592245 &   0.094453  \\
\hline
 $[56,497]$ & 5 & 
DS &   0.073058 &   0.329229 &   0.657365 &     1.0668 &    4.51608  \\
 & & IT &   0.173868 &   0.374044 &   0.597913 &   0.800166 &    1.83209  \\
 & & RP &   0.273307 &   0.637968 &    1.02498 &     1.4059 &    3.19528  \\
 & & DP &   0.177225 &   0.382114 &    0.59731 &    0.80289 &     1.6773  \\
\hline
 $[596,1486]$ & 133 & 
DS &    6.62451 &    30.4929 &    33.6252 &    36.9311 &    43.6407  \\
 & & IT &    1.86562 &    12.7032 &    25.3882 &    65.6882 &    136.935  \\
 & & RP &    3.69818 &    20.2136 &    33.3548 &    77.6126 &    155.415  \\
 & & DP &    1.96853 &    13.9917 &    30.0676 &     49.426 &    90.8324  \\
\hline
 $[1541,2279]$ & 195 & 
DS &    41.2984 &    46.7501 &    53.9483 &    61.8286 &    95.1265  \\
 & & IT &    9.84382 &    13.6284 &    32.0618 &    90.4042 &     209.58  \\
 & & RP &      18.53 &    23.4963 &    42.6607 &    106.814 &    235.653  \\
 & & DP &     9.9354 &    13.9977 &     36.494 &    76.5581 &     110.33  \\
\hline
\end{tabular}}

\bigskip

Time per subproblem in seconds vs.\ last barrier parameter $\mu$:\\ 
\centerline{\small\begin{tabular}{lr|c|rrrrr}
$\mu$-range & \# & solver & min & $Q_1$ & median & $Q_3$ & max  \\\hline
 \upd{$[2.3e-01,3.6e+01]$} & 85 & 
DS &   0.003832 &    23.4027 &    49.1143 &    61.5018 &    95.1265  \\
 & & IT &   0.020491 &     7.2764 &    10.8231 &    11.9651 &    16.4467  \\
 & & RP &   0.030548 &    13.4739 &    19.5347 &    22.0031 &    27.3839  \\
 & & DP &   0.020898 &     7.2377 &    10.8421 &    11.9242 &    18.6599  \\
\hline
 \upd{$[1.6e-02,2.3e-01]$} & 85 & 
DS &    24.2787 &    33.6785 &    41.3086 &    48.8582 &     83.315  \\
 & & IT &    10.8004 &    13.7717 &    15.5724 &    17.2952 &    33.9445  \\
 & & RP &    17.3893 &    22.0349 &    24.1102 &    27.3977 &    44.4086  \\
 & & DP &    12.0901 &    15.1287 &    17.2843 &    18.8158 &    37.7594  \\
\hline
 \upd{$[7.7e-03,1.6e-02]$} & 85 & 
DS &    26.8515 &    35.9349 &    44.6111 &    54.1812 &    73.3521  \\
 & & IT &    23.8603 &     28.455 &    33.4883 &    86.0521 &    103.067  \\
 & & RP &    30.9349 &    37.0011 &    43.6655 &    100.814 &    120.387  \\
 & & DP &    28.6909 &    32.9195 &    37.3972 &    74.0808 &    87.0083  \\
\hline
 \upd{$[3.1e-04,7.7e-03]$} & 83 & 
DS &    31.4011 &    37.4751 &    45.8135 &    57.0963 &    75.7966  \\
 & & IT &    64.8852 &    75.3411 &    87.3138 &    154.968 &     209.58  \\
 & & RP &    77.5913 &    87.3499 &    99.7462 &    175.749 &    235.653  \\
 & & DP &    49.3655 &    53.8412 &    62.3306 &    96.8023 &     110.33  \\
\hline
\end{tabular}}

\bigskip
\newpage

Time per KKT system in seconds grouped by value ranges of the barrier
parameter $\mu$:\\
\centerline{\small\begin{tabular}{lr|c|rrrrr}
$\mu$-range & \# & solver & min & $Q_1$ & median & $Q_3$ & max  \\\hline
 $[4.5e+02,9.9e+02]$ & 676 & 
IT &   0.001092 &   0.463337 &   0.578703 &   0.917693 &    1.34934  \\
 & & RP &   0.002696 &    1.54367 &    1.80541 &     2.1684 &    3.22664  \\
 & & DP &    0.00252 &   0.515255 &   0.630482 &   0.916957 &    1.37408  \\
\hline
 \upd{$[1.4e+00,5.1e+01]$} & 1490 & 
IT &   0.005506 &     1.2121 &    1.83286 &    2.31178 &    6.49419  \\
 & & RP &   0.007486 &    2.36829 &    3.09237 &    3.69645 &    7.72418  \\
 & & DP &   0.005681 &    1.22031 &    1.82556 &    2.29723 &    7.32683  \\
\hline
 \upd{$[1.6e-02,3.1e-01]$} & 510 & 
IT &    3.20468 &    5.69856 &    8.79662 &    21.8884 &    56.4563  \\
 & & RP &    4.19829 &    6.85746 &    10.4541 &    23.0927 &    60.4059  \\
 & & DP &    4.28737 &    7.03695 &    10.7758 &    18.0144 &    23.1913  \\
\hline
 \upd{$[3.1e-04,9.3e-03]$} & 156 & 
IT &    17.2603 &      22.94 &    37.6166 &     52.734 &    74.6325  \\
 & & RP &    19.0854 &    25.2808 &    40.9692 &      55.18 &    82.3144  \\
 & & DP &    16.0158 &    18.3264 &    19.1491 &    19.8728 &    22.2577  \\
\hline
\end{tabular}}

\bigskip

Number of matrix vector multiplications per KKT system grouped by
value ranges of $\mu$:\\
\centerline{\small\begin{tabular}{lr|c|rrrrr}
$\mu$-range & \# & solver & min & $Q_1$ & median & $Q_3$ & max  \\\hline
 $[4.5e+02,9.9e+02]$ & 676 & 
IT &          4 &          6 &          6 &         10 &         12  \\
 & & RP &          4 &          6 &          6 &         10 &         12  \\
 & & DP &          4 &          6 &          6 &         10 &         12  \\
\hline
 \upd{$[1.4e+00,5.1e+01]$} & 1490 & 
IT &          9 &         17 &         22 &         23 &         64  \\
 & & RP &          9 &         17 &         22 &         23 &         64  \\
 & & DP &          9 &         17 &         22 &         23 &         55  \\
\hline
\upd{$[1.6e-02,3.1e-01]$} & 510 & 
IT &         54 &         77 &         84 &        252 &        612  \\
 & & RP &         54 &         77 &         84 &        253 &        627  \\
 & & DP &         45 &         49 &         52 &         54 &         63  \\
\hline
  \upd{$[3.1e-04,9.3e-03]$} & 156 & 
IT &        240 &        248 &        550 &        574 &        706  \\
 & & RP &        249 &        253 &        574 &        592 &        739  \\
 & & DP &         31 &         37 &         38 &         42 &         54  \\
\hline
\end{tabular}}

\bigskip

Condition number estimate of the KKT systems grouped by value ranges
of $\mu$:\\
\centerline{\small\begin{tabular}{lr|c|rrrrr}
$\mu$-range & \# & solver & min & $Q_1$ & median & $Q_3$ & max  \\\hline
 $[4.5e+02,9.9e+02]$ & 676 & 
IT &      1.003 &      1.009 &      1.009 &       1.01 &      1.065  \\
 & & RP &      1.003 &      1.009 &      1.009 &       1.01 &      1.065  \\
 & & DP &      1.003 &      1.009 &      1.009 &       1.01 &      1.065  \\
\hline
 \upd{$[1.4e+00,5.1e+01]$} & 1490 & 
IT &      1.013 &      1.473 &      1.566 &      2.393 &      17.06  \\
 & & RP &      1.013 &      1.473 &      1.566 &      2.393 &      16.98  \\
 & & DP &      1.013 &      1.473 &      1.566 &      2.393 &      11.38  \\
\hline
\upd{$[1.6e-02,3.1e-01]$} & 510 & 
IT &      12.03 &      26.72 &      29.72 &      253.7 &       2349  \\
 & & RP &      11.87 &      26.44 &      29.33 &      250.7 &       2289  \\
 & & DP &      10.62 &      11.11 &      11.46 &      11.75 &      11.96  \\
\hline
 \upd{$[3.1e-04,9.3e-03]$} & 156 & 
IT &      523.5 &      536.4 &       2057 &       2404 &       4711  \\
 & & RP &      507.1 &      527.3 &       2017 &       2371 &       4599  \\
 & & DP &      10.17 &      10.43 &      10.56 &      10.66 &      10.93  \\
\hline
\end{tabular}}

\bigskip

Number of preconditioning columns per KKT system grouped by value
ranges of $\mu$:\\
\centerline{\small\begin{tabular}{lr|c|rrrrr}
$\mu$-range & \# & solver & min & $Q_1$ & median & $Q_3$ & max  \\\hline
 $[4.5e+02,9.9e+02]$ & 676 & 
RP &          0 &          0 &          0 &          0 &          0  \\
 & & DP &          0 &          0 &          0 &          0 &          0  \\
\hline
 \upd{$[1.4e+00,5.1e+01]$} & 1490 & 
RP &          0 &          0 &          0 &          0 &          6  \\
 & & DP &          0 &          0 &          0 &          0 &         12  \\
\hline
 \upd{$[1.6e-02,3.1e-01]$} & 510 & 
RP &          6 &          6 &          6 &          6 &          8  \\
 & & DP &          5 &        171 &        212 &       1012 &       1167  \\
\hline
 \upd{$[3.1e-04,9.3e-03]$} & 156 & 
RP &          7 &          8 &         10 &         10 &         13  \\
 & & DP &       1010 &       1113 &       1156 &       1175 &       1177  \\
\hline
\end{tabular}}

\bigskip
\newpage

Euclidean norm of the residual of \eqref{eq:fullKKT} per KKT system
grouped by value ranges of $\mu$:\\
\centerline{\small\begin{tabular}{lr|c|rrrrr}
$\mu$-range & \# & solver & min & $Q_1$ & median & $Q_3$ & max  \\\hline
 $[4.5e+02,9.9e+02]$ & 676 & 
IT &    6.4e-10 &    1.4e-09 &    2.6e-09 &    5.2e-09 &      1e-06  \\
 & & RP &    6.4e-10 &    1.4e-09 &    2.6e-09 &    5.2e-09 &      1e-06  \\
 & & DP &    6.4e-10 &    1.4e-09 &    2.6e-09 &    5.2e-09 &      1e-06  \\
\hline
 \upd{$[1.4e+00,5.1e+01]$} & 1490 & 
IT &    3.6e-10 &    6.7e-09 &    5.2e-08 &    3.9e-07 &      1e-06  \\
 & & RP &    3.6e-10 &    6.4e-09 &    5.2e-08 &    3.9e-07 &    1.1e-06  \\
 & & DP &    3.6e-10 &    6.7e-09 &    5.2e-08 &    3.9e-07 &      1e-06  \\
\hline
 \upd{$[1.6e-02,3.1e-01]$} & 510 & 
IT &    8.6e-09 &      7e-07 &    8.2e-07 &    9.2e-07 &      1e-06  \\
 & & RP &    9.7e-09 &    7.3e-07 &    9.2e-07 &    1.1e-06 &    2.4e-06  \\
 & & DP &    5.2e-09 &    5.4e-07 &    6.8e-07 &    7.9e-07 &    1.2e-06  \\
\hline
 \upd{$[3.1e-04,9.3e-03]$} & 156 & 
IT &    9.7e-09 &    9.6e-07 &    9.8e-07 &    9.9e-07 &      1e-06  \\
 & & RP &    1.9e-08 &    1.4e-06 &    2.2e-06 &    4.2e-06 &    6.2e-06  \\
 & & DP &    8.6e-09 &    5.8e-07 &    7.1e-07 &    8.1e-07 &      1e-06  \\
\hline
\end{tabular}}

\newpage
\subsection{Minimum Bisection (Instance BIS, Figure \ref{fig:BIS})}

\noindent Time per subproblem in seconds (195 instances):\\
\centerline{\small\begin{tabular}{c|rrrrr}
solver & min & $Q_1$ & median & $Q_3$ & max  \\\hline
DS &    0.00923 &    14.1723 &    91.7467 &    359.175 &    15841.2  \\
IT &   0.044218 &    42.3879 &    340.967 &    3329.25 &    24016.1  \\
RP &   0.055209 &    36.3358 &    238.703 &     1079.8 &    4866.25  \\
DP &   0.049162 &    15.9154 &    89.9136 &    277.508 &    890.086  \\
\end{tabular}}

\bigskip

\noindent Time per subproblem in seconds vs.\ ranges of bundle sizes:\\
\centerline{\small\begin{tabular}{lr|c|rrrrr}
\upd{bundle size} & \# & solver & min & $Q_1$ & median & $Q_3$ & max  \\\hline
 $[6,28]$ & 2 & 
DS &    0.00923 & -- &    0.00923 & -- &   0.091545  \\
 & & IT &   0.044218 & -- &   0.044218 & -- &   0.153935  \\
 & & RP &   0.055209 & -- &   0.055209 & -- &   0.222727  \\
 & & DP &   0.049162 & -- &   0.049162 & -- &   0.177788  \\
\hline
 $[67,497]$ & 128 & 
DS &   0.594552 &    8.69331 &    22.7682 &    132.809 &    428.513  \\
 & & IT &   0.960953 &    43.8529 &    466.297 &    2076.04 &      11867  \\
 & & RP &    1.23022 &    33.9205 &     219.14 &    835.573 &    2856.49  \\
 & & DP &   0.941817 &    11.8932 &    53.3373 &    183.946 &    560.093  \\
\hline
 $[529,1379]$ & 40 & 
DS &    34.3749 &    81.2278 &    517.036 &    720.019 &     1505.5  \\
 & & IT &    9.82356 &    16.6676 &    4819.87 &    15092.7 &    24016.1  \\
 & & RP &    12.5897 &     20.357 &    1586.65 &    3277.46 &    4866.25  \\
 & & DP &    10.0625 &    16.2665 &    340.397 &     612.94 &    890.086  \\
\hline
 $[1597,5887]$ & 25 & 
DS &    302.944 &    783.079 &    1403.96 &    3335.16 &    15841.2  \\
 & & IT &    29.3698 &    85.8307 &    211.467 &    294.115 &    932.896  \\
 & & RP &    38.2829 &    108.971 &    248.069 &    341.147 &    955.341  \\
 & & DP &    31.1555 &    87.6845 &    204.633 &    295.527 &    789.044  \\
\hline
\end{tabular}}

\bigskip

Time per subproblem in seconds vs.\ last barrier parameter $\mu$:\\ 
\centerline{\small\begin{tabular}{lr|c|rrrrr}
$\mu$-range & \# & solver & min & $Q_1$ & median & $Q_3$ & max  \\\hline
 \upd{$[1.9e+00,1.0e+03]$} & 49 & 
DS &    0.00923 &    8.97908 &    69.8681 &    955.047 &    15841.2  \\
 & & IT &   0.044218 &    4.08472 &     14.083 &    154.701 &    932.896  \\
 & & RP &   0.055209 &    4.93747 &    18.2564 &     178.33 &    955.341  \\
 & & DP &   0.049162 &    4.21385 &    14.8836 &    143.015 &    789.044  \\
\hline
 \upd{$[2.8e-03,1.4e+00]$} & 49 & 
DS &   0.594552 &    9.95126 &    52.3298 &    128.666 &    6981.47  \\
 & & IT &   0.960953 &    46.9621 &    218.946 &    937.902 &    1927.16  \\
 & & RP &    1.23022 &    38.2691 &    183.586 &    508.265 &    933.538  \\
 & & DP &   0.941817 &    15.3137 &    66.5916 &    139.005 &    734.869  \\
\hline
\upd{$[9.4e-04,2.8e-03]$} & 49 & 
DS &    2.12836 &    18.2489 &    73.8103 &    267.252 &    705.689  \\
 & & IT &    10.3255 &     198.11 &    1744.97 &    4383.54 &    7789.69  \\
 & & RP &    9.92439 &    111.622 &    670.068 &    1522.76 &    2315.44  \\
 & & DP &    4.13901 &    32.1878 &    138.761 &    282.768 &     459.79  \\
\hline
\upd{$[1.1e-04,9.4e-04]$} & 48 & 
DS &    3.77535 &    20.2483 &    189.134 &    599.782 &     1505.5  \\
 & & IT &    21.8778 &    575.065 &    6101.11 &    13836.1 &    24016.1  \\
 & & RP &    19.1305 &     219.14 &    1524.83 &    2942.66 &    4866.25  \\
 & & DP &    5.98724 &    52.4419 &    299.974 &    580.588 &    890.086  \\
\hline
\end{tabular}}

\bigskip
\newpage

Time per KKT system in seconds grouped by value ranges of the barrier
parameter $\mu$:\\
\centerline{\small\begin{tabular}{lr|c|rrrrr}
$\mu$-range & \# & solver & min & $Q_1$ & median & $Q_3$ & max  \\\hline
 $[1.1e+02,2.6e+04]$ & 824 & 
IT &   0.001823 &   0.328123 &   0.786603 &    1.66599 &     8.7638  \\
 & & RP &   0.007159 &   0.410637 &   0.980222 &      2.023 &    11.5388  \\
 & & DP &   0.006695 &   0.275799 &   0.669313 &    1.45762 &    7.73375  \\
\hline
 \upd{$[1.0e+00,9.9e+01]$} & 2262 & 
IT &   0.110818 &     2.7715 &    7.29372 &    14.4584 &    58.5776  \\
 & & RP &   0.145002 &    2.60334 &    6.57376 &    13.2324 &    54.0312  \\
 & & DP &    0.10742 &    1.46962 &     3.4827 &    7.40111 &     37.485  \\
\hline
\upd{$[1.0e-02,1.0e+00]$} & 1991 & 
IT &   0.817783 &    21.4256 &    46.2355 &    85.6325 &    327.564  \\
 & & RP &   0.744242 &    11.4276 &    23.3757 &    37.8191 &    100.053  \\
 & & DP &   0.365208 &    2.70049 &    5.41858 &    8.25427 &     18.679  \\
\hline
\upd{$[1.1e-04,1.0e-02]$} & 1103 & 
IT &    3.30796 &    173.717 &    295.614 &    470.923 &    1316.81  \\
 & & RP &    2.48871 &    44.4168 &     62.784 &    86.7278 &    168.322  \\
 & & DP &    0.40273 &    5.37352 &    6.99157 &    10.5791 &    14.8519  \\
\hline
\end{tabular}}

\bigskip

Number of matrix vector multiplications per KKT system grouped by
value ranges of $\mu$:\\
\centerline{\small\begin{tabular}{lr|c|rrrrr}
$\mu$-range & \# & solver & min & $Q_1$ & median & $Q_3$ & max  \\\hline
 $[1.1e+02,2.6e+04]$ & 824 & 
IT &         10 &         23 &         30 &         45 &         99  \\
 & & RP &          8 &         19 &         26 &         38 &         77  \\
 & & DP &          7 &         19 &         26 &         38 &         69  \\
\hline
 \upd{$[1.0e+00,9.9e+01]$}  & 2262 & 
IT &         23 &        107 &        204 &      386.5 &        820  \\
 & & RP &         23 &         97 &        166 &        277 &        555  \\
 & & DP &         23 &         47 &         62 &         80 &        133  \\
\hline
\upd{$[1.0e-02,1.0e+00]$} & 1991 & 
IT &        141 &     1024.5 &       1632 &     2719.5 &       5960  \\
 & & RP &        130 &      548.5 &      722.5 &      910.5 &       1774  \\
 & & DP &         19 &         31 &         33 &         36 &         70  \\
\hline
 \upd{$[1.1e-04,1.0e-02]$} & 1103 & 
IT &        750 &     5607.5 &     7552.5 &      10475 &      19571  \\
 & & RP &        399 &       1249 &     1419.5 &       1657 &       2620  \\
 & & DP &         17 &         31 &         33 &         35 &         41  \\
\hline
\end{tabular}}

\bigskip

Condition number estimate of the KKT systems grouped by value ranges
of $\mu$:\\
\centerline{\small\begin{tabular}{lr|c|rrrrr}
$\mu$-range & \# & solver & min & $Q_1$ & median & $Q_3$ & max  \\\hline
 $[1.1e+02,2.6e+04]$ & 824 & 
IT &       2.15 &      3.773 &       5.36 &       12.1 &      997.8  \\
 & & RP &      1.008 &      2.465 &       3.84 &      7.696 &      67.38  \\
 & & DP &      1.122 &      2.469 &      3.834 &      7.442 &      25.34  \\
\hline
 \upd{$[1.0e+00,9.9e+01]$}  & 2262 & 
IT &      3.681 &      93.22 &      277.5 &       1041 &  1.234e+04  \\
 & & RP &      2.035 &      59.19 &        160 &        477 &       2544  \\
 & & DP &      2.035 &      7.361 &      15.99 &      27.31 &      407.7  \\
\hline
\upd{$[1.0e-02,1.0e+00]$} & 1991 & 
IT &      289.5 &  1.409e+04 &  4.986e+04 &  1.869e+05 &  2.075e+06  \\
 & & RP &      246.8 &       3107 &       6609 &  1.405e+04 &  9.741e+04  \\
 & & DP &      1.717 &       2.29 &      2.632 &      3.128 &      27.66  \\
\hline
 \upd{$[1.1e-04,1.0e-02]$} & 1103 & 
IT &  4.501e+04 &  1.324e+06 &  2.869e+06 &   6.53e+06 &  4.576e+07  \\
 & & RP &       7554 &  3.375e+04 &  5.161e+04 &  8.013e+04 &  4.099e+05  \\
 & & DP &      1.703 &      2.048 &      2.185 &      2.522 &      12.55  \\
\hline
\end{tabular}}

\bigskip

Number of preconditioning columns per KKT system grouped by value
ranges of $\mu$:\\
\centerline{\small\begin{tabular}{lr|c|rrrrr}
$\mu$-range & \# & solver & min & $Q_1$ & median & $Q_3$ & max  \\\hline
 $[1.1e+02,2.6e+04]$ & 824 & 
RP &          0 &          0 &          0 &          0 &          6  \\
 & & DP &          0 &          0 &          0 &          0 &          9  \\
\hline
 \upd{$[1.0e+00,9.9e+01]$}  & 2262 & 
RP &          0 &          2 &          6 &          8 &         24  \\
 & & DP &          0 &          9 &         43 &         82 &        152  \\
\hline
\upd{$[1.0e-02,1.0e+00]$} & 1991 & 
RP &          4 &         20 &         25 &         29 &         38  \\
 & & DP &         36 &        113 &        141 &        154 &        169  \\
\hline
\upd{$[1.1e-04,1.0e-02]$} & 1103 & 
RP &         13 &         31 &         33 &         36 &         44  \\
 & & DP &         68 &      147.5 &        159 &        168 &        177  \\
\hline
\end{tabular}}

\bigskip
\newpage

Euclidean norm of the residual of \eqref{eq:fullKKT} per KKT system
grouped by value ranges of $\mu$:\\
\centerline{\small\begin{tabular}{lr|c|rrrrr}
$\mu$-range & \# & solver & min & $Q_1$ & median & $Q_3$ & max  \\\hline
 $[1.1e+02,2.6e+04]$ & 824 & 
IT &    9.8e-13 &    9.1e-09 &    3.2e-07 &    5.6e-07 &      1e-06  \\
 & & RP &    1.4e-12 &    6.9e-09 &    1.9e-07 &    3.5e-07 &    1.7e-06  \\
 & & DP &    1.3e-12 &      7e-09 &    1.9e-07 &    3.5e-07 &    1.1e-06  \\
\hline
 \upd{$[1.0e+00,9.9e+01]$}  & 2262 & 
IT &    2.9e-10 &    4.7e-09 &    6.7e-09 &      5e-07 &      1e-06  \\
 & & RP &    3.1e-10 &      4e-09 &    5.7e-09 &    3.3e-07 &    3.2e-06  \\
 & & DP &    3.8e-10 &      5e-09 &    7.4e-09 &    3.5e-07 &    2.8e-06  \\
\hline
\upd{$[1.0e-02,1.0e+00]$} & 1991 & 
IT &    1.2e-09 &      5e-09 &      9e-09 &    1.9e-08 &      1e-06  \\
 & & RP &    1.6e-09 &    6.4e-09 &    1.2e-08 &    2.4e-08 &    2.2e-05  \\
 & & DP &    6.6e-10 &    5.2e-09 &    9.4e-09 &      2e-08 &    7.2e-06  \\
\hline
\upd{$[1.1e-04,1.0e-02]$} & 1103 & 
IT &    4.3e-09 &    4.9e-08 &    1.1e-07 &    3.3e-07 &    2.9e-06  \\
 & & RP &    3.5e-09 &    3.3e-08 &    5.7e-08 &    9.5e-08 &    7.3e-05  \\
 & & DP &    2.2e-09 &    3.4e-08 &      6e-08 &      1e-07 &    5.3e-06  \\
\hline
\end{tabular}}

\newpage
\subsection{Min-Max Bisection (Instance MMBIS, Figure \ref{fig:MMBIS}}

\noindent Time per subproblem in seconds (35 instances):\\
\centerline{\small\begin{tabular}{c|rrrrr}
solver & min & $Q_1$ & median & $Q_3$ & max  \\\hline
DS &   0.001669 &   0.958311 &       4942 &    35185.7 &     140651  \\
IT &    0.06925 &     1.7246 &    319.124 &    2069.75 &    13368.7  \\
RP &   0.059093 &    1.02453 &    215.038 &    1412.83 &    7763.78  \\
DP &   0.046371 &   0.864647 &    154.628 &     590.86 &    2083.68  \\
\end{tabular}}

\bigskip

\noindent Time per subproblem in seconds vs.\ ranges of bundle sizes:\\
\centerline{\small\begin{tabular}{lr|c|rrrrr}
\upd{bundle size} & \# & solver & min & $Q_1$ & median & $Q_3$ & max  \\\hline
 $[0,37]$ & 7 & 
DS &   0.001669 &   0.008018 &  0.0415445 &   0.091635 &   0.240095  \\
 & & IT &    0.06925 &   0.098375 &   0.229849 &   0.332737 &   0.752143  \\
 & & RP &   0.059093 &   0.069004 &   0.144198 &   0.224022 &   0.449136  \\
 & & DP &   0.046371 &   0.056267 &   0.121948 &   0.185404 &   0.391171  \\
\hline
 $[56,352]$ & 4 & 
DS &    0.46651 &    0.46651 &    1.45011 &    3.83018 &     17.636  \\
 & & IT &    1.08647 &    1.08647 &    2.36273 &    4.81555 &    16.5215  \\
 & & RP &   0.668257 &   0.668257 &     1.3808 &    3.04164 &    9.26137  \\
 & & DP &   0.567064 &   0.567064 &    1.16223 &    2.53117 &    6.96995  \\
\hline
 $[1327,1379]$ & 2 & 
DS &    300.616 & -- &    300.616 & -- &    969.539  \\
 & & IT &     56.453 & -- &     56.453 & -- &      416.7  \\
 & & RP &    35.8027 & -- &    35.8027 & -- &    257.055  \\
 & & DP &    28.0559 & -- &    28.0559 & -- &     167.41  \\
\hline
 $[2702,16291]$ & 22 & 
DS &    1807.41 &    6318.51 &    22730.5 &    53047.4 &     140651  \\
 & & IT &    223.837 &    319.124 &    641.539 &    4509.64 &    13368.7  \\
 & & RP &    163.342 &    215.038 &    450.345 &    2558.71 &    7763.78  \\
 & & DP &    119.399 &    157.057 &    315.234 &    814.558 &    2083.68  \\
\hline
\end{tabular}}

\bigskip

Time per subproblem in seconds vs.\ last barrier parameter $\mu$:\\ 
\centerline{\small\begin{tabular}{lr|c|rrrrr}
$\mu$-range & \# & solver & min & $Q_1$ & median & $Q_3$ & max  \\\hline
 \upd{$[3.6e-02,1.4e+00]$} & 9 & 
DS &   0.014367 &  0.0464135 &   0.186595 &    2.14835 &    1807.41  \\
 & & IT &     0.1275 &   0.285552 &   0.553999 &    2.95101 &      416.7  \\
 & & RP &   0.078915 &   0.183008 &   0.342684 &    1.85495 &    257.055  \\
 & & DP &   0.066163 &   0.154237 &   0.292596 &    1.54912 &     167.41  \\
\hline
 \upd{$[7.0e-03,3.5e-02]$} & 9 & 
DS &    1.45011 &    159.126 &    3825.74 &       4942 &    16156.8  \\
 & & IT &    2.36273 &    36.4873 &    230.856 &    257.261 &    539.789  \\
 & & RP &     1.3808 &    22.5321 &    163.699 &    181.259 &    397.781  \\
 & & DP &    1.16223 &    17.5129 &     121.09 &    135.255 &     246.92  \\
\hline
 \upd{$[1.0e-03,5.9e-03]$} & 9 & 
DS &   0.040436 &    9770.54 &    18267.6 &    30691.4 &    47352.7  \\
 & & IT &   0.198213 &    388.296 &    598.446 &    2093.91 &    5072.23  \\
 & & RP &   0.134192 &    265.913 &    414.119 &    1302.96 &     2838.2  \\
 & & DP &    0.11221 &    189.641 &    276.743 &    511.322 &    890.336  \\
\hline
 \upd{$[2.4e-04,9.6e-04]$} & 8 & 
DS &   0.001669 &    49409.5 &    64909.4 &    83523.5 &     140651  \\
 & & IT &    0.06925 &    1582.35 &    5497.04 &    8664.78 &    13368.7  \\
 & & RP &   0.059093 &    1175.55 &    3060.07 &    4787.77 &    7763.78  \\
 & & DP &   0.046371 &    614.286 &    954.903 &    1273.72 &    2083.68  \\
\hline
\end{tabular}}

\bigskip

\newpage
Time per KKT system in seconds grouped by value ranges of the barrier
parameter $\mu$:\\
\centerline{\small\begin{tabular}{lr|c|rrrrr}
$\mu$-range & \# & solver & min & $Q_1$ & median & $Q_3$ & max  \\\hline
 \upd{$[1.0e+02,1.0e+03]$} & 295 & 
IT &   0.002025 &   0.259585 &    11.9083 &    22.8061 &    54.4588  \\
 & & RP &   0.001737 &   0.159694 &    7.95722 &    14.0881 &    32.6083  \\
 & & DP &   0.001638 &   0.119556 &     5.4995 &    10.3867 &    23.2037  \\
\hline
 \upd{$[1.0e+00,9.9e+01]$} & 146 & 
IT &    0.00647 &     1.8041 &    17.2027 &    28.9793 &    72.3436  \\
 & & RP &   0.005452 &    1.23622 &     11.287 &    20.3489 &    36.1462  \\
 & & DP &   0.003676 &    0.98141 &    9.03463 &    15.5483 &    28.2709  \\
\hline
\upd{$[1.2e-02,1.0e+00]$} & 159 & 
IT &   0.003261 &     8.2878 &    16.9553 &    27.6698 &    101.194  \\
 & & RP &   0.003848 &    5.98822 &    13.8806 &    24.4045 &    79.1858  \\
 & & DP &   0.002763 &    4.14727 &    8.67571 &    18.2895 &    52.1301  \\
\hline
 \upd{$[2.4e-04,8.6e-03]$} & 182 & 
IT &   0.002879 &    205.544 &     310.59 &    439.276 &    833.569  \\
 & & RP &   0.003367 &    123.886 &    170.845 &     253.51 &    417.102  \\
 & & DP &    0.00254 &    33.7196 &    42.3111 &    58.5254 &    79.5814  \\
\hline
\end{tabular}}

\bigskip

Number of matrix vector multiplications per KKT system grouped by
value ranges of $\mu$:\\
\centerline{\small\begin{tabular}{lr|c|rrrrr}
$\mu$-range & \# & solver & min & $Q_1$ & median & $Q_3$ & max  \\\hline
 \upd{$[1.0e+02,1.0e+03]$} & 295 & 
IT &         13 &         57 &         69 &         85 &        146  \\
 & & RP &          7 &         21 &         29 &         31 &         45  \\
 & & DP &          7 &         21 &         28 &       32.5 &         45  \\
\hline
\upd{$[1.0e+00,9.9e+01]$} & 146 & 
IT &         28 &         67 &         76 &         85 &        133  \\
 & & RP &         17 &         33 &         34 &         37 &         51  \\
 & & DP &         18 &         32 &         33 &         35 &         42  \\
\hline
\upd{$[1.2e-02,1.0e+00]$} & 159 & 
IT &         23 &         63 &        107 &        181 &        321  \\
 & & RP &         17 &         40 &         81 &        117 &        175  \\
 & & DP &         17 &         31 &         56 &       74.5 &         84  \\
\hline
\upd{$[2.4e-04,8.6e-03]$} & 182 & 
IT &         18 &        478 &        761 &      987.5 &       1590  \\
 & & RP &         13 &      321.5 &        420 &        503 &        726  \\
 & & DP &         14 &       69.5 &         89 &       92.5 &        110  \\
\hline
\end{tabular}}

\bigskip

Condition number estimate of the KKT systems grouped by value ranges
of $\mu$:\\
\centerline{\small\begin{tabular}{lr|c|rrrrr}
$\mu$-range & \# & solver & min & $Q_1$ & median & $Q_3$ & max  \\\hline
 \upd{$[1.0e+02,1.0e+03]$} & 295 & 
IT &      12.93 &       6950 &       7879 &       9417 &  1.496e+06  \\
 & & RP &      1.001 &      2.869 &      4.467 &      14.82 &      81.57  \\
 & & DP &      1.001 &      1.549 &      3.146 &      5.245 &      13.87  \\
\hline
\upd{$[1.0e+00,9.9e+01]$} & 146 & 
IT &      3.102 &      213.5 &      714.8 &       6152 &    4.1e+05  \\
 & & RP &      1.434 &      5.384 &      8.388 &      38.71 &      271.3  \\
 & & DP &      2.706 &      5.471 &       6.97 &      8.305 &      11.16  \\
\hline
\upd{$[1.2e-02,1.0e+00]$} & 159 & 
IT &      2.177 &      404.9 &       1566 &       2739 &  2.984e+04  \\
 & & RP &      1.639 &      43.97 &      79.24 &      195.1 &      730.4  \\
 & & DP &      2.031 &        6.9 &      8.502 &      10.03 &      12.24  \\
\hline
 \upd{$[2.4e-04,8.6e-03]$} & 182 & 
IT &      1.179 &  3.755e+04 &  5.674e+04 &  9.099e+04 &  2.553e+05  \\
 & & RP &      1.131 &      580.6 &      737.5 &       1039 &       2200  \\
 & & DP &      1.169 &      11.08 &      11.61 &      12.17 &      20.08  \\
\hline
\end{tabular}}

\bigskip

Number of preconditioning columns per KKT system grouped by value
ranges of $\mu$:\\
\centerline{\small\begin{tabular}{lr|c|rrrrr}
$\mu$-range & \# & solver & min & $Q_1$ & median & $Q_3$ & max  \\\hline
 \upd{$[1.0e+02,1.0e+03]$} & 295 & 
RP &          0 &          0 &          0 &          0 &          4  \\
 & & DP &          0 &          0 &          0 &          1 &          3  \\
\hline
 \upd{$[1.0e+00,9.9e+01]$} & 146 & 
RP &          0 &          0 &          0 &          1 &          3  \\
 & & DP &          0 &          0 &          0 &          1 &          2  \\
\hline
\upd{$[1.2e-02,1.0e+00]$} & 159 & 
RP &          0 &          0 &          2 &          3 &          5  \\
 & & DP &          0 &          2 &          3 &         15 &         84  \\
\hline
\upd{$[2.4e-04,8.6e-03]$} & 182 & 
RP &          0 &          3 &          4 &          6 &         11  \\
 & & DP &          0 &      111.5 &        127 &      148.5 &        175  \\
\hline
\end{tabular}}

\bigskip
\newpage

Euclidean norm of the residual of \eqref{eq:fullKKT} per KKT system
grouped by value ranges of $\mu$:\\
\centerline{\small\begin{tabular}{lr|c|rrrrr}
$\mu$-range & \# & solver & min & $Q_1$ & median & $Q_3$ & max  \\\hline
 \upd{$[1.0e+02,1.0e+03]$} & 295 & 
IT &      1e-09 &    7.8e-08 &    2.5e-07 &      5e-07 &    9.8e-07  \\
 & & RP &    2.7e-11 &    3.6e-09 &      1e-08 &      5e-08 &    5.3e-07  \\
 & & DP &    3.4e-11 &    3.8e-09 &    1.2e-08 &    3.8e-08 &    3.9e-07  \\
\hline
\upd{$[1.0e+00,9.9e+01]$} & 146 & 
IT &    3.3e-09 &    3.4e-08 &    3.3e-07 &    6.5e-07 &      1e-06  \\
 & & RP &    1.2e-12 &    2.3e-08 &    9.1e-08 &    1.8e-07 &    1.1e-06  \\
 & & DP &    2.2e-10 &    1.7e-08 &    7.6e-08 &    2.5e-07 &    1.1e-06  \\
\hline
\upd{$[1.2e-02,1.0e+00]$} & 159 & 
IT &    6.5e-10 &    1.3e-08 &    1.4e-07 &    4.9e-07 &    9.8e-07  \\
 & & RP &    5.4e-10 &    8.5e-09 &    8.8e-08 &    1.7e-07 &    7.7e-07  \\
 & & DP &      2e-10 &    1.1e-08 &    1.2e-07 &    2.9e-07 &    8.9e-07  \\
\hline
 \upd{$[2.4e-04,8.6e-03]$} & 182 & 
IT &    3.5e-09 &    3.4e-08 &    7.8e-08 &    6.5e-07 &    1.3e-06  \\
 & & RP &    6.5e-10 &    1.2e-08 &    2.3e-08 &    8.5e-08 &    7.7e-07  \\
 & & DP &    7.8e-10 &    1.5e-08 &    3.1e-08 &      1e-07 &      9e-07  \\
\hline
\end{tabular}}

\newpage

\upd{\subsection{Performance on Random Max-Cut Instances}}\label{A:MCperf}

\upd{The graphs were randomly generated  with \texttt{rudy} \cite{rudy}.
An instance denoted by $MC(n,d,s)$  refers to a random graph on
$n$ nodes with edge density $d$ and seed $s$ for the random number
generator. It is generated by the call \mbox{\texttt{rudy
    -rnd\_graph n d s}}.}

\upd{Together with each instance a comparison value is listed. This value
was obtained by running ConicBundle \cite{ConicBundle2021} on
computers having a QUAD-Core-Processor
INTEL-Core-I7-4770 (4$\times$ 3400MHz, 8 MB Cache) with 32 GB RAM
 and operating system Ubuntu 18.04 in sequential mode (but mostly two instances at
 the same time) for termination precision $10^{-6}$
for each method and then taking among all these the minimum value
produced. For this value $\gamma$ we then determined within the log
files for each method and relative precision level
$\varepsilon\in\{10^{-3},10^{-4},10^{-5},10^{-6}\}$ the
first descent step that yields an objective value of at most
$\gamma+\varepsilon(1+ |\gamma|)$ and list the user time 
needed to reach this step (rounded to seconds) together with the total number of steps (null and descent
steps) and the total number of KKT systems
solved up to this point. The (non rounded) numbers are employed for the
performance profiles in figures \ref{fig:timeperf1} and \ref{fig:timeperf2}.}

\upd{{\small
\begin{tabularx}{\textwidth}{@{\,}l@{\ }l|r@{\;}r@{\;}r|r@{\;}r@{\;}r|r@{\;}r@{\;}r|r@{\;}r@{\;}r@{\,}}
 instance & meth. & \multicolumn{3}{c|}{$10^{-3}$}& \multicolumn{3}{c|}{$10^{-4}$}& \multicolumn{3}{c|}{$10^{-5}$}& \multicolumn{3}{c}{$10^{-6}$}\\
comparison value &  & time & steps & KKT  & time & steps & KKT  & time & steps & KKT  & time & steps & KKT \\\hline
 $MC(10000,0.1,1)$  & DS  & 693 & 46 & 354 & 1239 & 68 & 580 & 1662 & 117 & 1288 & 2330 & 213 & 2933\\
$3.953149412e+04$ & DSp  & 436 & 45 & 296 & 935 & 74 & 535 & 1508 & 142 & 1291 & 2118 & 232 & 2410\\
 & IT  & 156 & 45 & 359 & 363 & 68 & 611 & 1433 & 136 & 1608 & 5501 & 247 & 3516\\
 & ITp  & 117 & 60 & 379 & 351 & 90 & 628 & 1523 & 155 & 1384 & 3356 & 225 & 2215\\
 & RP  & 136 & 48 & 369 & 313 & 77 & 692 & 1237 & 146 & 1698 & 5353 & 246 & 3472\\
 & RPp  & 82 & 34 & 218 & 292 & 63 & 450 & 1703 & 128 & 1131 & 4673 & 219 & 2218\\
\hline
\end{tabularx}}
}

\end{appendix}

\end{document}

%% file: inputs.tex
\usepackage{color}
\usepackage{amsmath,amsfonts,amstext,amssymb,mathtools,bbm,mathdots}
\usepackage{graphicx}
\usepackage{pifont}
\usepackage{lscape}
\usepackage{ltablex}
\usepackage{varioref}
\usepackage[shortlabels]{enumitem}
\usepackage{hyperref}

\newtheorem{theorem}{Theorem}

\newtheorem{algorithm}[theorem]{Algorithm}

\newtheorem{corollary}[theorem]{Corollary}
\newtheorem{lemma}[theorem]{Lemma}

\newenvironment{proof}{{\noindent\bf Proof.}$\;$}{\hfill\rule{1.6ex}{1.6ex}\par\medskip}

{}

\newlist{compactenum}{enumerate}{4}
\setlist[compactenum]{label=\arabic*.,itemsep=0pt,topsep=0pt,parsep=0pt,leftmargin=*}
\newlist{compactitem}{itemize}{4}
\setlist[compactitem]{label=\textbullet,itemsep=0pt,topsep=0pt,parsep=2pt
  plus1pt minus1pt}

\newcommand{\calB}{{\mathcal{B}}}

\newcommand{\calE}{{\mathcal{E}}}
\newcommand{\calF}{{\mathcal{F}}}

\newcommand{\calI}{{\mathcal{I}}}
\newcommand{\calJ}{{\mathcal{J}}}

\newcommand{\calN}{{\mathcal{N}}}

\newcommand{\calQ}{{\mathcal{Q}}}

\newcommand{\calS}{{\mathcal{S}}}

\newcommand{\frakF}{{\mathfrak{F}}}

\newcommand{\frakH}{{\mathfrak{H}}}

\newcommand{\frakX}{{\mathfrak{X}}}

\newcommand{\frakp}{{\mathfrak{p}}}

\newcommand{\E}{\mathbb{E}}
\newcommand{\F}{\mathbb{F}}

\newcommand{\N}{\mathbb{N}}

\newcommand{\R}{\mathbb{R}}

\newcommand{\clR}{\overline{\mathbb{R}}}
\newcommand{\Sym}{\mathbb{S}}

\newcommand{\Prob}{\mathbb{P}}

\newcommand{\one}{\mathbbm{1}}

\newcommand{\gset}{C}

\newcommand{\hati}{{\ensuremath{\hat\imath}}}
\newcommand{\hatj}{{\ensuremath{\hat\jmath}}}

\newcommand{\socn}{{\ensuremath{\calQ^n}}}
\newcommand{\soc}[1]{{\ensuremath{\calQ^{#1}}}}

\newcommand{\ie}{\mbox{\emph{i.\,e}.,}\ }

\newcommand{\eg}{\emph{e.\,g.}\ }

\newcommand{\Var}{\mathop{\mathrm{Var}}}

\newcommand{\argmin}{\mathop {\rm argmin}}

\newcommand{\tr}{\mathop{\rm tr}}
\newcommand{\rank}{\mathop{\rm rank}}
\newcommand{\dom}{\mathop{\mathrm{dom}}}

\newcommand{\Diag}{\mathop{\rm Diag}}

\newcommand{\svec}{\mathop{\rm svec}}
\newcommand{\skron}{\otimes\unkern_{\unkern s}}

\newcommand{\conv}{\mathop{\rm conv}}

\newcommand{\lmin}{\lambda_{\min}}
\newcommand{\lmax}{\lambda_{\max}}

\newcommand{\mat}[2]{\ensuremath\left[\begin{array}{@{\,}#1@{\,}}#2\end{array}\right]}
\newcommand{\lmat}{\left[}
\newcommand{\rmat}{\right]}
\newcommand{\lvec}{\left(}
\newcommand{\rvec}{\right)}
\newcommand{\vectwo}[2]{\lvec{{#1}\atop{#2}}\rvec}

\newcommand{\ip}[2]{\left\langle {#1},{#2}\right\rangle}

\newcommand{\maxset}[1]{\max_{\begin{array}{c}#1\end{array}}}
\newcommand{\minset}[1]{\min_{\begin{array}{c}#1\end{array}}}

\newenvironment{sbmatrix}{\ensuremath\left[\begin{smallmatrix}}{\end{smallmatrix}\right]}